\documentclass[11pt]{amsart}

\usepackage{amssymb, amscd}
\usepackage{epsfig, mathtools, tensor}
\usepackage[vmargin=1in, hmargin=1.25in]{geometry}
\usepackage[font=small,format=plain,labelfont=bf,up,textfont=it,up]{caption}
\usepackage{microtype}
\usepackage[shortlabels]{enumitem}
\setlist[itemize]{nosep}
\usepackage[backref=page, bookmarks, bookmarksdepth=2, colorlinks=true, linkcolor=blue, citecolor=blue, urlcolor=blue]{hyperref}
\usepackage{etoolbox}
\usepackage{tikz-cd}
\usepackage{tabularray}
\apptocmd{\thebibliography}{\raggedright}{}{}
\usepackage[only,llparenthesis, rrparenthesis, llbracket,rrbracket]{stmaryrd}
\usepackage{refcount}

\mathtoolsset{showonlyrefs}

\newcommand\Figure[1]{\centerline{\psfig{file=#1,scale=1}}}

\makeatletter
\patchcmd{\@maketitle}{\global\topskip42\p@\relax}
  {\global\topskip42\p@\relax \vspace*{-38pt}}
  {}{}
\makeatother

\setenumerate[0]{label=(\alph*)}

\renewcommand*{\backref}[1]{}
\renewcommand*{\backrefalt}[4]{%
    \ifcase #1 (Not cited.)%
    \or        (Cited on page~#2.)%
    \else      (Cited on pages~#2.)%
    \fi}

\DeclareSymbolFont{bbold}{U}{bbold}{m}{n}
\DeclareSymbolFontAlphabet{\mathbbold}{bbold}

\DeclareSymbolFont{extraup}{U}{zavm}{m}{n}
\DeclareMathSymbol{\varheartsuit}{\mathalpha}{extraup}{86}

\newcommand{\arxiv}[1]{\href{http://arxiv.org/abs/#1}{{\tt arXiv:#1}}}

\numberwithin{equation}{section}

\theoremstyle{plain}
\newtheorem{theorem}{Theorem}[section]
\newtheorem{maintheorem}{Theorem}
\newtheorem{maincorollary}[maintheorem]{Corollary}
\newtheorem{proposition}[theorem]{Proposition}
\newtheorem{lemma}[theorem]{Lemma}

\newtheorem{corollary}[theorem]{Corollary}

\newtheorem*{unnumberedclaim}{Claim}

\newtheorem{stepx}{Step}
\newenvironment{step}[1]
 {\renewcommand\thestepx{#1}\stepx}
 {\endstepx}

\newtheorem{casex}{Case}
\newenvironment{case}[1]
 {\renewcommand\thecasex{#1}\casex}
 {\endcasex}

\newtheorem{claimx}{Claim}
\newenvironment{claim}[1]
 {\renewcommand\theclaimx{#1}\claimx}
 {\endclaimx}

\providecommand{\previous}{}



\providecommand{\previous}{}



\providecommand{\previous}{}



\theoremstyle{definition}
\newtheorem{asm}[theorem]{Assumption}

\newtheorem{defn}[theorem]{Definition}

\newtheorem{notn}[theorem]{Notation}

\newtheorem{covn}[theorem]{Convention}
\newenvironment{convention}[1][]{\begin{covn}[#1]\pushQED{\qed}}{\popQED \end{covn}}

\theoremstyle{remark}
\newtheorem{rmk}[theorem]{Remark}
\newenvironment{remark}[1][]{\begin{rmk}[#1] \pushQED{\qed}}{\popQED \end{rmk}}
\newtheorem{eg}[theorem]{Example}
\newenvironment{example}[1][]{\begin{eg}[#1] \pushQED{\qed}}{\popQED \end{eg}}

\theoremstyle{plain}

\DeclareMathOperator{\Hom}{Hom}

\DeclareMathOperator{\coker}{coker}

\DeclareMathOperator{\Lie}{Lie}

\DeclareMathOperator{\GL}{GL}

\DeclareMathOperator{\SL}{SL}

\DeclareMathOperator{\Sp}{Sp}

\newcommand\Z{\ensuremath{\mathbb{Z}}}
\newcommand\Q{\ensuremath{\mathbb{Q}}}

\DeclareMathOperator{\HH}{H}

\DeclareMathOperator{\Res}{Res}

\DeclareMathOperator{\Sym}{Sym}

\newcommand\Span[1]{\ensuremath{\langle #1 \rangle}}
\newcommand\SpanSet[2]{\ensuremath{\langle \text{#1 $|$ #2} \rangle}}
\newcommand\Set[2]{\ensuremath{\left\{\text{#1 $|$ #2}\right\}}}

\newcommand\cK{\ensuremath{\mathcal{K}}}
\newcommand\cL{\ensuremath{\mathcal{L}}}

\newcommand\cP{\ensuremath{\mathcal{P}}}

\newcommand\cT{\ensuremath{\mathcal{T}}}

\newcommand\cV{\ensuremath{\mathcal{V}}}

\newcommand\fC{\ensuremath{\mathfrak{C}}}

\newcommand\fc{\ensuremath{\mathfrak{c}}}

\newcommand\bT{\ensuremath{\mathbf{T}}}
\newcommand\bU{\ensuremath{\mathbf{U}}}
\newcommand\bV{\ensuremath{\mathbf{V}}}
\newcommand\bW{\ensuremath{\mathbf{W}}}

\newcommand\bk{\ensuremath{\mathbf{k}}}

\newcommand\tJ{\ensuremath{\widetilde{J}}}

\newcommand\tkappa{\ensuremath{\widetilde{\kappa}}}

\newcommand\okappa{\ensuremath{\overline{\kappa}}}
\newcommand\otheta{\ensuremath{\overline{\theta}}}

\newcommand\oalpha{\ensuremath{\overline{\alpha}}}
\newcommand\obeta{\ensuremath{\overline{\beta}}}

\newcommand\hS{\ensuremath{\widehat{S}}}

\newcommand\htau{\ensuremath{\widehat{\tau}}}

\newcommand*{\Cdot}[1][1.25]{%
  \mathpalette{\CdotAux{#1}}\cdot%
}
\newdimen\CdotAxis
\newcommand*{\CdotAux}[3]{%
  {%
    \settoheight\CdotAxis{$#2\vcenter{}$}%
    \sbox0{%
      \raisebox\CdotAxis{%
        \scalebox{#1}{%
          \raisebox{-\CdotAxis}{%
            $\mathsurround=0pt #2#3$%
          }%
        }%
      }%
    }%
    \dp0=0pt %
    \sbox2{$#2\bullet$}%
    \ifdim\ht2<\ht0 %
      \ht0=\ht2 %
    \fi
    \sbox2{$\mathsurround=0pt #2#3$}%
    \hbox to \wd2{\hss\usebox{0}\hss}%
  }%
}

\newcommand\Mod{\ensuremath{\operatorname{Mod}}}
\newcommand\Torelli{\ensuremath{\operatorname{\mathcal{I}}}}
\newcommand\diag{\ensuremath{\operatorname{diag}}}
\newcommand\FLie{\ensuremath{\operatorname{FLie}}}
\newcommand\Brack[1]{\ensuremath{\llbracket #1 \rrbracket}}
\DeclareMathOperator{\SI}{\text{\tt SI}}
\DeclareMathOperator{\cp}{cup}

\title{Calculating the second rational cohomology group of the Torelli group}

\author{Andrew Putman}
\address{Dept of Mathematics; University of Notre Dame; 255 Hurley Hall; Notre Dame, IN 46556; USA}
\email{andyp@nd.edu}

\thanks{AP was supported by NSF grant DMS-2305183.}

\begin{document}
    
\newpage
        
\begin{abstract}
Minahan and the author recently proved results that allow the calculation
of the second rational cohomology group of the Torelli group.  This builds on two
key ingredients: Hain's calculation of the image of the cup product pairing on
the first cohomology group,
and Kupers--Randal-Williams's calculation of the maximal algebraic subrepresentation
of the second cohomology group.  This paper gives an
exposition of both of these results, including prerequisite material about
the Johnson homomorphism.
\end{abstract} 

\maketitle
\thispagestyle{empty}

\tableofcontents

\section{Introduction}
\label{section:introduction}

Let $\Sigma_{g,p}^b$ be an oriented genus $g$ surface with $p$ punctures
and $b$ boundary components.  Assume that $p+b \leq 1$.  We 
omit $p$ or $b$ if they vanish.  The
mapping class group $\Mod_{g,p}^b$ is
the group
of isotopy classes of orientation-preserving diffeomorphisms of $\Sigma_{g,p}^b$ that fix each puncture and boundary
component pointwise.  
The group $\Mod_{g,p}^b$ acts on $\HH_1(\Sigma_{g,p}^b)$ 
and fixes the algebraic intersection form.  Since $p + b \leq 1$, the algebraic
intersection form is symplectic and this action gives
a surjection $\Mod_{g,p}^b \rightarrow \Sp_{2g}(\Z)$ whose kernel $\Torelli_{g,p}^b$ is the Torelli group:
\[\begin{tikzcd}
1 \arrow{r} & \Torelli_{g,p}^b \arrow{r} & \Mod_{g,p}^b \arrow{r} & \Sp_{2g}(\Z) \arrow{r} & 1.
\end{tikzcd}\]
In the early 1980's, Johnson \cite{Johnson3} calculated $\HH^1(\Torelli_{g,p}^b)$ for $g \gg 0$.
Building on Minahan's PhD thesis \cite{MinahanThesis},
Minahan and the author \cite{MinahanPutmanH2Torelli} proved results that allowed them to
calculate $\HH^2(\Torelli_{g,p}^b;\Q)$ for $g \gg 0$.  We were able to do the following 
(see Corollary \ref{maincorollary:h2torelli} below):
\begin{itemize}
\item[(i)] Completely determine $\HH^2(\Torelli_{g,p}^b;\Q)$ as a vector space over $\Q$.  In fact,
we even describe it as a representation of $\Sp_{2g}(\Z)$, which acts in a way we will describe below.
\item[(ii)] Prove that all $\HH^2(\Torelli_{g,p}^b;\Q)$ is in the image of the cup product
pairing on $\HH^1(\Torelli_{g,p}^b;\Q)$.  
\end{itemize}
This calculation depends on previous
work of Hain \cite{HainInfinitesimal} and Kupers--Randal-Williams (\cite{KupersRandalWilliams};
see also its sequel \cite{RandalWilliamsII}) that we will describe below.
The goal of this paper is to explain how to do (i) and (ii) above given the
results of \cite{MinahanPutmanH2Torelli}, and in particular to give an exposition of this work of Hain and
Kupers--Randal-Williams

\subsection{Finite-dimensionality and algebraicity}

Recall our standing assumption that $p+b \leq 1$.  The group $\Mod_{g,p}^b$ acts on
$\Torelli_{g,b}^p$ via conjugation.  Since inner automorphisms act trivially on group cohomology, we
get an induced action of
$\Mod_{g,p}^b / \Torelli_{g,p}^b = \Sp_{2g}(\Z)$
on each cohomology group $\HH^k(\Torelli_{g,p}^b;\Q)$.  In other words, $\HH^k(\Torelli_{g,p}^b;\Q)$ is
a representation of the arithmetic group $\Sp_{2g}(\Z)$.

Set $H = \HH_1(\Sigma_{g,p}^b;\Q)$.  Johnson \cite{Johnson3} proved that for $g \geq 3$ we have
\[\HH^1(\Torelli_{g,p}^b;\Q) \cong \begin{cases}
\wedge^3 H & \text{if $p=1$ or $b=1$},\\
(\wedge^3 H)/H & \text{if $p=b=0$}.
\end{cases}\]
Here $H$ is embedded in $\wedge^3 H$ via the map $h \mapsto h \wedge \omega$, where
$\omega \in \wedge^2 H$ represents the algebraic intersection pairing.\footnote{The algebraic intersection pairing
$\omega$ is an alternating bilinear form on $H$, i.e., an element of the dual space
$(\wedge^2 H)^{\ast} = \Hom(\wedge^2 H,\Q)$.  The form $\omega$ identifies $H$ with $H^{\ast}$, and thus
$(\wedge^2 H)^{\ast}$ with $\wedge^2 H$.  If $\{a_1,b_1,\ldots,a_g,b_g\}$ is a symplectic basis for $H$,
then the resulting $\omega \in \wedge^2 H$ is $\omega = a_1 \wedge b_1 + \cdots + a_g \wedge b_g$.}  The
above isomorphism is an isomorphism of representations of $\Sp_{2g}(\Z)$.  From it, we deduce
that the following two things hold for $g \geq 3$:
\begin{itemize}
\item $\HH^1(\Torelli_{g,p}^b;\Q)$ is finite-dimensional; and
\item $\HH^1(\Torelli_{g,p}^b;\Q)$ is an algebraic representation of $\Sp_{2g}(\Z)$, i.e., the action
of $\Sp_{2g}(\Z)$ on it extends to a representation of the algebraic group $\Sp_{2g}(\Q)$.
\end{itemize}
The main result of \cite{MinahanPutmanH2Torelli} extends this to the second rational cohomology group:

\begin{maintheorem}[{Minahan--Putman, \cite[Theorem B]{MinahanPutmanH2Torelli}}]
\label{maintheorem:h2finite}
Fix some $p,b \geq 0$ with $p+b \leq 1$.  Then $\HH^2(\Torelli_{g,p}^b;\Q)$ is finite-dimensional
for $g \geq 5$ and an algebraic representation of $\Sp_{2g}(\Z)$ for $g \geq 6$.
\end{maintheorem}

As we said above, we will assume the truth of Theorem \ref{maintheorem:h2finite} and
describe how to calculate $\HH^2(\Torelli_{g,p}^b;\Q)$ as in
(i) and (ii) above.

\begin{remark}
Dualizing, Theorem \ref{maintheorem:h2finite} implies that
$\HH_2(\Torelli_{g,p}^b;\Q)$ is finite-dimensional for $g \geq 5$.
It is still not known if $\HH_2(\Torelli_{g,p}^b)$ is finitely generated
or if $\Torelli_{g,p}^b$ is finitely presented for $g \gg 0$.
\end{remark}

\subsection{Cup product pairing}

The cup product pairing is an alternating pairing
\[\fc\colon \wedge^2 \HH^1(\Torelli_{g,p}^b;\Q) \rightarrow \HH^2(\Torelli_{g,p}^b;\Q).\]
Hain \cite{HainInfinitesimal} calculated the image of this cup product pairing.  To state
his result, we must introduce some representation theoretic language.
The irreducible algebraic representations of $\Sp_{2g}(\Z)$ are indexed by partitions $\sigma$ with at most
$g$ parts (see \cite[\S 17]{FultonHarris}; we will describe this in more detail later).  
Let $\bV_{\sigma}(g)$ be the representation corresponding to $\sigma$, so $\bV_1(g) = H$ and $\bV_{1^3}(g) = (\wedge^3 H)/H$.\
The domain of the cup product pairing is
\begin{align*}
\wedge^2 \HH^1(\Torelli_g;\Q)   &\cong \wedge^2((\wedge^3 H)/H),\\
\wedge^2 \HH^1(\Torelli_g^1;\Q) &\cong \wedge^2(\wedge^3 H),\\
\wedge^2 \HH^1(\Torelli_{g,1};\Q) &\cong \wedge^2(\wedge^3 H).
\end{align*}
These decompose as\footnote{This calculation can easily be done using the program ``LiE''; see \cite{LieProgram}.}
\[\wedge^2 ((\wedge^3 H)/H) \cong \wedge^2 \bV_{1^3}(g) \cong \bV_0(g) \oplus \bV_{1^2}(g) \oplus \bV_{2^2}(g) \oplus \bV_{1^4}(g) \oplus \bV_{2^2,1^2}(g) \oplus \bV_{1^6}(g)\]
and
\[\wedge^2 (\wedge^3 H)
\cong \bV_0(g)^{\oplus 2} \oplus \bV_{1^2}(g)^{\oplus 3} \oplus \bV_{2^2}(g) \oplus \bV_{2,1^2}(g) \oplus \bV_{1^4}(g)^{\oplus 2} \oplus \bV_{2^2,1^2}(g) \oplus \bV_{1^6}(g).\]
Since the cup product pairing is equivariant, its image is isomorphic to a subrepresentation of this.
Hain calculated it as follows:\footnote{Hain also worked out the image for $g \in \{3,4,5\}$.  We restrict to $g \geq 6$ since
the work of Minahan--Putman \cite{MinahanPutmanH2Torelli} only works in this range and it allows a more uniform proof.}

\begin{maintheorem}[{Hain, \cite{HainInfinitesimal}}]
\label{maintheorem:hain}
Let $g \geq 6$.  Fix some $p,b \geq 0$ with $p+b \leq 1$.  Then the image of 
the cup product pairing $\fc\colon \wedge^2 \HH^1(\Torelli_{g,p}^b;\Q) \rightarrow \HH^2(\Torelli_{g,p}^b;\Q)$
is isomorphic to the following representation of $\Sp_{2g}(\Z)$:
\begin{itemize}
\item For $\Torelli_g$, the representation $\bV_{1^2}(g) \oplus \bV_{1^4}(g) \oplus \bV_{2^2,1^2}(g) \oplus \bV_{1^6}(g)$.
\item For $\Torelli_g^1$, the representation $\bV_{1^2}(g)^{\oplus 2} \oplus \bV_{2,1^2}(g) \oplus \bV_{1^4}(g)^{\oplus 2} \oplus \bV_{2^2,1^2}(g) \oplus \bV_{1^6}(g)$.
\item For $\Torelli_{g,1}$, the representation $\bV_0(g) \oplus \bV_{1^2}(g)^{\oplus 2} \oplus \bV_{2,1^2}(g) \oplus \bV_{1^4}(g)^{\oplus 2}\oplus \bV_{2^2,1^2}(g) \oplus \bV_{1^6}(g)$.
\end{itemize}
\end{maintheorem}

\begin{remark}
The discrete group $\Sp_{2g}(\Z)$ is Zariski dense in the algebraic group
$\Sp_{2g}(\Q)$.  If $W$ is an algebraic representation of $\Sp_{2g}(\Z)$, it
follows that representation-theoretic properties of $W$ as a representation
of $\Sp_{2g}(\Z)$ are the same as the corresponding properties of $W$ as
a representation of $\Sp_{2g}(\Q)$.  For instance, if $W$ is an irreducible
algebraic representation of $\Sp_{2g}(\Q)$ then it is also an irreducible
representation of $\Sp_{2g}(\Z)$, and similarly for the decomposition of
$W$ into a direct sum of irreducible representations.
\end{remark}

\subsection{Maximal algebraic subrepresentation}

For $g \geq 12$, Kupers--Randal-Williams (\cite{KupersRandalWilliams}; see also its sequel \cite{RandalWilliamsII}) 
proved that the fragment
of $\HH^2(\Torelli_{g,p}^b;\Q)$ identified by Theorem \ref{maintheorem:hain} is the maximal algebraic
subrepresentation of $\HH^2(\Torelli_{g,p}^b;\Q)$.  They proved this
without knowing that $\HH^2(\Torelli_{g,p}^b;\Q)$ is finite-dimensional, much less
an algebraic representation of $\Sp_{2g}(\Z)$.  Roughly speaking, they prove that
a larger algebraic subrepresentation would contradict known calculations of $\HH^2(\Mod_{g,p}^b;\bV)$
with $\bV$ an algebraic representation of $\Sp_{2g}(\Z)$.  

\begin{remark}
The work of Kupers--Randal-Williams does not depend on Hain's calculation
of the image of the cup product pairing.\footnote{Our proof of their
theorem will use Hain's calculation, which we will use to
give a lower bound on the size of $\HH^2(\Torelli_{g,p}^b;\Q)$.}
What they actually prove is that the maximal algebraic
subrepresentation of $\HH^2(\Torelli_{g,p}^b;\Q)$ is isomorphic to
the representation identified by Theorem \ref{maintheorem:hain}.  It then
follows that it must be isomorphic to the image of the cup product pairing.  If
there was cohomology that was not in the image of the cup product pairing, their
techniques would detect it.
\end{remark} 

Using the fact that we now know that $\HH^2(\Torelli_{g,p}^b;\Q)$ is a finite-dimensional
algebraic representation of $\Sp_{2g}(\Z)$ for $g \geq 6$ (Theorem \ref{maintheorem:h2finite}),
we will give a simplified proof of Kupers--Randal-Williams's theorem.  In fact, we will use
not only Theorem \ref{maintheorem:h2finite}, but also some of the details from its proof.
We will also use some work of Patzt \cite{PatztRepStability} on representation stability to
control certain stabilization maps.  We will prove the following, whose statement combines
Kupers--Randal-Williams's theorem with Theorem \ref{maintheorem:h2finite} of Minahan
and the author:

\begin{maintheorem}
\label{maintheorem:h2calc}
Let $g \geq 12$.  Fix some $p,b \geq 0$ with $p+b \leq 1$.  Then the image of
the cup product pairing $\fc\colon \wedge^2 \HH^1(\Torelli_{g,p}^b;\Q) \rightarrow \HH^2(\Torelli_{g,p}^b;\Q)$
spans $\HH^2(\Torelli_{g,p}^b;\Q)$.
\end{maintheorem}

\begin{remark}
The main way our proof is simpler than that of \cite{KupersRandalWilliams} is that
we are able to avoid much of their complicated combinatorics, 
as well the need to carefully develop properties of the twisted Morita--Mumford classes.  The
paper \cite{KupersRandalWilliams} proves many other results beyond the ones we discuss,
and for these their machinery seems essential.
\end{remark}

\begin{remark}
The published version of \cite{KupersRandalWilliams} claims to work for $g \geq 6$, but contains
an error; see \cite{KupersRandalWilliamsErratum}.  Their new range $g \geq 12$ is exactly what
is needed to apply recent work of Miller--Patzt--Petersen--Randal-Williams \cite{MPPRW} to deduce
uniform twisted homological stability for $\HH^2(\Mod_g^1;\bV)$.  Here ``uniform'' means
that the stable range does not depend on the algebraic representation $\bV$ of $\Sp_{2g}(\Z)$.
We will also use the results of \cite{MPPRW}.
\end{remark}

\subsection{Combining the results}

Combining Theorems \ref{maintheorem:hain} and \ref{maintheorem:h2calc}, we deduce
the following corollary:

\begin{maincorollary}[{Minahan--Putman, \cite{MinahanPutmanH2Torelli}}]
\label{maincorollary:h2torelli}
Let $g \geq 12$.  Fix some $p,b \geq 0$ with $p+b \leq 1$.  Then the
$\Sp_{2g}(\Z)$-representation $\HH^2(\Torelli_{g,p}^b;\Q)$ is as follows:
\begin{itemize}
\item For $\Torelli_g$, the representation $\bV_{1^2}(g) \oplus \bV_{1^4}(g) \oplus \bV_{2^2,1^2}(g) \oplus \bV_{1^6}(g)$.
\item For $\Torelli_g^1$, the representation $\bV_{1^2}(g)^{\oplus 2} \oplus \bV_{2,1^2}(g) \oplus \bV_{1^4}(g)^{\oplus 2} \oplus \bV_{2^2,1^2}(g) \oplus \bV_{1^6}(g)$.
\item For $\Torelli_{g,1}$, the representation $\bV_0(g) \oplus \bV_{1^2}(g)^{\oplus 2} \oplus \bV_{2,1^2}(g) \oplus \bV_{1^4}(g)^{\oplus 2}\oplus \bV_{2^2,1^2}(g) \oplus \bV_{1^6}(g)$.
\end{itemize}
In all three cases, $\HH^2(\Torelli_{g,p}^b;\Q)$ is spanned by the image of the cup
product pairing on $\HH^1(\Torelli_{g,p}^b;\Q)$.
\end{maincorollary}

\subsection{Remark on representation theory}

The proofs of Theorems \ref{maintheorem:hain} and \ref{maintheorem:h2calc} rely on
fairly intricate representation-theoretic calculations.  Courses on representation
theory often focus on theoretical aspects of the subject and do not prepare
students to make explicit calculations.  A secondary goal of this paper is
to give many examples of concrete representation-theoretic calculations.

\subsection{Johnson homomorphisms}

The Johnson homomorphisms are a key tool in the proof 
of Theorem \ref{maintheorem:hain}.  They were introduced
by Johnson \cite{JohnsonHomo, JohnsonSurvey} and further
developed by Morita \cite{MoritaAbelian} and many others.
We will give a fairly complete account (with proofs) of
the part of this theory needed for Theorem \ref{maintheorem:hain}.
In addition to making this paper more self-contained, this
will also give us the opportunity to give some easier
representation-theoretic calculations before moving
on to the more difficult calculations in Theorem \ref{maintheorem:hain}.

\subsection{Outline of paper}

This paper has two parts:
Part \ref{part:1} proves Theorem \ref{maintheorem:hain}, and
Part \ref{part:2} proves Theorem \ref{maintheorem:h2calc}.
Each part begins with an outline.

\subsection{Acknowledgments}

I wish to thank Richard Hain for several inspiring conversations about his work, Annie Holden for corrections and useful discussions about
computing cup products, Oscar Randal-Williams for corrections, and Gavril Farkas for encouraging me to write this survey.

\part{Image of cup product pairing}
\label{part:1}

In this first part of the paper, we calculate the image of the cup product pairing.
Most of this part is devoting to calculating this for $\Torelli_g^1$, and then at the end
we show how to deal with $\Torelli_{g,1}$ and $\Torelli_g$.
We start by discussing the representation theory of the algebraic
group $\SL_n$ in \S \ref{section:slrep}.  Using this, in \S \ref{section:freelie} we
discuss the free Lie algebra and its connection to the lower central series of a free group.

With these preliminaries out of the way, we can now discuss our main tool: the generalized
Johnson homomorphism.  We introduce it in \S \ref{section:johnsongeneral}.  We then
make a number of calculations with it in \S \ref{section:johnsoncalc} and prove a theorem
of Morita about its image in \S \ref{section:johnsonimage}.  
We then have two final sections of preliminary results in \S \ref{section:sprep} -- \S \ref{section:spprojection},
which discuss the representation theory of the algebraic group $\Sp_{2g}$.
Using this, we calculate the images of the the first and second Johnson homomorphisms
in \S \ref{section:firstjohnson} and \S \ref{section:secondjohnson}.

We then turn to the image of the cup product pairing.  In \S \ref{section:h2notrivial}
we prove this image has no trivial subrepresentations.  Next, in \S \ref{section:cupjohnson}, we prove
that this image can be calculated in terms of the first Johnson homomorphism.
We then perform this calculation in \S \ref{section:hainboundary}.  The second Johnson homomorphism
plays an important role here, where it is used to compute an upper bound for the image of
the cup product pairing.  Everything we have done up until now is for $\Torelli_g^1$, and we close
this part in \S \ref{section:hainpuncture} and \S \ref{section:hainclosed}
by showing how to deal with $\Torelli_{g,1}$ and $\Torelli_g$.

\section{Representation theory of \texorpdfstring{$\SL_n$}{SLn}}
\label{section:slrep}

We begin with some preliminaries about the representation theory of the algebraic group
$\SL_n$.  Fix a field $\bk$ of characteristic $0$.  We will regard $\SL_n$ as being defined
over $\bk$, and all representations of $\SL_n$ will be finite-dimensional and defined over $\bk$.  Everything
we discuss can be found in \cite{BorelLinear}, and in the slightly different language
of Lie algebra representations can also be found in \cite{FultonHarris}.

\subsection{Maximal split torus}

Let $\bT$ be the subgroup of $\SL_n$ consisting of diagonal matrices, so
$\bT(\bk) \cong (\bk^{\times})^{n-1}$.  This is what is called a split
algebraic torus, and is a maximal split torus in $\SL_n$.  For $t_1,\ldots,t_n \in \bk$
with $t_1 \cdots t_n = 1$, let $\diag(t_1,\ldots,t_n) \in \bT(\bk)$ be the corresponding
diagonal matrix.  A {\em character} of $\bT$
is a homomorphism $\chi\colon \bT \rightarrow \GL_1$.  For some $d_1,\ldots,d_n \in \Z$, such a
character can be written as
\[\chi(\diag(t_1,\ldots,t_n)) = t_1^{d_1} \cdots t_n^{d_n} \quad \text{for all $\diag(t_1,\ldots,t_n) \in \bT(\bk)$}.\]
Since $t_1 \cdots t_n = 1$, we can assume that $d_n = 1$, in which case this representation is unique.
Let $\chi(\bT)$ be the set of all characters of $\bT$.  This is an abelian group,
and by the above we have $\chi(\bT) \cong \Z^{n-1}$.  For $1 \leq i \leq n$, let
$E_i \in \chi(\bT)$ be character defined by
\[E_i(\diag(t_1,\ldots,t_n)) = t_i.\]
The characters $\{E_1,\ldots,E_n\}$ generate $\chi(\bT)$ and satisfy the single relation
$E_1 + \cdots + E_n = 0$.

\subsection{Weight decomposition}

Let $W$ be a representation of $\SL_n$.  For $\chi \in \chi(\bT)$, let
\[W_{\chi} = \Set{$w \in W$}{$D \cdot w = \chi(D) w$ for all $D \in \bT(\bk)$}.\]
If $W_{\chi} \neq 0$, then 
$\chi$ is a {\em weight} of $W$ and $W_{\chi}$ is its corresponding {\em weight space}.
A nonzero vector $w \in W_{\chi}$ is a {\em weight vector} with weight $\chi$.
We have a direct sum decomposition
\[W = \bigoplus_{\chi \in \chi(\bT)} W_{\chi}\]
called the {\em weight decomposition} of $W$.

\begin{example}
Let $\{e_1,\ldots,e_n\}$ be the standard basis for the standard representation $\bk^n$
of $\SL_n$.
The weight decomposition of $\bk^n$ is then
\[\bk^n = \bigoplus_{i=1}^n (\bk^n)_{E_i} \quad \text{with $(\bk^n)_{E_i} = \Span{e_i}$}.\]
Similarly, for $n \geq 2$ the weight decomposition of $\wedge^2 \bk^n$ is
\[\wedge^2 \bk^n = \bigoplus_{1 \leq i < j \leq n} (\wedge^2 \bk^n)_{E_i + E_j} \quad \text{with $(\wedge^2 \bk^n)_{E_i+E_j} = \Span{e_i \wedge e_j}$}\]
and the weight decomposition of $\Sym^2(\bk^n)$ is
\[\Sym^2(\bk^n) = \left(\bigoplus_{i=1}^n \Sym^2(\bk^n)_{2 E_i}\right) \oplus \left(\bigoplus_{1 \leq i < j \leq n} \Sym^2(\bk^n)_{E_i + E_j}\right)\]
with
\[\Sym^2(\bk^n)_{2 E_i} = \Span{e_i \Cdot e_i} \quad \text{and} \quad \Sym^2(\bk^n)_{E_i + E_j} = \Span{e_i \Cdot e_j}.\qedhere\]
\end{example}

\subsection{Highest weight vectors}

Let $\bU$ be the subgroup of $\SL_n$ consisting of strictly upper triangular matrices.
If $W$ is a representation of $\SL_n$, then a {\em highest weight vector} in $W$ with weight $\chi \in \chi(\bT)$
is a nonzero vector $w \in W$ such that
\begin{itemize}
\item $w$ is a weight vector with weight $\chi$, so $w \neq 0$ and $w \in W_{\chi}$; and
\item $w$ is fixed by all elements of $\bU(\bk)$.
\end{itemize}
A {\em highest weight} of $W$ is the weight of a highest weight vector in $W$.

The integer points $\bU(\Z)$ are Zariski dense in $\bU(\bk)$,
so to check that a weight vector $w \in W$ is a highest weight vector it
is enough to check that it is invariant under $\bU(\Z)$.  In 
fact, it is enough to check this on generators for this group, which are as
follows.  Let $\{e_1,\ldots,e_n\}$ be the standard basis for $\bk^n$.
\begin{itemize}
\item For $1 \leq i < j \leq n$, the matrix $E_{ij} \in \bU(\bk)$ that fixes
all elements of $\{e_1,\ldots,e_n\}$ except for $e_j$,
on which it does the following:
\[E_{ij}(e_j) =  e_j + e_i.\]
\end{itemize}
We remark that these are examples of elementary matrices.

\begin{example}
Let $\{e_1,\ldots,e_n\}$ be the standard basis for the standard representation $\bk^n$
of $\SL_n$.
The vector $e_1 \in \bk^n$ is a highest weight vector with weight $E_1$.  For $1 \leq k \leq n$, the
vector $e_1 \wedge \cdots \wedge e_k \in \wedge^k \bk^n$ is a highest weight vector with weight $E_1 + \cdots+E_k$.
Also, for $k \geq 1$ the vector $e_1 \Cdot \ldots \Cdot e_1 \in \Sym^k(\bk^n)$ is a highest weight vector with weight $k E_1$.
One can check that up to scaling these are the only highest weight vectors in these representations.
\end{example}

\subsection{Theorem of the highest weight}

The representation theory of $\SL_n$ is controlled by highest weight vectors.
Indeed, the theorem of the highest weight says that:
\begin{enumerate}
\item All representations of $\SL_n$ decompose as direct sums of irreducible representations.
\item Up to scaling, each irreducible representation $W$ of $\SL_n$ contains a unique highest weight vector.
\item If $W$ is an arbitrary representation of $\SL_n$ and $w \in W$ is a highest weight vector, then
the smallest subrepresentation containing $w$ is irreducible.
\item If $W$ and $W'$ are irreducible representations of $\SL_n$ with highest weight vectors $w \in W$ and $w' \in W'$
and if the weights of $w$ and $w'$ are the same, then $W$ and $W'$ are isomorphic representations.
\end{enumerate}

\begin{example}
It follows from the above that the following are all distinct irreducible
representations of $\SL_n$:
\begin{itemize}
\item $\wedge^k \bk^n$ for $1 \leq k \leq n$; and
\item $\Sym^k(\bk^n)$ for $k \geq 1$.\qedhere
\end{itemize}
\end{example}

\subsection{Dominant weights}

To complete the picture of the representation theory of $\SL_n$, we must identify the weights
that can appear as highest weights of irreducible representations of $\SL_n$.  A {\em dominant weight}
is a character 
\[\chi = k_1 E_1 + \cdots + k_n E_n \quad \text{with $k_1,\ldots,k_n \in \Z$}\]
such that the following hold:
\begin{itemize}
\item $k_1 \geq \cdots \geq k_n$; and
\item $k_n = 0$, so in particular $k_i \geq 0$ for all $1 \leq i \leq n$.  We remark that since
the $E_i$ satisfy the single relation
$E_1 + \cdots + E_n = 0$, any character $\chi$ can be written uniquely
as $\chi = k_1 E_1 + \cdots + k_n E_n$ with $k_n = 0$.
\end{itemize}
The dominant weights are exactly the weights of irreducible representations of $\SL_n$.
We will write $\bW_{\chi}(n)$ for the irreducible representation of $\SL_n$
with highest weight a given dominant weight $\chi$.

Recall that a {\em partition} of an integer $d$ of length $m \geq 0$ is
a tuple $\sigma = (k_1,\ldots,k_m)$ with $k_1 \geq \cdots \geq k_m \geq 1$
and $k_1 + \cdots k_m = d$.  The dominant weights for $\SL_n$ are in bijection
with partitions of integers with length at most $n-1$.  Given
such a partition $\sigma = (k_1,\ldots,k_m)$, the corresponding
dominant weight is
\[\chi = k_1 E_1 + \cdots + k_m E_m.\]
We will also write $\bW_{\sigma}(n)$ for $\bW_{\chi}(n)$.

\begin{convention}
We will denote multiplicities in partitions using superscripts.
For instance, if $\sigma = (5,4,4,4,1,1)$ then we will write
$\bW_{5,4^3,1^2}(n)$ for $\bW_{\sigma}(n)$.  
\end{convention}

Here are some examples of this notation:

\begin{example}
For $\SL_n$, we have the following:
\begin{itemize}
\item $\bW_{k}(n) = \Sym^k(\bk^n)$ for $k \geq 1$; and
\item $\bW_{1^k}(n) = \wedge^k \bk^n$ for $1 \leq k < n$.  
\end{itemize}
It is not the case that $\bW_{1^n}(n) = \wedge^n \bk^n$.  Indeed, according to our
conventions $1^n$ does not correspond to a dominant weight.  Instead
$\wedge^n \bk^n$ is isomorphic to the trivial representation, so $\wedge^n \bk^n = \bW_0(n)$.
\end{example}

\subsection{Stable decompositions}

Let $\sigma = (k_1,\ldots,k_m)$ be a partition of $d$ with at most $n-1$ parts.  We will call $d$ the
{\em degree} of the partition.  We will also say that $d$ is the degree
of the irreducible representation $\bW_{\sigma}(n)$.
Schur--Weyl duality implies that the irreducible representation $\bW_{\sigma}(n)$ appears in in
$(\bk^n)^{\otimes d}$, and moreover that all irreducible representations that appear
in $(\bk^n)^{\otimes d}$ have degree at most $d$.  

A classical observation is that for $n \geq d+1$, the decomposition of $(\bk^n)^{\otimes d}$ 
into irreducible factors is independent of the parameter $n$ in the following sense:
\begin{itemize}
\item If 
\[(\bk^n)^{\otimes d} = \bW_{\sigma_1}(n) \oplus \cdots \bW_{\sigma_k}(n)\] 
for partitions $\sigma_1,\ldots,\sigma_k$, then we also have
\[(\bk^{n+1})^{\otimes d} = \bW_{\sigma_1}(n+1) \oplus \cdots \bW_{\sigma_k}(n+1).\]
\end{itemize}
Here are some examples of this:\footnote{All these decompositions are calculated using the program ``LiE''; see \cite{LieProgram}.} 

\begin{example}
Consider $(\bk^n)^{\otimes 2}$.  This decomposes into irreducible representations in
the following ways:
\begin{itemize}
\item For $n=2$, as $(\bk^2)^{\otimes 2} = \bW_0(2) \oplus \bW_2(2)$.
\item For $n \geq 3$, as $(\bk^n)^{\otimes 3} = \bW_{1^2}(n) \oplus \bW_{2}(n)$.
\end{itemize}
For $n \geq 3$, this is the decomposition into antisymmetric and symmetric tensors
\[(\bk^n)^{\otimes 3} = \bW_{1^2}(n) \oplus \bW_2(n) = \left(\wedge^2 \bk^n\right) \oplus \Sym^2(\bk^n).\]
For $n=2$, this decomposition is degenerate since $\wedge^2 \bk^2$ is the trivial representation $\bW_0$.
\end{example}

\begin{example}
Consider $(\bk^n)^{\otimes 3}$.  This decomposes into irreducible representations in
the following ways:
\begin{itemize}
\item For $n=2$, as $(\bk^2)^{\otimes 3}= \bW_{1}(2)^{\oplus 2} \oplus \bW_3(2)$.
\item For $n=3$, as $(\bk^3)^{\otimes 3}= \bW_0(3) \oplus \bW_{2,1}(3)^{\oplus 2} \oplus \bW_3(3)$.
\item For $n \geq 4$, as $(\bk^n)^{\otimes 3}= \bW_{1^3}(n) \oplus \bW_{2,1}(n)^{\oplus 2} \oplus \bW_3(n)$.
\end{itemize}
Here for $n \geq 4$ two of the factors are $\bW_{1^3}(n) = \wedge^3 \bk^n$ and $\bW_3(n) = \Sym^3(\bk^n)$, while
the other factor $\bW_{2,1}(n)$ which appears with multiplicity $2$ is harder to interpret.  
Let $\{e_1,\ldots,e_n\}$ be the standard basis for $\bk^n$.
Embedding
$\wedge^2 \bk^n$ into $(\bk^n)^{\otimes 2}$ in the usual way, the following are highest weight vectors
with weight $2 E_1 + E_2$:
\begin{align*}
w  &= e_1 \otimes (e_1 \wedge e_2) = e_1 \otimes e_1 \otimes e_2 - e_1 \otimes e_2 \otimes e_1,\\
w' &= (e_1 \wedge e_2) \otimes e_1 = e_1 \otimes e_2 \otimes e_1 - e_2 \otimes e_1 \otimes e_1.
\end{align*}
The subrepresentations of $(\bk^n)^{\otimes 3}$ spanned by $w$ and $w'$ are the two copies of $\bW_{2,1}(n)$.
These are not unique, and any nontrivial linear combination of $w$ and $w'$ is also a highest
weight vector of weight $2 E_1 + E_2$.
\end{example}

More generally, if $U$ is a representation of $\SL_n$ constructed from the standard
representation $\bk^n$ using tensor
powers, exterior powers, and symmetric powers, then $U$ naturally embeds
into $(\bk^n)^{\oplus d}$ for a $d \geq 1$ called its degree, and the decomposition
of $U$ into irreducible factors is independent of $n$ as long as $n \geq d+1$.
For example:

\begin{example}
Let $U = \Sym^2(\wedge^3 \bk^n) \otimes \bk^n$.  Then $U$ embeds into
$(\bk^n)^{\otimes 7}$ and has degree $7$, and the decomposition of $U$
into irreducible representations is independent of $n$ once $n \geq 8$.
\end{example}

\begin{convention}
In light of all of this, we can decompose representations of $\SL_n$ for $n \gg 0$
by using a computer to make this decomposition for $n$ larger than the degree of
the representation.  This will be done silently throughout the remainder of the
paper.
\end{convention}

\begin{remark}
The idea of stable decompositions of representations has since been
subsumed into Church--Farb's notion of representation stability \cite{ChurchFarbRepStability}.  See
\cite{FarbICM} for a survey.
\end{remark}

\section{Free groups and free Lie algebras}
\label{section:freelie}

Our next topic is the Lie algebra associated to the lower central series
of a group.  Everything we discuss here without reference can be found
in \cite{SerreLie}.

\subsection{Lower central series}

Let $G$ be a group.  The {\em lower central series} of $G$ is the
following inductively defined sequence of subgroups $\gamma_d(G)$:
\[\gamma_1(G) = G \quad \text{and} \quad \text{$\gamma_{d+1}(G) = [G,\gamma_d(G)]$ for $d \geq 1$}.\]
By definition, $G$ is nilpotent of degree at most $d$ if $\gamma_{d+1}(G)=1$.  Each
quotient group $\gamma_d(G)/\gamma_{d+1}(G)$ is abelian.  In fact, even more is
true: the conjugation action of $G$ on its normal subgroup $\gamma_d(G)$ descends
to a trivial action of $G$ on $\gamma_d(G)/\gamma_{d+1}(G)$.  This subsumes
being abelian since it implies in particular that the conjugation action
of $\gamma_d(G)$ on itself descends to the trivial action on $\gamma_d(G)/\gamma_{d+1}(G)$.

\subsection{Lie ring associated to the lower central series}

For a group $G$ and a field $\bk$, define
\[\cL_d(G;\bk) = \gamma_d(G)/\gamma_{d+1}(G) \otimes \bk \quad \text{for $d \geq 1$}.\]
The commutator bracket $[-,-]$ on $G$ descends to a bilinear map
$\cL_d(G;\bk) \times \cL_e(G;\bk) \rightarrow \cL_{d+e}(G;\bk)$ that we will also
denote by $[-,-]$.  Set
\[\cL(G;\bk) = \bigoplus_{d=1}^{\infty} \cL_d(G;\bk).\]
The bracket discussed above turns $\cL(G;\bk)$ into a graded Lie algebra over $\bk$ that
is generated by its degree-$1$ elements $\cL_1(G;\bk) \cong G^{\text{ab}} \otimes \bk$.

\subsection{Free groups}

Let $F_n$ be the free group on $n$ letters $\{x_1,\ldots,x_n\}$.  We thus
have 
\[\cL_1(F_n;\bk) = F_n^{\text{ab}} \otimes \bk \cong \bk^n.\]  
For $1 \leq i \leq n$, let $e_i \in \cL_1(F_n;\bk)$ be the image of $x_i$.  It follows
from work of Magnus and Witt that the
graded Lie algebra $\cL(F_n;\bk)$ is isomorphic to the free Lie algebra
$\FLie(\bk^n)$ over $\bk$ generated
by $\{e_1,\ldots,e_n\}$.  

\subsection{Free Lie algebra in tensor algebra}

Another way of viewing
the free Lie algebra is as follows.  Let
\[\cT(\bk^n) = \bigoplus_{d \geq 1} (\bk^n)^{\otimes d}\]
be the tensor algebra generated by $\bk^n$.  This can be turned into a graded Lie algebra
using the bracket
\[[t,t'] = t t' - t' t \quad \text{for $t,t' \in \cT(\bk^n)$}.\]
We then have that $\FLie(\bk^n)$ is isomorphic to the sub-Lie algebra of
$\cT(\bk^n)$ generated by its elements of degree $1$, i.e., generated by the standard basis
$\{e_1,\dots,e_n\}$ for $\bk^n$.  In fact, $\cT(\bk^n)$ is what is called the universal enveloping
algebra of the free Lie algebra.

\subsection{Lie representation}

Let $\bk$ be a field of characteristic $0$.
The action of $\SL_n(\bk)$ on $\bk^n$ induces an action of $\SL_n(\bk)$ on
$\FLie(\bk^n)$.  Each graded term $\FLie_d(\bk^n)$ is
a finite-dimensional algebraic representation of $\SL_n(\bk)$.  In fact,
using the embedding of the free Lie algebra into the tensor algebra we
see that $\FLie_d(\bk^n)$ is a subrepresentation of $(\bk^n)^{\otimes d}$.
It is often called the $d^{\text{th}}$ Lie representation of $\SL_n(\bk)$.
See \cite{Klyachko} for a complete description of the decomposition
of the Lie representation into irreducible factors.  

\subsection{Low degree Lie representations}

We will only need to understand the first few terms of the
Lie representation, which can easily be worked out by hand.
Let $E_1,\ldots,E_n$ be the characters of the maximal split
torus of $\SL_n$, so
\[E_i(\diag(t_1,\ldots,t_n)) = t_i \quad \text{for all $t_1,\ldots,t_n \in \bk^{\times}$ such that $t_1 \cdots t_n = 1$}.\]
To simplify our notation in the calculations below, we will omit the $\otimes$ symbols
when writing elements of tensor powers of $\bk^n$.  For instance, $e_1 e_2 e_1$ stands
for $e_1 \otimes e_2 \otimes e_1$.  Just like before, all the representation 
theory calculations in the examples below were performed with LiE; see \cite{LieProgram}.

\begin{example}
\label{example:lie1}
For $n \geq 2$, we have $\FLie_1(\bk^n) = \bk^n \cong \bW_1(n)$.
\end{example}

\begin{example}
\label{example:lie2}
For $n \geq 3$, we have $\FLie_2(\bk^n) = \wedge^2 \bk^n \cong \bW_{1^2}(n)$.
This reflects the fact that the Lie bracket is alternating.  Regarding
$\FLie_2(\bk^n)$ as a subspace of $(\bk^n)^{\otimes 2}$, it
has the following highest weight vector with weight $E_1+E_2$:
\[[e_1,e_2] = e_1 e_2 - e_2 e_1.\qedhere\]
\end{example}

\begin{example}
\label{example:lie3}
For $n \geq 4$, we have
\[\FLie_3(\bk^n) = \frac{\bk^n \otimes \wedge^2 \bk^n}{\wedge^3 \bk^n} \cong \bW_{2,1}(n).\]
This isomorphism arises from the surjective Lie bracket map from
\[\bk^n \otimes \FLie_2(\bk^n) = \bk^n \otimes \wedge^2 \bk^n \cong \bW_{2,1}(n) \oplus \bW_{1^3}(n)\]
to $\FLie_3(\bk^n)$.  The fact that $\wedge^3 \bk^n \cong \bW_{1^3}(n)$ is in the kernel
of this Lie bracket map follows from the Jacobi identity.  Regarding
$\FLie_3(\bk^n)$ as a subspace of $(\bk^n)^{\otimes 3}$, it
has the following highest weight vector with weight $2E_1+E_2$:
\begin{align*}
[e_1,[e_1,e_2]] &= [e_1,e_1 e_2 - e_2 e_1] = e_1 e_1 e_2 - 2e_1 e_2 e_1 + e_2 e_1 e_1.\qedhere
\end{align*}
\end{example}

\begin{example}
\label{example:lie4}
For $n \geq 5$, we have
\[\FLie_4(\bk^n) \cong \bW_{3,1}(n) \oplus \bW_{2,1^2}(n).\]
To see this, note that just like in degree $3$ the Lie bracket map gives a surjection from
\[\bk^n \otimes \FLie_3(\bk^n) \cong \bk^n \otimes \bW_{2,1}(n) \cong \bW_{3,1}(n) \oplus \bW_{2,1^2}(n) \oplus \bW_{2,2}(n)\]
to $\FLie_4(\bk^n)$.  To see that this Lie bracket map takes $\bk^n \otimes \FLie_3(\bk^n)$ to a representation
containing $\bW_{3,1}(n)$ and $\bW_{2,1^2}(n)$, note that its image contains the following highest
weight vectors:
\begin{align*}
[e_1,[e_1,[e_1,e_2]]] &= [e_1,e_1 e_1 e_2 - 2e_1 e_2 e_1 + e_2 e_1 e_1] \\
&= e_1 e_1 e_1 e_2 - 3 e_1 e_1 e_2 e_1 +3 e_1 e_2 e_1 e_1 - e_2 e_1 e_1 e_1, \\
[[e_1,e_2],[e_1,e_3]] &= [e_1 e_2 - e_2 e_1,e_1 e_3-e_3 e_1] \\
&= (e_1 e_2 e_1 e_3 - e_1 e_3 e_1 e_2) - (e_1 e_2 e_3 e_1 - e_3 e_1 e_1 e_2) \\
&\quad -(e_2 e_1 e_1 e_3 - e_1 e_3 e_2 e_1) + (e_2 e_1 e_3 e_1 - e_3 e_1 e_2 e_1).
\end{align*}
On the other hand, the kernel of the Lie bracket map contains a subrepresentation isomorphic
to $\bW_{2,2}(n)$.  Indeed, the domain of the Lie bracket map is
\[\bk^n \otimes \FLie_3(\bk^n) \cong \frac{(\bk^n)^{\otimes 2} \otimes (\wedge^2 \bk^n)}{\bk^n \otimes \wedge^3 \bk^n}.\]
The representation $(\bk^n)^{\otimes 2} \otimes (\wedge^2 \bk^n)$ contains the subrepresentation
$\Sym^2(\wedge^2 \bk^n)$, whose image in $\bk^n \otimes \FLie_3(\bk^n)$ lies in the kernel
of the Lie bracket map since the Lie bracket map is alternating.  Since $\Sym^2(\wedge^2 \bk^n)$ is not
contained in $\bk^n \otimes \wedge^3 \bk^n$, this give a nonzero subrepresentation in the
kernel of the Lie bracket map.  The only possibility for it is $\bW_{2,2}(n)$.  In fact,
\[\Sym^2(\wedge^2 \bk^n) \cong \bW_{2,2}(n) \oplus \bW_{1^4}(n) \cong \bW_{2,2}(n) \oplus \wedge^4 \bk^n,\]
so it must be the case that the image of $\Sym^2(\wedge^2 \bk^n)$ in
$\bk^n \otimes \FLie_3(\bk^n)$ is
\[\frac{\Sym^2(\wedge^2 \bk^n)}{\wedge^4 \bk^n} \cong \bW_{2,2}(n).\qedhere\]
\end{example}

\section{Johnson filtration and homomorphisms}
\label{section:johnsongeneral}

The first Johnson homomorphism was introduced by Johnson \cite{JohnsonHomo} following
earlier work by Sullivan \cite{SullivanJohnson}.  Johnson \cite{JohnsonSurvey} later
extended this to the higher Johnson homomorphisms.  This theory was then developed
further by Morita \cite{MoritaAbelian} and many others.  See \cite{HainJohnson, SatohJohnson}
for surveys.  This section discusses the basic properties of the
Johnson homomorphisms.

\subsection{Johnson filtration}

Recall that $\Sigma_g^1$ is a genus-$g$ surface with one boundary component.  Fix a basepoint
$\ast \in \partial \Sigma_g^1$ and let $\pi = \pi_1(\Sigma_g^1,\ast)$.  Since the mapping class
group $\Mod_g^1$ fixes $\partial \Sigma_g^1$ pointwise, the group $\Mod_g^1$ acts on $\pi$.
This action preserves the lower central series $\gamma_k(\pi)$, so for each $d \geq 1$ we get
an action of $\Mod_g^1$ on the $d$-step nilpotent group
\[N_d(\pi) = \pi / \gamma_{d+1}(\pi).\]
The $d^{\text{th}}$ term of the {\em Johnson filtration}, denoted $\Torelli_g^1[d]$, is the
kernel of the action of $\Mod_g^1$ on $N_d(\pi)$.  Since $N_1(\pi) = \pi^{\text{ab}} = \HH_1(\Sigma_g^1;\Z)$,
the first term of the Johnson filtration is the Torelli group $\Torelli_g^1$.  The Johnson
filtration thus forms a descending sequence of groups
\[\Torelli_g^1 = \Torelli_g^1[1] \rhd \Torelli_g^1[2] \rhd \Torelli_g^1[3] \rhd \cdots.\]
The Dehn--Nielsen--Baer theorem \cite[Chapter 8]{FarbMargalitPrimer} says that
$\Mod_g^1$ acts faithfully on $\pi$.  Since $\cap_{d=1}^{\infty} \gamma_{d+1}(\pi) = 1$ (see
\cite[Lemma 4.2]{FoxCommutator} or \cite[Theorem 1.2]{MalesteinPutmanLCS}), it follows
that $\cap_{d=1}^\infty \Torelli_g^1[d] = 1$.

\subsection{Basic properties}

We will soon prove that each $\Torelli_g^1[d] / \Torelli_g^1[d+1]$ is abelian; indeed, the
$d^{\text{th}}$ Johnson homomorphism is a homomorphism from $\Torelli_g^1[d]$ to an abelian
group whose kernel is $\Torelli_g^1[d+1]$.  Before doing this, we explain some basic
properties of the Johnson filtration.  

As we said above, $\Torelli_g^1[1]$ is the ordinary
Torelli group.  Building on work of Birman \cite{BirmanSiegel}, Powell (\cite{PowellTorelli}; see
\cite{HatcherMargalit, PutmanCutPaste} for modern proofs) proved
that $\Torelli_g^1 = \Torelli_g^1[1]$ is generated by two kinds of elements.  For a simple closed
curve $\gamma$ on $\Sigma_g^1$, let $T_{\gamma}$ denote the left Dehn twist about $\gamma$.  Powell's
generating set then includes:
\begin{itemize}
\item Separating twists, that is, Dehn twists $T_z$ such that $z$ is a simple closed separating
curve on $\Sigma_g^1$.
\item Bounding pair maps, that is products $T_x T_y^{-1}$ such that $x$ and $y$ are disjoint
nonseparating simple closed curves on $\Sigma_g^1$ such that $x \cup y$ separates $\Sigma_g^1$.
\end{itemize}
Johnson (\cite{Johnson2}; see \cite{PutmanJohnson} for a modern proof) proved that $\Torelli_g^1[2]$ is the subgroup
of Torelli group generated by Dehn twists about separating curves.  For $d \geq 3$, no
explicit generating set for $\Torelli_g^1[d]$ is known, though \cite{ChurchPutmanGenJohnson}
proves that for each $d \geq 1$ there is some $h_d$ such that $\Torelli_g^1[d]$ is generated
by mapping classes supported on subsurfaces of genus at most $h_d$ for all $g \geq h_d$.

As far as finiteness properties go, Johnson \cite{Johnson1} proved that
$\Torelli_g^1 = \Torelli_g^1[1]$ is finitely generated for $g \geq 3$.  This is in
contrast to $\Torelli_2^1$, which McCullough--Miller \cite{McCulloughMiller} proved
is not finitely generated.  More recently, Ershov--He \cite{ErshovHe} and
Church--Ershov--Putman \cite{ChurchErshovPutman} proved that $\Torelli_g^1[2]$ is
finitely generated for $g \geq 4$.  Ershov--Franz \cite{ErshovFranz} gave an enormous explicit
finite generating set for $\Torelli_g^1[2]$, though finding one of a reasonable size
is still an open question.  For $d \geq 3$, Church--Ershov--Putman \cite{ChurchErshovPutman}
proved that $\Torelli_g^1[d]$ is finitely generated for $g \geq 2d-1$.  

\subsection{Johnson homomorphism preliminaries, I}

Set $H = \HH_1(\Sigma_g^1;\Q)$.  Recall that $\FLie_d(H)$ is the $d^{\text{th}}$ graded
piece of the free Lie algebra generated by $H$.  Our next goal is to construct a homomorphism
$\tau_d\colon \Torelli_g^1[d] \rightarrow H \otimes \FLie_{d+1}(H)$ called the Johnson homomorphism
such that $\ker(\tau_d) = \Torelli_g^1[d+1]$.  The construction we give originated in
\cite{JohnsonSurvey} and was elaborated on by Morita \cite{MoritaAbelian}. 
Recall from \S \ref{section:freelie} that
\[\gamma_{d+1}(\pi) / \gamma_{d+2}(\pi) \otimes \Q \cong \FLie_{d+1}(H).\]
For $w \in \gamma_{d+1}(\pi)$, let $\Brack{w}_{d+1}$ be its image in $\FLie_{d+1}(H)$.  

Consider $f \in \Mod_g^1$.  For $x \in \pi$, set $f_{\ell}(x) = f(x) x^{-1}$.  We thus
have
\[f(x) = f_{\ell}(x) x \quad \text{for all $x \in \pi$}.\]
The $\ell$ stands for ``left''.  If $f \in \Torelli_g^1[d]$, then by definition
we have $f_{\ell}(x) \in \gamma_{d+1}(\pi)$ for all $x \in \pi$.  It
thus makes sense to talk about $\Brack{f_{\ell}(x)}_{d+1} \in \FLie_{d+1}(H)$.
These elements satisfy the following key lemma:

\begin{lemma}
\label{lemma:johnsonhomo1}
Let the notation be as above, and consider $f \in \Torelli_g^1[d]$.  Then:
\begin{itemize}
\item[(i)] For $x,y \in \pi$, we have $\Brack{f_{\ell}(xy)}_{d+1} = \Brack{f_{\ell}(x)}_{d+1} + \Brack{f_{\ell}(y)}_{d+1}$.
\item[(ii)] For $z \in [\pi,\pi]$, we have $\Brack{f_{\ell}(z)}_{d+1} = 0$.
\end{itemize}
\end{lemma}
\begin{proof}
Conclusion (ii) follows from conclusion (i) since conclusion (i) implies that the map
\[z \mapsto \Brack{f_{\ell}(z)}_{d+1} \quad \text{for all $z \in \pi$}\]
is a homomorphism from $\pi$ to the abelian group $\FLie_{d+1}(H)$.  We must therefore only
prove (i).  Consider $x,y \in \pi$.  We have $f(xy) = f_{\ell}(xy) xy$ and
\[f(xy) = f(x) f(y) = f_{\ell}(x) x f_{\ell}(y) y = f_{\ell}(x) x f_{\ell}(y) x^{-1} x y,\]
so $f_{\ell}(xy) = f_{\ell}(x) x f_{\ell}(y) x^{-1}$.  It follows that
\[\Brack{f_{\ell}(xy)}_{d+1} = \Brack{f_{\ell}(x) x f_{\ell}(y) x^{-1}}_{d+1}
= \Brack{f_{\ell}(x)}_{d+1} + \Brack{x f_{\ell}(y) x^{-1}}_{d+1}
= \Brack{f_{\ell}(x)}_{d+1} + \Brack{f_{\ell}(y)}_{d+1}.\]
Here the final equality follows from the fact that the conjugation action of $\pi$ on
$\gamma_{d+1}(\pi)$ descends to the trivial action on $\gamma_{d+1}(\pi) / \gamma_{d+2}(\pi)$.
\end{proof}

\subsection{Johnson homomorphism preliminaries, II}

For $f \in \Torelli_g^1[d]$, define $\htau_d(f)\colon \pi \rightarrow \FLie_{d+1}(H)$ via the formula
\[\htau_d(f)(x) = \Brack{f_{\ell}(x)}_{d+1} \quad \text{for all $x \in \pi$}.\]
Lemma \ref{lemma:johnsonhomo1} implies that $\htau_d(f)$ is a homomorphism.  
The collection of homomorphisms $\pi \rightarrow \FLie_{d+1}(H)$ forms a $\Q$-vector space,
so it makes sense to add two of them.  We then have:

\begin{lemma}
\label{lemma:johnsonhomo2}
Let the notation be as above.  For $f,h \in \Torelli_g^1[d]$, we have
$\htau_d(fh) = \htau_d(f) + \htau_d(h)$.
\end{lemma}
\begin{proof}
Consider $x \in \pi$.  By definition,
\[f h(x) = f(h_{\ell}(x) x) = = f(h_{\ell}(x)) f(x) = f_{\ell}(h_{\ell}(x)) h_{\ell}(x) f_{\ell}(x) x.\]
It follows that
\begin{align*}
\htau_d(f h)(x) &= \Brack{f_{\ell}(h_{\ell}(x)) h_{\ell}(x) f_{\ell}(x)}_{d+1} \\
                &= \Brack{f_{\ell}(h_{\ell}(x))}_{d+1} + \Brack{h_{\ell}(x)}_{d+1} + \Brack{f_{\ell}(x)}_{d+1} \\
                &= 0 + \htau_d(f)(x) + \htau_d(h)(x).
\end{align*}
Here the final equality uses the fact that
\[h_{\ell}(x) \in \gamma_{d+1}(\pi) \subset [\pi,\pi],\]
so by conclusion (ii) of Lemma \ref{lemma:johnsonhomo1} we have
$\Brack{f_{\ell}(h_{\ell}(x))}_{d+1} = 0$.
\end{proof}

\subsection{The Johnson homomorphism}
\label{section:johnsondef}

Lemma \ref{lemma:johnsonhomo2} implies that we can define a homomorphism
$\htau_d\colon \Torelli_g^1[d] \rightarrow \Hom(\pi,\FLie_{d+1}(H))$ via the formula
\[\htau_d(f) = \Brack{f_{\ell}(x)}_{d+1} \quad \text{for all $x \in \pi$}.\]
Since $\FLie_{d+1}(H)$ is a $\Q$-vector space and $H = \pi^{\text{ab}} \otimes \Q$, we have
a canonical identification
\[\Hom(\pi,\FLie_{d+1}(H)) = \Hom(H,\FLie_{d+1}(H)).\]
Let $\omega$ be the algebraic intersection form on $H = \HH_1(\Sigma_g^1;\Q)$.  This is a symplectic
form, and thus establishes an isomorphism between $H$ and its dual $H^{\ast} = \Hom(H,\Q)$.  Using
$\omega$, we can identify
\[\Hom(H,\FLie_{d+1}(H)) = H^{\ast} \otimes \FLie_{d+1}(H) = H \otimes \FLie_{d+1}(H).\]
This identification identifies $x \otimes \kappa \in H \otimes \FLie_{d+1}(H)$ with the following
map $H \rightarrow \FLie_{d+1}(H)$:
\[h \mapsto \omega(x,h) \kappa \quad \text{for $h \in H$}.\]
Combining all of our identifications, we define 
$\tau_d\colon \Torelli_g^1[d] \rightarrow \FLie_{d+1}(H) \otimes H$ to be the composition
\[\begin{tikzcd}
\Torelli_g^1[d] \arrow{r}{\htau_d} & \Hom(\pi,\FLie_{d+1}(H)) = \Hom(H,\FLie_{d+1}(H)) = H \otimes \FLie_{d+1}(H).
\end{tikzcd}\]
This is the $d^{\text{th}}$ Johnson homomorphism.  

\subsection{Basic properties of Johnson homomorphism}

We close this section by discussing two basic properties of the Johnson homomorphism.  The
first is as follows:

\begin{lemma}
\label{lemma:johnsonker}
For all $d \geq 1$, the kernel of $\tau_d\colon \Torelli_g^1[d] \rightarrow H \otimes \FLie_{d+1}(H)$ is
$\Torelli_g^1[d+1]$.  Consequently, $\Torelli_g^1[d]/\Torelli_g^1[d+1]$ is free abelian.
\end{lemma}
\begin{proof}
By construction, the image of $\tau_d$ lies in the integer points of $H \otimes \FLie_{d+1}(H)$, so the second
conclusion follows from the first.  To prove the first conclusion, consider $f \in \Torelli_g^1[d]$.  We
have $\tau_d(f) = 0$ if and only if $\Brack{f_{\ell}(x)}_{d+1} = 0$ for all $x \in \pi$.  By definition,
$\Brack{f_{\ell}(x)}_{d+1}$ is the image in 
\[\FLie_{d+1}(H) = \gamma_{d+1}(\pi) / \gamma_{d+2}(\pi) \otimes \Q\]
of $f_{\ell}(x) \in \gamma_{d+1}(\pi)$.  Since $\pi$ is a free group, $\gamma_{d+1}(\pi) / \gamma_{d+2}(\pi)$ is free abelian (see \cite{SerreLie}).
It follows that $\Brack{f_{\ell}(x)}_{d+1} = 0$ for all $x \in \pi$ if and only if $f_{\ell}(x) \in \gamma_{d+2}(\pi)$ for all $x \in \pi$.
Since $f_{\ell}(x) = f(x) x^{-1}$, this holds if and only if $f$ acts trivially on $\pi/\gamma_{d+2}(\pi)$, i.e., if and only if
$f \in \Torelli_g^1[d+1]$.
\end{proof}

For the second, the action of $\Mod_g^1$ on $H = \HH_1(\Sigma_g^1;\Q)$ induces an action of $\Mod_g^1$
on $H \otimes \FLie_{d+1}(H)$.  For $\kappa \in H \otimes \FLie_{d+1}(H)$ and $\phi \in \Mod_g^1$, write
$\phi_{\ast}(\kappa)$ for the image of $\kappa$ under $\phi$.  We then have:

\begin{lemma}
\label{lemma:johnsonequivariant}
Let $d \geq 1$.  For all $f \in \Torelli_g^1[d]$ and $\phi \in \Mod_g^1$, we have
$\tau_d(\phi f \phi^{-1}) = \phi_{\ast}(\tau_d(f))$.
\end{lemma}
\begin{proof}
Regard $\tau_d(\phi f \phi^{-1})$ as a homomorphism $\pi \rightarrow \FLie_{d+1}(H)$.  By definition,
for $x \in \pi$ we have
\[\tau_d(\phi f \phi^{-1})(x) = \Brack{(\phi f \phi^{-1})_{\ell}(x)}_{d+1}.\]
We have
\begin{align*}
(\phi f \phi^{-1})_{\ell}(x) &= \phi(f(\phi^{-1}(x))) x^{-1} = \phi\left(f(\phi^{-1}(x)) \phi^{-1}(x)\right) 
                             = \phi\left(f_{\ell}(\phi^{-1}(x))\right),
\end{align*}
so
\begin{align*}
\tau_d(\phi f \phi^{-1})(x) &= \Brack{\phi\left(f_{\ell}(\phi^{-1}(x))\right)}_{d+1} 
                            = \phi \Brack{f_{\ell}(\phi^{-1}(x))}_{d+1} 
                            = \phi\left(\tau_d(f)(\phi^{-1}(x))\right).
\end{align*}
This is exactly how $\phi$ is supposed to act on homomorphisms $\pi \rightarrow \FLie_{d+1}(H)$.
\end{proof}

This has the following corollary.  The action of $\Mod_g^1$ on $H = \HH_1(\Sigma_g^1)$ preserves the
algebraic intersection form $\omega$, which is a symplectic form.  As we discussed in the introduction,
the resulting homomorphism $\Mod_g^1 \rightarrow \Sp_{2g}(\Z)$ is surjective, and by definition
$\Torelli_g^1$ is its kernel:
\[1 \longrightarrow \Torelli_g^1 \longrightarrow \Mod_g^1 \longrightarrow \Sp_{2g}(\Z) \longrightarrow 1.\]
The action of $\Sp_{2g}(\Z)$ on $H$ extends to the algebraic group $\Sp_{2g}(\Q)$.  It follows that
the action of $\Mod_g^1$ on $H \otimes \FLie_{d+1}(H)$ also comes from an action of $\Sp_{2g}(\Z)$
that extends to the algebraic group $\Sp_{2g}(\Q)$.  We then have:

\begin{corollary}
\label{corollary:imagejohnsonalgebraic}
For all $g,d \geq 1$, the $\Q$-span of the image of $\tau_d\colon \Torelli_g^1 \rightarrow H \otimes \FLie_{d+1}(H)$
is an $\Sp_{2g}(\Q)$-subrepresentation of $H \otimes \FLie_{d+1}(H)$.
\end{corollary}
\begin{proof}
Lemma \ref{lemma:johnsonequivariant} implies that the $\Q$-span of the image of $\tau_d$ is an $\Sp_{2g}(\Z)$-subrepresentation.
Since $\Sp_{2g}(\Z)$ is Zariski dense in $\Sp_{2g}(\Q)$, all $\Sp_{2g}(\Z)$-subrepresentations are
actually $\Sp_{2g}(\Q)$-subrepresentations.  The corollary follows.
\end{proof}

This corollary implies that to understand the image of $\tau_d$, we need to understand the
representation theory of the algebraic group $\Sp_{2g}(\Q)$.  We will discuss this in \S \ref{section:sprep} below
after first performing a number of calculations.

\section{Johnson homomorphism calculations}
\label{section:johnsoncalc}

We next calculate the images under the first and second Johnson homomorphisms of
some basic elements.  Set $H = \HH_1(\Sigma_g^1;\Q)$ and
$\pi = \pi_1(\Sigma_g^1,\ast)$ with $\ast \in \partial \Sigma_g^1$.

\subsection{Separating twists, first Johnson homomorphism}

We start with the following.  Recall that the first Johnson homomorphism
takes the form $\tau_1\colon \Torelli_g^1[1] \rightarrow H \otimes \Lie_2(H)$.  Since
$\Lie_2(H) \cong \wedge^2 H$ (see Example \ref{example:lie2}), the target of this
is $H \otimes \wedge^2 H$.

\begin{lemma}
\label{lemma:tau1septwist}
For some $g \geq 1$, let $T_z$ be a separating twist in $\Torelli_g^1 = \Torelli_g^1[1]$.  Then
$\tau_1(T_z) = 0$.
\end{lemma}
\begin{proof}
Let $S$ and $S'$ be the subsurfaces of $\Sigma_g^1$ on either side of $z$.  Order
them such that $\partial \Sigma_g^1 \subset S'$, so $S \cong \Sigma_h^1$ for some
$1 \leq h \leq g$.  Let $\zeta \in \pi$ be the following curve:\\
\Figure{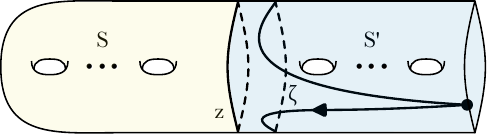}
Regard $\tau_1(T_z)$ as a map $H \rightarrow \FLie_2(H)$.  To prove that $\tau_1(T_z) = 0$,
it is enough to prove that $\tau_1(T_z)$ vanishes on the images of $\HH_1(S)$ and $\HH_1(S')$
in $H$.  This is trivial for the image of $\HH_1(S')$, so we must only prove it
for the image of $\HH_1(S)$.  Identify $\HH_1(S)$ with its image in $H$.
As the following shows, an element $u \in \HH_1(S)$ can be written as $u = [\sigma]$, where
$\sigma \in \pi$ satisfies $T_z(\sigma) = \zeta \sigma \zeta^{-1}$:\\
\Figure{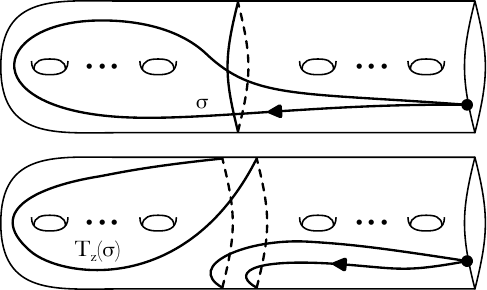}
We have $(T_z)_{\ell}(\sigma) = \zeta \sigma \zeta^{-1} \sigma^{-1} = [\zeta,\sigma]$, so
$\tau_1(T_z)(u) = \Brack{[\zeta,\sigma]}_2$.  Since $\zeta \in [\pi,\pi]$ we have
$[\zeta,\sigma] \in \gamma_3(\pi)$, so $\Brack{[\zeta,\sigma]}_2 = 0$, as desired.
\end{proof}

\begin{remark}
Lemma \ref{lemma:tau1septwist} implies that separating twists lie in $\Torelli_g^1[2]$.  See \S \ref{section:tau2septwist}
for how to calculate their image under 
$\tau_2\colon \Torelli_g^1[2] \rightarrow H \otimes \FLie_3(H)$.
\end{remark}

\subsection{Separating simple closed curves}

Before we perform our next Johnson homomorphism calculation, we need a preliminary result.  As
in the following figure, let
$\zeta \in \pi$ be a simple closed curve that bounds on its right a subsurface
$S$ with $S \cong \Sigma_h^1$ for some $h \geq 1$:\\
\Figure{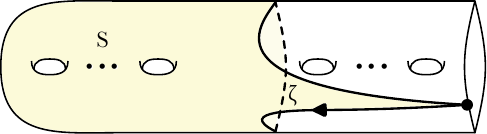}
We have $\zeta \in [\pi,\pi]$, so it makes sense to discuss $\Brack{\zeta}_{2} \in \FLie_2(H) \cong \wedge^2 H$.
Our goal is to calculate $\Brack{\zeta}_2$.

Identify $\HH_1(S;\Q)$ with its image in $H$.  The subspace $\HH_1(S;\Q)$ is a {\em symplectic subspace}
of $H$, i.e., a subspace on which the algebraic intersection pairing restricts to a nondegenerate form.  Let
$V$ be a symplectic subspace of $H$ and let $\omega_V$ be the restriction of the algebraic intersection
form to $V$.  Since $\omega_V$ is nondegenerate, it identifies $V$ with its dual $V^{\ast}$.  Using
this identification, we can identify $\omega_V$ with an element
\[\omega_V \in (\wedge^2 V)^{\ast} = \wedge^2 V \subset \wedge^2 H.\]
If $\{u_1,v_1,\ldots,u_h,v_h\}$ is a symplectic basis for $V$, then $\omega_V = u_1 \wedge v_1 + \cdots + u_h \wedge v_h$.
We then have the following:

\begin{lemma}
\label{lemma:separatingcurvebrack}
Let $\zeta \in \pi$ be a simple closed curve that bounds on its right a subsurface
$S$ with $S \cong \Sigma_h^1$ for some $h \geq 1$.  Set $V = \HH_1(S;\Q)$.  We
then have $\Brack{\zeta}_2 = \omega_V$.
\end{lemma}
\begin{proof}
As the figure below shows, we can write $\zeta = [\mu_1,\nu_1] \cdots [\mu_h,\nu_h]$ for simple closed curves
$\mu_1,\nu_1,\ldots,\mu_h,\nu_h \in \pi$ such that the following holds:
\begin{itemize}
\item Letting $u_i = [\mu_i] \in H$ and $v_i = [\nu_i] \in H$ be the homology
classes of $\mu_i,\nu_i \in \pi$, the vectors $\{u_1,v_1,\ldots,u_h,v_h\}$ form
a symplectic basis for $V = \HH_1(S;\Q)$.
\end{itemize}
\Figure{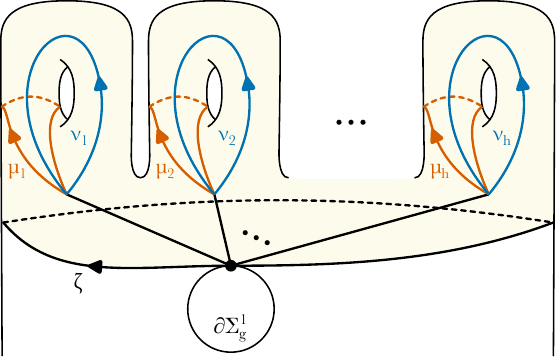}
We then have
\begin{align*}
\Brack{\zeta}_2 &= \Brack{[\mu_1,\nu_1] \cdots [\mu_h,\nu_h]}_2 = \Brack{[\mu_1,\nu_1]}_2 + \cdots + \Brack{[\mu_h,\nu_h]}_2 \\
                &= u_1 \wedge v_1 + \cdots + u_h \wedge v_h = \omega_V.\qedhere
\end{align*}
\end{proof}

\subsection{Bounding pair maps, I}

Recall that a bounding pair map is a product $T_x T_y^{-1} \in \Torelli_g^1$, where $x$ and $y$ are disjoint
nonseparating curves on $\Sigma_g^1$ such that $x \cup y$ separates $\Sigma_g^1$.  Let $S$ and $S'$ be the
subsurfaces on either side of $x \cup y$, ordered such that $\partial \Sigma_g^1 \subset S'$.  We
therefore have $S \cong \Sigma_h^2$ and $S' \cong \Sigma_{g-h-1}^3$ for some $1 \leq h \leq g-1$:\\
\Figure{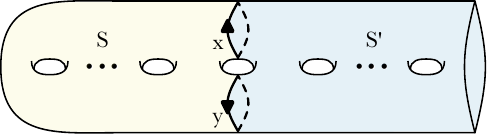}
As in this figure, orient $x$ and $y$ such that $S$ is to the left of $x$ and the right of $y$.  We
call this the {\em canonical orientation} of $x$ and $y$.  Let $U \subset H$ and $U' \subset H$ be
the images of $\HH_1(S;\Q)$ and $\HH_1(S';\Q)$.  We call the ordered pair $(U,U')$ the
{\em bounded subspaces} of $x \cup y$.  Let $\omega$ be the algebraic intersection pairing on $H$.  For
a subspace $W \subset H$, let
\[W^{\perp} = \Set{$h \in H$}{$\omega(h,w) = 0$ for all $w \in W$}.\]
We have $U + U' = \Span{[x]}^{\perp}$ and $U \cap U' = \Span{[x]}$.  Finally, say that an {\em integral dual} to
$[x]$ is an element $h \in H$ with the following two properties:
\begin{itemize}
\item $\omega(h,[x]) = 1$; and
\item $h \in \HH_1(\Sigma_g^1;\Z) \subset H$, i.e., $h$ is an integral point of $H$.
\end{itemize}
This implies that $\Span{[x],h}^{\perp}$ is a symplectic subspace of $H$.  Letting
$V = U \cap \Span{[x],h}^{\perp}$ and $V' = U' \cap \Span{[x],h}^{\perp}$, it also implies that $V$ and $V'$
are symplectic subspaces of $H$ such that $\Span{[x],h}^{\perp} = V \oplus V'$.  This direct sum
decomposition is orthogonal with respect to $\omega$.  We call $(V,V')$ the {\em bounded splitting} of
the pair $(x \cup y,h)$.  With these preliminaries, we have the following: 

\begin{lemma}
\label{lemma:taubpmap}
Let $T_x T_y^{-1} \in \Torelli_g^1$ be a bounding pair map.  Give $x$ and $y$ their canonical orientations, and
let $(U,U')$ be the bounded subspaces of $x \cup y$.
Regard $\tau_1(T_x T_y^{-1})$ as a map $H \rightarrow \wedge^2 H$.  The following hold:
\begin{itemize}
\item[(i)] For $u \in U$, we have $\tau_1(T_x T_y^{-1})(u) = -[x] \wedge u$.
\item[(ii)] For $u' \in U'$, we have $\tau_1(T_x T_y^{-1})(u') = 0$.
\item[(iii)] For an integral dual $h \in H$ to $[x]$, let $(V,V')$ be the bounded splitting of the pair
$(x \cup y,h)$.  We then have $\tau_1(T_x T_y^{-1})(h) = -\omega_V$.
\end{itemize}
\end{lemma}
\begin{proof}
We prove the three parts separately:

\begin{claim}{1}
For $u \in U$, we have $\tau_1(T_x T_y^{-1})(u) = -[x] \wedge u$.
\end{claim}

Let $\xi \in \pi$ be the following curve, so $[\xi] = -[y] = -[x]$; see here:\\
\Figure{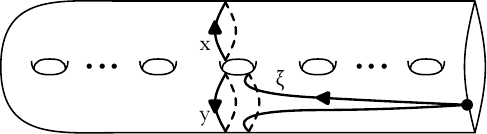}
It is enough to prove the claim for $u$ an integral point of $U$.  This implies that
$u = [\mu]$ for a curve $\mu \in \pi$ with $T_x T_y^{-1}(\mu) = \xi \mu \xi^{-1}$; see here:\\
\Figure{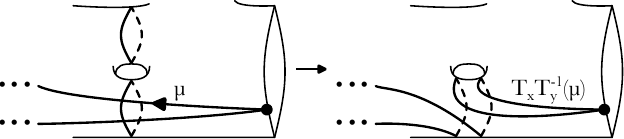}
This implies that $(T_x T_y^{-1})_{\ell}(\mu) = \xi \mu \xi^{-1} \mu^{-1} = [\xi,\mu]$, so
\[\tau_1(T_x T_y^{-1})(u) = \Brack{[\xi,\mu]}_2 = [\xi] \wedge [\mu] = -[x] \wedge u.\]

\begin{claim}{2}
For $u' \in U'$, we have $\tau_1(T_x T_y^{-1})(u') = 0$.
\end{claim}

It is enough to prove the claim for $u'$ an integral point of $U'$.  This implies that
$u' = [\mu']$ for a curve $\mu' \in \pi$ that is fixed by $T_x T_{y}^{-1}$.  For such a
$u'$, the claim is obvious.

\begin{claim}{3}
For an integral dual $h \in H$ to $[x]$, let $(V,V')$ be the bounded splitting of the pair
$(x \cup y,h)$.  We then have $\tau_1(T_x T_y^{-1})(h) = -\omega_V$.
\end{claim}

To make our pictures easier to follow, we will invert $T_x T_y^{-1}$ and prove
that $\tau_1(T_y T_x^{-1}) = \omega_V$.  Since $h$ is an integral dual to $[x]$, we
can find a simple closed curve $\nu \in \pi$ with $h = [\nu]$ that intersects
both $x$ and $y$ once; see here:\\
\Figure{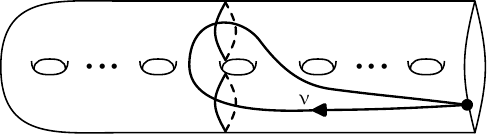}
Set $\zeta = (T_y T_x^{-1})_{\ell}(\nu) = \left(T_y T_x^{-1}(\nu)\right) \nu^{-1}$; see here:\\
\Figure{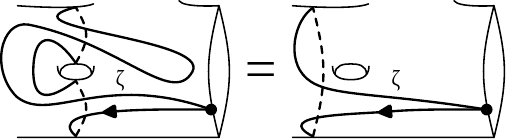}
As this figure shows, $\zeta$ is a simple closed separating curve that bounds on its
right a surface $T$ with $\HH_1(T;\Q) = V$.  It therefore follows from Lemma \ref{lemma:separatingcurvebrack}
that
\[\tau_1(T_y T_x^{-1})(h) = \Brack{(T_y T_x^{-1})_{\ell}(\nu)}_2 = \Brack{\zeta}_2 = \omega_V.\qedhere\]
\end{proof}

\subsection{Bounding pair maps, II}

We next translate Lemma \ref{lemma:taubpmap} into a more useful form.  The first
Johnson homomorphism is of the form $\tau_1\colon \Torelli_g^1 \rightarrow H \otimes \wedge^2 H$.
Its target $H \otimes \wedge^2$ contains a copy of $\wedge^3 H$, namely the image of the
map
\[h_1 \wedge h_2 \wedge h_3 \in \wedge^3 H \mapsto h_1 \otimes (h_2 \otimes h_3) - h_2 \otimes (h_1 \wedge h_3) + h_3 \otimes (h_1 \wedge h_2) \in H \otimes \wedge^3 H.\]
We then have the following.

\begin{lemma}
\label{lemma:taubp}
Let $T_x T_y^{-1} \in \Torelli_g^1$ be a bounding pair map.  Give $x$ and $y$ their canonical orientations, and
let $\{u_1,v_1,\ldots,u_h,v_h\}$ be the homology classes of loops as in the following figure:\\
\Figure{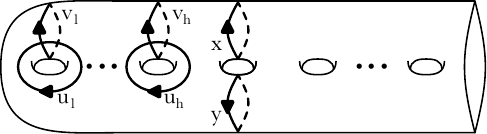}
Then 
$\tau_1(T_x T_y^{-1}) = [x] \wedge (u_1 \wedge v_1 + \cdots + u_h \wedge v_h) \in \wedge^3 H \subset H \otimes \wedge^2 H$.
\end{lemma}
\begin{proof}
Complete $\{u_1,v_1,\ldots,u_h,v_h\}$ to a symplectic basis 
$\{u_1,v_1,\ldots,u_g,v_g\}$
for $H$ with $v_{h+1} = [x]$ and $u_{h+1}$ an
integral dual to $[x]$.  Lemma \ref{lemma:taubpmap} says that regarded as a map $H \rightarrow \wedge^2 H$,
we have that $\tau_1(T_x T_y^{-1})$ is given by the following formulas:
\begin{itemize}
\item $\tau_1(T_x T_y^{-1})(u_i) = -[x] \wedge u_i$ and
$\tau_1(T_x T_y^{-1})(v_i) = -[x] \wedge v_i$ for $1 \leq i \leq h$.
\item $\tau_1(T_x T_y^{-1})(u_{h+1}) = -(u_1 \wedge v_1 + \cdots + u_h \wedge v_h)$.
\item $\tau_1(T_x T_y^{-1})$ vanishes on all other elements of $\{u_1,v_1,\ldots,u_g,v_g\}$.
\end{itemize}
Translating this into an element of $H \otimes \wedge^2 H$, we get
\begin{align*}
&\left(\sum_{i=1}^h v_i \otimes ([x] \wedge u_i) - u_i \otimes ([x] \wedge v_i)\right) + [x] \otimes (u_1 \wedge v_1 + \cdots + u_h \wedge v_h) \\
=&\sum_{i=1}^h \left([x] \otimes (u_i \wedge v_i) - u_i \otimes ([x] \wedge v_i) + v_i \otimes ([x] \wedge u_i)\right) 
=\sum_{i=1}^h [x] \wedge u_i \wedge v_i.\qedhere
\end{align*}
\end{proof}

\subsection{Simply intersecting pair maps}

Let $\partial_1,\partial_2,\partial_3$ be three disjoint oriented simple closed curves on
$\Sigma_g^1$.  Our next goal is to construct $f \in \Torelli_g^1$ with 
$\tau_1(f) = \pm \partial_1 \wedge \partial_2 \wedge \partial_3$ for some choice of sign.  
Assume that $Z$ is a subsurface
of $\Sigma_g^1$ such that $Z \cong \Sigma_0^4$ and such that the $\partial_i$ are each boundary
components of $Z$:\\
\Figure{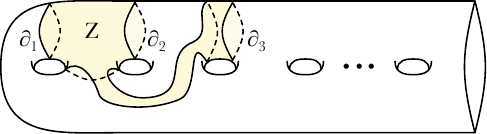}
As in the following figure, let $x$ and $y$ be simple closed curves on $Z$ that intersect twice
with opposite signs, so their algebraic intersection number is $0$:\\
\Figure{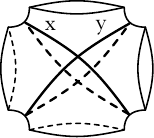}
Since the algebraic intersection number of $x$ and $y$ is $0$, their actions on $H$ commute
and $[T_x,T_y] \in \Torelli_g^1$.  We will call $[T_x,T_y] \in \Torelli_g^1$ a {\em simply intersecting
pair map} supported on $Z$.  These have the desired image under $\tau_1$:

\begin{lemma}
\label{lemma:simplyintersectingpair}
Let $Z$ be a subsurface of $\Sigma_g^1$ such that $Z \cong \Sigma_0^4$ and let $[T_x,T_y] \in \Torelli_g^1$ be 
a simply intersecting pair map supported on $Z$.  Let $\partial_1,\partial_2,\partial_3$ be three
boundary components of $Z$.  Orienting the $\partial_i$ arbitrarily, we then have
\[\tau_1([T_x, T_y]) = \pm [\partial_1] \wedge [\partial_2] \wedge [\partial_3] \in \wedge^3 H \subset H \otimes \wedge^2 H.\]
\end{lemma}
\begin{proof}
Since we have already made a number of calculations, we leave this one to the reader.  It can be done by carefully
studying the action of $[T_x,T_y]$ on $\pi$, or alternatively by factoring $[T_x,T_y]$ as a product of bounding
pair maps.  See \cite{ChildersSIP} for the details of this calculation.
\end{proof}

\begin{remark}
Let $Z$ and $\partial_1,\partial_2,\partial_3$ be as in Lemma \ref{lemma:simplyintersectingpair} and
let $\partial_4$ be the fourth boundary component of $Z$.  It might seem confusing that
$\partial_4$ plays no role in Lemma \ref{lemma:simplyintersectingpair}.  The reason for this
is as follows.  Orienting $\partial_4$ arbitrarily, since
$\partial_1 \cup \cdots \cup \partial_4$ bounds $Z$ there exists $\epsilon_1,\ldots,\epsilon_4 \in \{\pm 1\}$ such that
\[\epsilon_1 [\partial_1] + \epsilon_2 [\partial_2] + \epsilon_3 [\partial_3] + \epsilon_4 [\partial_4] = 0.\]
This implies that replacing one of $\partial_1,\partial_2,\partial_3$ with $\partial_4$ would not
not change $\pm [\partial_1] \wedge [\partial_2] \wedge [\partial_3]$.
\end{remark}

\subsection{Separating twists, second Johnson homomorphism}
\label{section:tau2septwist}

Lemma \ref{lemma:tau1septwist} implies that separating twists $T_z$ lie in
$\Torelli_g^1[2]$.  In fact, Johnson \cite{Johnson2} proved that separating twists
generate $\Torelli_g^1[2]$, which is the domain of the second Johnson homomorphism
$\tau_2\colon \Torelli_g^1[2] \rightarrow H \otimes \FLie_3(H)$.  Our final
goal in this section is to calculate 
\[\tau_2(T_z) \in H \otimes \FLie_3(H) \cong H \otimes \frac{H \otimes \wedge^2 H}{\wedge^3 H} \cong \frac{H^{\otimes 2} \otimes \wedge^2 H}{H \otimes \wedge^3 H}.\]
Here the second equality was explained in Example \ref{example:lie3}.  
We have the following:

\begin{lemma}
\label{lemma:tau2septwist}
Let $T_z$ be a separating twist in $\Torelli_g^1[2]$.
Let $S$ and $S'$ be the subsurfaces of $\Sigma_g^1$ on either side of $z$.  Order
them such that $\partial \Sigma_g^1 \subset S'$, so $S \cong \Sigma_h^1$ for some
$1 \leq h \leq g$.  Let $V$ be the image of $\HH_1(S;\Q)$ in $H$.  Then $\tau_2(T_z)$ is the image of
$-\omega_V \otimes \omega_V \in (\wedge^2 H) \otimes (\wedge^2 H)$ in
$(H^{\otimes 2} \otimes \wedge^2 H)/(H \otimes \wedge^3 H)$.
\end{lemma}
\begin{proof}
Let $\zeta \in \pi$ be the following curve:\\
\Figure{Tau1SepTwist.1}
By the proof of Lemma \ref{lemma:separatingcurvebrack}, we can 
write $\zeta = [\mu_1,\nu_1] \cdots [\mu_h,\nu_h]$ for simple closed curves
$\mu_1,\nu_1,\ldots,\mu_h,\nu_h \in \pi$ such that the following holds:
\begin{itemize}
\item Letting $u_i = [\mu_i] \in H$ and $v_i = [\nu_i] \in H$ be the homology
classes of $\mu_i,\nu_i \in \pi$, the vectors $\{u_1,v_1,\ldots,u_h,v_h\}$ form
a symplectic basis for $V = \HH_1(S;\Q)$.
\end{itemize}
As in Lemma \ref{lemma:separatingcurvebrack}, the image of $\zeta$ in $\FLie_2(H) \cong \wedge^2 H$ is
$\omega_V = u_1 \wedge v_1 + \cdots + u_h \wedge v_h$.  

Just like in the proof of Lemma \ref{lemma:tau2septwist}, we have $T_z(\mu_i) = \zeta \mu_i \zeta^{-1}$
and $T_z(\nu_i) = \zeta \nu_i \zeta^{-1}$ for $1 \leq i \leq h$.  This implies that
$(T_z)_{\ell}(\mu_i) = \zeta \mu_i \zeta^{-1} \mu_i^{-1} = [\mu_i,\zeta]^{-1}$ 
and $(T_z)_{\ell}(\nu_i) = \zeta \nu_i \zeta^{-1} \nu_i^{-1} = [\nu_i,\zeta]^{-1}$.
Considered as a map
\[H \rightarrow \FLie_3(H) = \frac{H \otimes \wedge^2 H}{\wedge^3 H},\]
we deduce that $\tau_2(T_z)$ vanishes on the image of $\HH_1(S')$ and
has the following behavior on $\HH_1(S)$: for $1 \leq i \leq h$, we have
\begin{align*}
\tau_2(T_z)(u_i) &= \Brack{[\mu_i,\zeta]^{-1}}_3 = -[u_i,\omega_V],\\
\tau_2(T_z)(v_i) &= \Brack{[\nu_i,\zeta]^{-1}}_3 = -[v_i,\omega_V].
\end{align*}
Translating this into an element of $H \otimes \Lie_3(H) = (H^{\otimes 2} \otimes \wedge^2 H)/(H \otimes \wedge^3 H)$,
this is exactly the image of
\[-\sum_{i=1}^h \left((u_i \otimes v_i) \otimes \omega_V - (v_i \otimes u_i) \otimes \omega_V\right) = -\sum_{i=1}^h (u_i \wedge v_i) \otimes \omega_V = -\omega_V \otimes \omega_V\]
in $(H^{\otimes 2} \otimes \wedge^2 H)/(H \otimes \wedge^3 H)$.
\end{proof}

\section{Image of the Johnson homomorphism}
\label{section:johnsonimage}

For some $g \geq 1$, let $H = \HH_1(\Sigma_g^1;\Q)$.
From the calculations in the previous section, it is immediate that
the Johnson homomorphism $\tau_d\colon \Torelli_g^1[d] \rightarrow H \otimes \FLie_{d+1}(H)$ is not surjective.
Our goal in this section is to prove the following theorem of Morita which explains this:

\begin{theorem}[{Morita, \cite[Corollary 3.2]{MoritaAbelian}}]
\label{theorem:moritajohnsonbracket}
For all $g,d \geq 1$ the image of $\tau_d\colon \Torelli_g^1[d] \rightarrow H \otimes \FLie_{d+1}(H)$
is contained in the kernel of the Lie algebra bracket map $H \otimes \FLie_{d+1}(H) \rightarrow \FLie_{d+2}(H)$.
\end{theorem}

Before proving this, we discuss the extent to which it characterizes the image of $\tau_d$:
\begin{itemize}
\item For $d = 1$, by Example \ref{example:lie2} we have
$\FLie_2(H) \cong \wedge^2 H$.  We observed in Example \ref{example:lie3} that the kernel
of the Lie bracket map from
$H \otimes \FLie_2(H) = H \otimes \wedge^2 H$
to $\Lie_3(H)$ is $\wedge^3 H$, so by Theorem \ref{theorem:moritajohnsonbracket} the $\Q$-span of the image
of $\tau_1$ is contained in $\wedge^3 H$.  Johnson \cite{JohnsonHomo} proved that
the $\Q$-span of the image of $\tau_1$ equals $\wedge^3 H$.  We
describe this calculation in \S \ref{section:firstjohnson} below.
\item For $d = 2$, by Example \ref{example:lie3} we have
\[\FLie_3(H) = \frac{H \otimes \wedge^2 H}{\wedge^3 H}.\]
We observed in Example \ref{example:lie4} that the kernel of the Lie bracket map from
\[H \otimes \FLie_3(H) = \frac{H^{\otimes 2} \otimes \wedge^2 H}{H \otimes \wedge^3 H}\]
to $\Lie_4(H)$ is isomorphic to $\Sym^2(\wedge^2 H) / \wedge^4 H$, so by Theorem \ref{theorem:moritajohnsonbracket} the 
$\Q$-span of the image  
of $\tau_2$ is contained in $\Sym^2(\wedge^2 H) / \wedge^4 H$.  Morita \cite{MoritaCasson} proved
that the $\Q$-span of the image of $\tau_2$ equals $\Sym^2(\wedge^2 H)/\wedge^4 H$.
We describe this calculation in \S \ref{section:secondjohnson} below.
\end{itemize}
Morita \cite{MoritaAbelian} proved that for $d \geq 3$ there are restrictions on the image
of $\tau_d$ beyond Theorem \ref{theorem:moritajohnsonbracket}.  There is now an enormous
literature on the image of the higher Johnson homomorphisms.  See \cite{HainJohnson, MoritaSurvey, MoritaSakasaiSuzuki, SatohJohnson}
for discussions of this literature.

\begin{proof}[Proof of Theorem \ref{theorem:moritajohnsonbracket}]
Fix a basepoint $\ast \in \partial \Sigma_g^1$ and set 
$\pi = \pi_1(\Sigma_g^1,\ast)$.  Let $\partial \in \pi$ be the loop around the boundary component,
oriented such that the surface is to its right.
Since $\Mod_g^1$ fixes $\partial \Sigma_g^1$ pointwise, the action of $\Mod_g^1$ on $\pi$ fixes the curve $\partial$.
As we will see, this will ultimately be responsible for this theorem.

As in the proof of Lemma \ref{lemma:separatingcurvebrack},
we can write $\partial = [\alpha_1,\beta_1] \cdots [\alpha_g,\beta_g]$ for simple closed curves
$\alpha_1,\beta_1,\ldots,\alpha_g,\beta_g \in \pi$.
Letting $a_i = [\alpha_i] \in H$ and $b_i = [\beta_i] \in H$
be the homology classes of these curves, the elements $\{a_1,b_1,\ldots,a_g,b_g\}$ 
form a symplectic basis for $H$.

Consider $f \in \Torelli_g^1[d]$.  For $1 \leq i \leq g$, write 
$f_{\ell}(\alpha_i) = \oalpha_i \in \gamma_{d+1}(\pi)$ and $f_{\ell}(\beta_i) = \obeta_i \in \gamma_{d+1}(\pi)$.  
Regarding $\tau_d(f)$ as an element of $\Hom(H,\FLie_{d+1}(H))$, for $1 \leq i \leq g$ we have
\begin{align*}
\tau_d(f)(a_i) &= \Brack{\oalpha_i}_{d+1}, \\
\tau_d(f)(b_i) &= \Brack{\obeta_i}_{d+1}.
\end{align*}
It follows that as an element of $H \otimes \FLie_{d+1}(H)$, we have
\[\tau_d(f) = \sum_{i=1}^g \left(a_i \otimes \Brack{\obeta_i}_{d+1}-b_i \otimes \Brack{\oalpha_i}_{d+1}\right).\]
From this, we see that the following element of $\gamma_{d+2}(\pi)$ maps
to the image of $\tau_d(f)$ under the bracket map $H \otimes \FLie_{d+1}(H) \rightarrow \FLie_{d+2}(H)$:
\[\zeta = \prod_{i=1}^g [\alpha_i,\obeta_i] [\oalpha_i,\beta_i]\]
To prove the theorem, we must prove that $\Brack{\zeta}_{d+2} = 0$.  Writing
$\equiv$ for equality modulo $\gamma_{d+3}(\pi)$, we must equivalently prove
that $\zeta \equiv 1$.

To see this, observe that
\begin{equation}
\label{eqn:faibi}
f([\alpha_i,\beta_i]) = [\oalpha_i \alpha_i, \obeta_i \beta_i].
\end{equation}
For a group $G$ and $x,y \in G$, our commutator convention is $[x,y] = x y x^{-1} y^{-1}$.
For $x,y,z \in G$,  we thus have
\[[xy,z] = x [y,z] x^{-1} [x,z] \quad \text{and} \quad [x,yz] = [x,y] y [x,z] y^{-1}.\]
Applying this to \eqref{eqn:faibi}, we see that
\begin{align*}
f([\alpha_i,\beta_i]) &= \oalpha_i [\alpha_i,\obeta_i \beta_i] \oalpha_i^{-1} [\oalpha_i,\obeta_i \beta_i] \\
&= \left(\oalpha_i [\alpha_i,\obeta_i] \oalpha_i^{-1}\right) \left(\oalpha_i \obeta_i [\alpha_i,\beta_i] \obeta_i^{-1} \oalpha_i^{-1}\right)
[\oalpha_i,\obeta_i] \left(\obeta_i [\oalpha_i,\beta_i] \obeta_i^{-1}\right). 
\end{align*}
Since $\oalpha_i,\obeta_i \in \gamma_{d+1}(\pi)$, the following hold:
\begin{itemize}
\item $\oalpha_i \obeta_i [\alpha_i,\beta_i] \obeta_i^{-1} \oalpha_i^{-1} \equiv [\alpha_i,\beta_i]$; and
\item $[\oalpha_i,\obeta_i] \equiv 1$; and
\item $[\alpha_i,\obeta_i]$ and $[\oalpha_i,\beta_i]$ commute with everything modulo $\gamma_{d+3}(\pi)$.
\end{itemize}
Applying these to the above identity, we deduce that
\[f([\alpha_i,\beta_i]) \equiv [\alpha_i,\beta_i] [\alpha_i,\obeta_i] [\oalpha_i,\beta_i].\]
Keeping in mind that $f(\partial) = \partial$ and
$\partial = [\alpha_1,\beta_1] \cdots [\alpha_g,\beta_g]$, we can multiply
all of these terms together and commute the terms $[\alpha_i,\obeta_i]$ and $[\oalpha_i,\beta_i]$
around and see that
\[\partial = f(\partial) \equiv \prod_{i=1}^g [\alpha_i,\beta_i] [\alpha_i,\obeta_i] [\oalpha_i,\beta_i] 
         \equiv \left(\prod_{i=1}^g [\alpha_i,\beta_i]\right) \left(\prod_{i=1}^g [\alpha_i,\obeta_i] [\oalpha_i,\beta_i]\right) 
         = \partial \Cdot \zeta.\]
Rearranging this, we see that $\zeta \equiv 1$, as desired.
\end{proof}

\section{Representation theory of \texorpdfstring{$\Sp_{2g}$}{Sp2g}}
\label{section:sprep}

Before describing the images of $\tau_1$ and $\tau_2$, we pause to
describe the representation theory of the algebraic group $\Sp_{2g}$
over a field $\bk$ of characteristic $0$.  As in our discussion of
$\SL_n$ in \S \ref{section:slrep}, 
everything we discuss can be found in \cite{BorelLinear}, and in the slightly different language
of Lie algebra representations can also be found in \cite{FultonHarris}.

\subsection{Symplectic basis}

Let $\{a_1,b_1,\ldots,a_g,b_g\}$ be the standard symplectic basis for $\bk^{2g}$.  As
is standard, we will describe elements of $\Sp_{2g}(\bk)$ using matrices whose
first $g$ columns give the images of $a_1,\ldots,a_g$ and whose last $g$ columns
gives the images of $b_1,\ldots,b_g$.  Letting
\[J = \left(\begin{matrix} 0 & I_n \\ -I_n & 0 \end{matrix}\right),\]
the group $\Sp_{2g}(\bk)$ thus consists of $2g \times 2g$ matrices $M$ such that
$M^t J M = J$.

\subsection{Self-duality and the symplectic form}
\label{section:selfduality}

Let $H = \bk^{2g}$ be the standard representation of $\Sp_{2g}(\bk)$ and let
$\omega$ be the symplectic form on $H$ that is preserved by $\Sp_{2g}(\bk)$.
The form $\omega$ is an alternating bilinear form $\omega\colon H \times H \rightarrow \bk$
that is nondegenerate in the sense that it identifies $H$ with its
dual $H^{\ast} = \Hom(H,\bk)$.  Since $\Sp_{2g}(\bk)$ preserves $\omega$, it
follows that $H$ and $H^{\ast}$ are isomorphic representations of $\Sp_{2g}(\bk)$.
It follows that
\[\omega \in (\wedge^2 H)^{\ast} \cong \wedge^2 H^{\ast} \cong \wedge^2 H.\]
We will henceforth identify $\omega$ with its image in $\wedge^2 H$.  The
element $\omega \in \wedge^2 H$ is invariant under $\Sp_{2g}(\bk)$, so it
spans a trivial subrepresentation of $\wedge^2 H$.  With
respect to the symplectic basis $\{a_1,b_1,\ldots,a_g,b_g\}$, we have
\[\omega = a_1 \wedge b_1 + \cdots + a_g \wedge b_g.\]

\subsection{Maximal split torus}

Let $\bT$ be the subgroup of $\Sp_{2g}$ consisting of diagonal matrices.  This
notation conflicts with our notation for the corresponding subgroup of $\SL_n$,
but context should make it clear which group we are talking about.  For
$t_1,\ldots,t_g \in \bk^{\times}$, set
\begin{equation}
\label{eqn:sptorus}
\diag_{\Sp}(t_1,\ldots,t_g) = \diag(t_1,\ldots,t_g,t_1^{-1},\ldots,t_g^{-1}) \in \bT(\bk).
\end{equation}
The group $\bT(\bk)$ is exactly the group of matrices
of the form $\diag_{\Sp}(t_1,\ldots,t_g)$, so
\[\bT(\bk) \cong (\bk^{\times})^g.\]
For $1 \leq i \leq g$, let $E_i \in \chi(\bT)$ be the character 
\[E_i(\diag_{\Sp}(t_1,\ldots,t_g)) = t_i.\]
The group of characters $\chi(\bT)$ is then isomorphic to $\Z^g$ and is generated
by $\{E_1,\ldots,E_g\}$.

\subsection{Weight decomposition}

Using $\bT$, we can talk about weight vectors for representations of $\Sp_{2g}$.
The inverses in \eqref{eqn:sptorus} make this slightly more complicated than for
$\SL_n$.  Here are some examples of weight vectors:

\begin{example}
In $H$, the vector $a_i$ is a weight vector with weight $E_i$ and the vector
$b_i$ is a weight vector with weight $-E_i$.
\end{example}

\begin{example}
In $\wedge^2 H$, for $1 \leq i < j \leq g$ the vector $a_i \wedge b_j$ is a weight
vector with weight $E_i - E_j$.  When $i=j$, things degenerate and
$a_i \wedge b_i$ is a weight vector with weight $0$, i.e., it is fixed by $\bT(\bk)$.
The sum of all of these is $\omega = a_1 \wedge b_1 + \cdots + a_g \wedge b_g$, which
is fixed not only by $\bT(\bk)$ but by all of $\Sp_{2g}(\bk)$.
\end{example}

\subsection{Standard unipotent subgroups}

For representations of $\SL_n$, we defined highest weight vectors to be weight vectors that
are fixed by the unipotent subgroup of upper triangular matrices.  The corresponding notion
of $\Sp_{2g}$ is more complicated since the subgroup corresponding
to ``upper triangular matrices'' is slightly more complicated.

The {\em standard embedding} of $\GL_g(\bk)$ into $\Sp_{2g}(\bk)$ is as follows:
\[M \in \GL_g(\bk) \mapsto \left(\begin{matrix} M & 0 \\ 0 & (M^t)^{-1} \end{matrix}\right) \in \Sp_{2g}(\bk).\]
The image of the standard embedding is exactly the subgroup of $\Sp_{2g}(\bk)$ consisting
of symplectic matrices preserving the subspaces $\Span{a_1,\ldots,a_g}$ and
$\Span{b_1,\ldots,b_g}$.  Define $\bU_1 < \Sp_{2g}$ to be the image of the group of strictly
upper triangular matrices in $\GL_g$ under the standard embedding.

Next, for a $g \times g$ matrix $M$ over $\bk$ with $M^t = M$, we have a corresponding element
of $\Sp_{2g}(\bk)$:
\[\left(\begin{matrix} I_g & M \\ 0 & I_g \end{matrix}\right) \in \Sp_{2g}(\Z).\]
Define $\bU_2$ to be the subgroup of $\Sp_{2g}$ consisting of such matrices.

\begin{remark}
The subgroup $\bU$ of $\Sp_{2g}$ generated by $\bU_1$ and $\bU_2$ is what plays
the role of strictly upper triangular matrices, but it is convenient to separate
out the roles played by $\bU_1$ and $\bU_2$.
\end{remark}

\subsection{Highest weight vectors}

If $V$ is a representation of $\Sp_{2g}$, then a {\em highest weight vector} in $V$ 
with weight $\chi \in \chi(\bT)$ is a nonzero vector $v \in V$ such that:
\begin{itemize}
\item $v$ is a weight vector with weight $\chi$, so $v \neq 0$ and $v \in V_{\chi}$; and
\item $v$ is fixed by all elements of $\bU_1(\bk)$ and $\bU_2(\bk)$.
\end{itemize}
A {\em highest weight} of $V$ is the weight of a highest weight vector in $V$.

The integer points $\bU_1(\Z)$ and $\bU_2(\Z)$ are Zariski dense in $\bU_1(\bk)$ and
$\bU_2(\bk)$, so to check that a weight vector $v \in V$ is a highest weight vector it
is enough to check that it is invariant under $\bU_1(\Z)$ and $\bU_2(\Z)$.  In
fact, it is enough to check this on generators for these groups, which are as
follows:
\begin{itemize}
\item For $1 \leq i < j \leq g$, the matrix $X_{ij} \in \bU_1(\bk)$ that fixes
all elements of $\{a_1,b_1,\ldots,a_g,b_g\}$ except for $a_j$ and $b_i$,
on which it does the following:
\begin{align*}
X_{ij}(a_j) &= a_j + a_i,\\
X_{ij}(b_i) &= b_i - b_j.
\end{align*}
\item For $1 \leq i \leq g$, the matrix $Y_i \in \bU_2(\bk)$ that fixes
all elements of $\{a_1,b_1,\ldots,a_g,b_g\}$ except for $b_i$,
on which it does the following:
\[Y_i(b_i) = b_i + a_i.\]
\item For $1 \leq i < j \leq g$, the matrix $Z_{ij} \in \bU_2(\bk)$ that fixes
all elements of $\{a_1,b_1,\ldots,a_g,b_g\}$ except for $b_i$ and $b_j$,
on which it does the following:
\begin{align*}
Z_{ij}(b_i) &= b_i + a_j,\\
Z_{ij}(b_j) &= b_j + a_i.
\end{align*}
\end{itemize}
We remark that these are examples of elementary symplectic matrices. 

\begin{example}
For $H$, the vector $a_1$ is a highest weight vector with weight $E_1$.
One can check that up to scaling this is the unique highest weight vector.
\end{example}

\begin{example}
\label{example:wedgeomega}
In $\wedge^2 H$, the vector $a_1 \wedge a_2$ is a highest weight vector with
weight $E_1 + E_2$.  However, it is not the only highest weight vector: the
vector
\[\omega = a_1 \wedge b_1 + \cdots + a_g \wedge b_g\]
is also a highest weight vector, this time with weight $0$.  
In fact, as we we noted above $\omega$ is fixed by
$\Sp_{2g}(\bk)$, not just by its subgroups $\bU_1(\bk)$ and $\bU_2(\bk)$.  It
is enlightening to prove this directly by examining the action of the elements
$X_{ij}$ and $Y_i$ and $Z_{ij}$ on $\omega$.
\end{example}

\subsection{Theorem of the highest weight}

Just like for $\SL_n$, the representation theory of $\Sp_{2g}$ is controlled by highest weight vectors.
Indeed, the theorem of the highest weight says that the following hold:
\begin{enumerate}
\item All representations of $\Sp_{2g}$ decompose as direct sums of irreducible representations.
\item Up to scaling, each irreducible representation $V$ of $\Sp_{2g}$ contains a unique highest weight vector.
\item If $V$ is an arbitrary representation of $\Sp_{2g}$ and $v \in V$ is a highest weight vector, then
the smallest subrepresentation containing $v$ is irreducible.
\item If $V$ and $V'$ are irreducible representations of $\Sp_{2g}$ with highest weight vectors $v \in V$ and $v' \in V'$
and if the weights of $v$ and $v'$ are the same, then $V$ and $V'$ are isomorphic representations.
\end{enumerate}

\subsection{Dominant weights}

To complete the picture of the representation theory of $\Sp_{2g}$, we must identify the weights
that can appear as highest weights of irreducible representations of $\Sp_{2g}$.  A {\em dominant weight}
is a character
\[\chi = k_1 E_1 + \cdots + k_g E_g \quad \text{with $k_1,\ldots,k_g \in \Z$}\]
such that $k_1 \geq \cdots \geq k_g \geq 0$.  Note that unlike for $\SL_n$, there are no
relations between the $E_i$, so we cannot ensure that the last coordinate is $0$.
The dominant weights are exactly the weights of irreducible representations of $\Sp_{2g}$.
If the $\chi$ above is a dominant weight, we will denote by
$\bV_{\chi}(g)$ the irreducible representation of $\Sp_{2g}$
with highest weight $\chi$.  

Recall that a {\em partition} of an integer $d$ of length $m \geq 0$ is
a tuple $\sigma = (k_1,\ldots,k_m)$ with $k_1 \geq \cdots \geq k_m \geq 1$
and $k_1 + \cdots k_m = d$.  The dominant weights for $\Sp_{2g}$ are in bijection
with partitions of integers with length at most $g$.  Given
such a partition $\sigma = (k_1,\ldots,k_m)$, the corresponding
dominant weight is
\[\chi = k_1 E_1 + \cdots + k_m E_m.\]
We will also write $\bV_{\sigma}(g)$ for $\bV_{\chi}(g)$.

\begin{convention}
We will denote multiplicities in partitions using superscripts.
For instance, if $\sigma = (5,4,4,4,1,1)$ then we will write
$\bV_{5,4^3,1^2}(g)$ for $\bV_{\sigma}(g)$.   
\end{convention}

Here are some examples of this notation.  Recall that $H \cong \bk^{2g}$ is the standard
representation of $\Sp_{2g}(\bk)$.

\begin{example}
We have $\bV_k(g) \cong \Sym^k(H)$.  The highest weight vector
in $\Sym^k(H)$ with weight $k E_1$ is $a_1 \Cdot \ldots \Cdot a_1$.
\end{example}

\begin{example}
\label{example:decomposegpwedge}
Unlike for $\SL_n$, in most cases $\wedge^k H$ is not irreducible.  Indeed,
for $1 \leq k \leq g$ we have
\[\wedge^k H = \begin{cases}
\bV_{1^k}(g) \oplus \bV_{1^{k-2}}(g) \oplus \cdots \oplus \bV_0(g) & \text{if $k$ is even},\\
\bV_{1^k}(g) \oplus \bV_{1^{k-2}}(g) \oplus \cdots \oplus \bV_1(g) & \text{if $k$ is odd}.
\end{cases}\]
The highest weight vectors corresponding to these decompositions are as follows.  As
in \S \ref{section:selfduality} let $\omega = a_1 \wedge b_1 + \cdots + a_g \wedge b_g$.
Then for all $d \geq 0$ such that $k-2d \geq 0$ the vector
\[a_1 \wedge a_2 \wedge \cdots \wedge a_{k-2d} \wedge \omega \wedge \cdots \wedge \omega \in \wedge^k H\]
with $d$ factors of $\omega$ is a highest weight vector with weight $E_1 + \cdots + E_{k-2d}$.  We will explain
how to justify these decompositions with a computer below.
\end{example}

\subsection{Stable decompositions}
\label{section:stablesp}

Just like for $\SL_n$, there is a notion of stability for representations of
$\Sp_{2g}(\bk)$.  
Let $\sigma = (k_1,\ldots,k_m)$ be a partition of $d$ with at most $g$ parts.  We will call $d$ the
{\em degree} of the partition.  We will also say that $d$ is the degree
of the irreducible representation $\bV_{\sigma}(g)$.
Schur--Weyl duality implies that the irreducible representation $\bV_{\sigma}(g)$ appears in in
$H^{\otimes d}$, and moreover that all irreducible representations that appear
in $H^{\otimes d}$ have degree at most $d$.

A classical observation is that for $g \geq d$, the decomposition of $H^{\otimes d}$
into irreducible factors is independent of the parameter $g$ in the following sense.
Since we will need to distinguish between $H = \bk^{2g}$ for different values of
$g$, we will write $H(g)$ for $H$.
\begin{itemize}
\item If
\[H(g)^{\otimes d} = \bV_{\sigma_1}(g) \oplus \cdots \bV_{\sigma_k}(g)\]
for partitions $\sigma_1,\ldots,\sigma_k$, then we also have
\[H(g+1)^{\otimes d} = \bV_{\sigma_1}(g+1) \oplus \cdots \bV_{\sigma_k}(g+1).\]
\end{itemize}

More generally, if $U$ is a representation of $\Sp_{2g}$ constructed from the standard
representation $H$ using tensor
powers, exterior powers, and symmetric powers, then $U$ naturally embeds
into $H(g)^{\oplus d}$ for a $d \geq 1$ called its degree, and the decomposition
of $U$ into irreducible factors is independent of $g$ as long as $g \geq d$.

\begin{convention}
In light of all of this, we can decompose representations of $\Sp_{2g}$ for $g \gg 0$
by using a computer to make this decomposition for $g$ at least the degree of
the representation.  This will be done silently throughout the remainder of the
paper.
\end{convention}

\section{Projection maps and representation theory}
\label{section:spprojection}

We now explore some more features of the representation theory of $\Sp_{2g}$ over a field $\bk$ of
characteristic $0$.  Let $H = \bk^{2g}$ be the standard representation, let
$\omega \in \wedge^2 H$ be the symplectic form, and let
$\{a_1,b_1,\ldots,a_g,b_g\}$ be the standard symplectic basis for $H$.

\subsection{Projection maps}

For $k \geq 0$, define a projection map $q_k\colon \wedge^{k+2} H \rightarrow \wedge^k H$ via the following
formula:
\[q_k(h_1 \wedge \cdots \wedge h_{k+2}) = \sum_{1 \leq i < j \leq k+2} (-1)^{i+j+1} \omega(h_i,h_k) h_1 \wedge \cdots \widehat{h_i} \wedge \cdots \wedge \widehat{h_j} \wedge \cdots \wedge h_{k+2}.\]
This formula makes sense since the right hand side changes sign when two of the $h_i$ are swapped.
We claim that $q_k$ is surjective.  To prove this, it is enough to prove that
its image contains all highest weight vectors in $\wedge^{k} H$.  As we saw in
Example \ref{example:decomposegpwedge}, these are exactly scalar multiples of the vectors
of the following form for $0 \leq d \leq k/2$:
\[w_{k,d} = a_1 \wedge \cdots \wedge a_{k-2d} \wedge \omega \wedge \cdots \wedge \omega \in \wedge^{k} H\]
In $w_{k,d}$, there are $d$ factors of $\omega$.  It is then an enlightening calculation
to show that $q_k(w_{k+2,d})$ is a nonzero multiple\footnote{We will write out the details
of this for $k=1$ in \S \ref{section:splith} below.  This will be the only case we need
in the remainder of the paper.} of $w_{k,d}$, so $w_{k,d}$ is
in the image of $q_k$.  The kernel of $q_k$ contains the highest weight vector
$a_1 \wedge \cdots \wedge a_{k+2}$ generating the irreducible factor $\bV_{1^{k+2}}(g)$ of
$\wedge^{k+2} H$.  Since this is the only irreducible factor of $\wedge^{k+2} H$ not accounted
for by the surjective map $q_k\colon \wedge^{k+2} H \rightarrow \wedge^{k} H$, this
must be the entire kernel.  In other words, we have a short exact sequence
of representations
\[0 \longrightarrow \bV_{1^{k+2}}(g) \longrightarrow \wedge^{k+2} H \stackrel{q_k}{\longrightarrow} \wedge^{k} H \longrightarrow 0.\]

\subsection{Splitting \texorpdfstring{$H$}{H}}
\label{section:splith}

We will be particularly interested in $q_1\colon \wedge^3 H \rightarrow H$.
Let $\iota\colon H \hookrightarrow \wedge^3 H$ be the embedding $\iota(h) = h \wedge \omega$.  
It is almost the case that $q_1$ splits $\iota$.  Indeed, consider a nonzero $u \in H$.  We can find
a symplectic basis $\{u_1,v_1,\ldots,u_g,v_g\}$ for $H$ such that $u = u_1$.  We then have
\begin{align*}
q_1(\iota(u)) &= q_1(u_1 \wedge (u_1 \wedge v_1 + \cdots + u_g \wedge v_g)) \\
&= q_1(u_1 \wedge u_2 \wedge v_2 + \cdots + u_1 \wedge u_g \wedge v_g) = (g-1) u.
\end{align*}
In other words, $q_1 \circ \iota\colon H \rightarrow H$ is multiplication by $(g-1)$.

\subsection{Splitting the quotient by \texorpdfstring{$H$}{H}}

Regard $H$ as a subspace of $\wedge^3 H$ via the inclusion $\iota$.  Let $p\colon \wedge^3 H \rightarrow (\wedge^3 H)/H$
be the projection, so we have a short exact sequence of $\Sp_{2g}$-representations
\[0 \longrightarrow H \stackrel{\iota}{\longrightarrow} \wedge^3 H \stackrel{p}{\longrightarrow} (\wedge^3 H)/H \longrightarrow 0.\]
This splits via the map $\sigma\colon (\wedge^3 H)/H \rightarrow \wedge^3 H$ taking
$\kappa \in (\wedge^3 H)/H$ to the unique $\sigma(\kappa) \in \wedge^3 H$ with
$p(\sigma(\kappa)) = \kappa$ and $q_1(\sigma(\kappa)) = 0$.  Since
$q_1 \circ \iota$ is multiplication by $(g-1)$, we can make this more explicit
as follows:
\begin{itemize}
\item Consider $\kappa \in (\wedge^3 H)/H$.  Let $\tkappa \in \wedge^3 H$ be any
element with $p(\tkappa) = \kappa$.  Set $h = q_1(\tkappa)$.  Then
\[\sigma(\kappa) = \tkappa - \frac{1}{g-1} h \wedge \omega.\]
Since $q_1(h \wedge \omega) = (g-1) h$, this is the unique element of
$\ker(q_1)$ projecting to $\kappa$.
\end{itemize}
For $\kappa \in \wedge^3 H$, let $\okappa = p(\kappa)$.  Below are two examples of $\sigma(\okappa)$
for different $\kappa \in \wedge^3 H$.  

\begin{example}
For $1 \leq i < j < k \leq g$, we have $\sigma(\overline{a_i \wedge a_j \wedge a_k}) = a_i \wedge a_j \wedge a_k$
since $q_1(a_i \wedge a_j \wedge a_k) = 0$.
\end{example}

\begin{example}
For distinct $1 \leq i,j \leq g$, we have
\[\sigma(\overline{a_i \wedge a_j \wedge b_j}) = a_i \wedge (a_j \wedge b_j - \frac{1}{g-1} \omega).\]
Indeed, since $q_1(a_i \wedge a_j \wedge b_j) = a_i \in H$, the above recipe shows that
\[\sigma(\overline{a_i \wedge a_j \wedge b_j}) = a_i \wedge a_j \wedge b_j - \frac{1}{g-1} a_i \wedge \omega = a_i \wedge (a_j \wedge b_j - \frac{1}{g-1} \omega).\qedhere\]
\end{example}

\subsection{Generating subrepresentations}

As we will soon see, for our proofs the most important representation of $\Sp_{2g}$ is $\wedge^2 \wedge^3 H$.
We will need to certify that subrepresentations of its quotient $\wedge^2 ((\wedge^3 H)/H)$ contain
specific irreducible factors.  Rather than try to prove a general result, we will prove two lemmas
that contain exactly what we need.\footnote{These results might seem unmotivated on a first reading,
but we promise the reader that they will prove to be crucial.}

\begin{lemma}
\label{lemma:certify1}
Let $g \geq 6$ and let $\{a_1,b_1,\ldots,a_g,b_g\}$ be the standard symplectic basis for
$H$.  Let $V$ be the subrepresentation of $\wedge^2 ((\wedge^3 H)/H)$ generated by
$\otheta=(\overline{a_1 \wedge a_2 \wedge a_3}) \wedge (\overline{a_4 \wedge a_5 \wedge b_5})$.  Then
$V$ contains a copy of $\bV_{1^4}(g)$.
\end{lemma}
\begin{proof}
Let $\sigma\colon (\wedge^3 H)/H \rightarrow \wedge^3 H$ be the section of
the projection $p\colon \wedge^3 H \rightarrow (\wedge^3 H)/H$ discussed above and
let $\phi\colon \wedge^2 \wedge^3 H \rightarrow \wedge^6 H$ be the following map:
\[\phi((u_1 \wedge v_1 \wedge w_1) \wedge (u_2 \wedge v_2 \wedge w_2)) = u_1 \wedge v_1 \wedge w_1 \wedge u_2 \wedge v_2 \wedge w_2.\]
Consider the composition
\[\begin{tikzcd}
\wedge^2 (\wedge^3 H)/H \arrow{r}{\wedge^2 \sigma} &
\wedge^2 \wedge^3 H \arrow{r}{\phi} &
\wedge^6 H \arrow{r}{q_4} &
\wedge^4 H.
\end{tikzcd}\]
Since $a_1 \wedge \cdots \wedge a_4 \in \wedge^4 H$ is a highest weight vector for
a copy of $\bV_{1^4}(g)$, it is enough to prove that this composition takes $\otheta$ to a nonzero multiple
of $a_1 \wedge \cdots \wedge a_4$.  In fact, to make our formulas clearer we will prove this
not for $\otheta$ but for $(1-g) \otheta$.  We calculate the image of $(1-g) \otheta$ under
the above composition as follows.  First, we calculate the image under $\wedge^2 \sigma$:
\begin{align*}
\wedge^2 \sigma((1-g)\otheta) &= (1-g) (a_1 \wedge \wedge a_2 \wedge a_3) \wedge (a_4 \wedge (a_5 \wedge b_5 - \frac{1}{g-1} \omega)) \\
                              &= (a_1 \wedge \wedge a_2 \wedge a_3) \wedge (a_4 \wedge ((1-g) a_5 \wedge b_5 + \omega)).
\end{align*}
Next, we apply $\phi$ and get
\begin{align*}
 & a_1 \wedge a_2 \wedge a_3 \wedge a_4 \wedge ((1-g) a_5 \wedge b_5 + \omega) \\
=& a_1 \wedge a_2 \wedge a_3 \wedge a_4 \wedge ((2-g) a_5 \wedge b_5 + a_6 \wedge b_6 + \cdots + a_g \wedge b_g).
\end{align*}
Finally, we apply $q_4$ and get
\[\left((2-g) + (g-5)\right) a_1 \wedge a_2 \wedge a_3 \wedge a_4
= -3 a_1 \wedge a_2 \wedge a_3 \wedge a_4.\qedhere\]
\end{proof}

\begin{lemma}
\label{lemma:certify2}
Let $g \geq 6$ and let $\{a_1,b_1,\ldots,a_g,b_g\}$ be the standard symplectic basis for
$H$.  Let $V$ be the subrepresentation of $\wedge^2 ((\wedge^3 H)/H)$ generated by
$\otheta=(\overline{a_1 \wedge a_4 \wedge b_4}) \wedge (\overline{a_2 \wedge a_3 \wedge b_3})$.  Then
$V$ contains a copy of $\bV_{1^2}(g)$.
\end{lemma}
\begin{proof}
Let $\sigma\colon (\wedge^3 H)/H \rightarrow \wedge^3 H$ be the section of
the projection $p\colon \wedge^3 H \rightarrow (\wedge^3 H)/H$ discussed above and
let $\phi\colon \wedge^2 \wedge^3 H \rightarrow \wedge^6 H$ be the following map:
\[\phi((u_1 \wedge v_1 \wedge w_1) \wedge (u_2 \wedge v_2 \wedge w_2)) = u_1 \wedge v_1 \wedge w_1 \wedge u_2 \wedge v_2 \wedge w_2.\]
Consider the composition
\[\begin{tikzcd}
\wedge^2 (\wedge^3 H)/H \arrow{r}{\wedge^2 \sigma} &
\wedge^2 \wedge^3 H \arrow{r}{\phi} &
\wedge^6 H \arrow{r}{q_4} &
\wedge^4 H \arrow{r}{q_2} &
\wedge^2 H.
\end{tikzcd}\]
Since $a_1 \wedge a_2 \in \wedge^2 H$ is a highest weight vector for
a copy of $\bV_{1^2}(g)$, it is enough to prove that this composition takes $\otheta$ to a nonzero multiple
of $a_1 \wedge a_2$.  In fact, to make our formulas clearer we will prove this
not for $\otheta$ but for $(1-g)^2 \otheta$.  We calculate the image of $(1-g)^2 \otheta$ under
the above composition as follows.  First, we calculate the image under $\wedge^2 \sigma$:
\begin{align*}
\wedge^2 \sigma((1-g)^2\otheta) &= (1-g)^2 (a_1 \wedge (a_4 \wedge b_4 - \frac{1}{g-1} \omega))
                                    \wedge (a_2 \wedge (a_3 \wedge b_3 - \frac{1}{g-1} \omega)) \\
                              &= (a_1 \wedge ((1-g) a_4 \wedge b_4 + \omega)) \wedge (a_2 \wedge ((1-g) a_3 \wedge b_3 + \omega)).
\end{align*}
Next, we apply $\phi$ and get
\begin{align*}
 & a_1 \wedge ((1-g) a_4 \wedge b_4 + \omega) \wedge a_2 \wedge ((1-g) a_3 \wedge b_3 + \omega) \\
=& a_1 \wedge a_2 \wedge (a_3 \wedge b_3 + (2-g) a_4 \wedge b_4 + a_5 \wedge b_5 + \cdots + a_g \wedge b_g) \\
 &\quad \quad \wedge ((2-g) a_3 \wedge b_3 + a_4 \wedge b_4 + a_5 \wedge b_5 + \cdots + a_g \wedge b_g).
\end{align*}
Finally, we apply $q_2 \circ q_4 \colon \wedge^6 H \rightarrow \wedge^2 H$.  It is convenient
to break this up into the sum of four terms:
\begin{align*}
&q_2 \circ q_4(a_1 \wedge a_2 \wedge (a_3 \wedge b_3 + \cdots + a_g \wedge b_g) \wedge (a_3 \wedge b_3 + \cdots + a_g \wedge b_g)) \\
&\quad = 2(g-2)(g-3) a_1 \wedge a_2,
\end{align*}
and
\begin{align*}
&q_2 \circ q_4(a_1 \wedge a_2 \wedge ((1-g)a_4 \wedge b_4) \wedge (a_5 \wedge b_5 + \cdots + a_g \wedge b_g)) \\
&\quad = 2(1-g)(g-4) a_1 \wedge a_2,
\end{align*}
and
\begin{align*}
&q_2 \circ q_4(a_1 \wedge a_2 \wedge (a_5 \wedge b_5 + \cdots + a_g \wedge b_g) \wedge ((1-g) a_3 \wedge b_3)) \\
&\quad = 2(1-g)(g-4) a_1 \wedge a_2,
\end{align*}
and
\begin{align*}
&q_2 \circ q_4(a_1 \wedge a_2 \wedge ((1-g)(a_4 \wedge b_4)) \wedge ((1-g)(a_3 \wedge b_3))) \\
&\quad = 2(1-g)(1-g) a_1 \wedge a_2.
\end{align*}
Adding these all up, we get
\[\left(2(g-2)(g-3) + 4(1-g)(g-4) + 2(1-g)(1-g)\right) a_1 \wedge a_2 = (6g-2) a_1 \wedge a_2.\qedhere\]
\end{proof}

\section{First Johnson homomorphism} 
\label{section:firstjohnson}

Fix some $g \geq 3$ and let $H = \HH_1(\Sigma_g^1;\Q)$.
Recall from Example \ref{example:lie2} that $\FLie_2(H) \cong \wedge^2 H \cong \bV_0(g) \oplus \bV_{1^2}(g)$.
We observed in Example \ref{example:lie3} that the kernel
of the Lie bracket map from
\[H \otimes \FLie_2(H) = H \otimes \wedge^2 H \cong \bV_{1^3}(g) \oplus \bV_1(g)^{\oplus 2} \oplus \bV_{2,1}(g)\]
to $\Lie_3(H)$ is 
\[\wedge^3 H \cong \bV_{1^3}(g) \oplus \bV_1(g).\]
Theorem \ref{theorem:moritajohnsonbracket} 
therefore implies that the $\Q$-span of the image
of $\tau_1\colon \Torelli_g^1 \rightarrow H \otimes \FLie_2(H)$ 
is contained in $\wedge^3 H$.  Johnson \cite{JohnsonHomo} proved that
the $\Q$-span of the image of $\tau_1$ equals $\wedge^3 H$.
This section proves this and uses it to describe $\HH_1(\Torelli_g^1;\Q)$.

In fact, Johnson proved something stronger:

\begin{theorem}[{Johnson, \cite{JohnsonHomo}}]
\label{theorem:imagetau1}
For some $g \geq 3$, let $H_{\Z} = \HH_1(\Sigma_g^1;\Z) \subset H$.  Then
the image of $\tau_1\colon \Torelli_g^1 \rightarrow H \otimes \wedge^2 H$
is $\wedge^3 H_{\Z}$.  In particular, the $\Q$-span of the image of
$\tau_1$ equals $\wedge^3 H$.
\end{theorem}
\begin{proof}
Since it will illustrate our representation-theoretic tools,
we will prove that the $\Q$-span of the image of $\tau_1$ equals $\wedge^3 H$
and leave the proof that the image of $\tau_1$ is $\wedge^3 H_{\Z}$ as an
exercise to the reader.\footnote{Here is a hint.  Letting $S = \{a_1,b_1,\ldots,a_g,b_g\}$ be the standard
symplectic basis for $H_{\Z}$, we have that $\wedge^3 H_{\Z}$ is spanned by
$\hS = \Set{$x \wedge y \wedge z$}{$x,y,z \in S$ distinct}$.  Using the formulas
for $\tau_1$ from Lemmas \ref{lemma:taubp} and \ref{lemma:simplyintersectingpair}, one can
show that each $x \wedge y \wedge z \in \hS$ is up to signs equal to $\tau_1(\phi)$
for $\phi$ either a bounding pair map or a simply intersecting pair map.}  
As in the following figure, let $\{a_1,b_1,\ldots,a_g,b_g\}$ be the standard symplectic basis for $H$:\\
\Figure{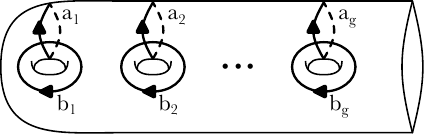}
As we noted above, Theorem \ref{theorem:moritajohnsonbracket}
implies that the $\Q$-span $J$ of the image of $\tau_1$ is contained in $\wedge^3 H$.  We must prove that 
$J=\wedge^3 H$.  The subspace $J$ is an $\Sp_{2g}(\Q)$-subrepresentation of $\wedge^3 H$ (Corollary \ref{corollary:imagejohnsonalgebraic}).
We have $\wedge^3 H = \bV_1(g) \oplus \bV_{1^3}(g)$.  We must prove that $J$ contains both of these
irreducible factors.  Let $\omega \in \wedge^2 H$ be the intersection pairing.

\begin{claim}{1}
The subspace $J$ contains the subrepresentation $\bV_1(g) \cong H$ of $\wedge^3 H$.
\end{claim}

This subrepresentation has the highest weight vector $a_1 \wedge \omega$, so we must prove that
$a_1 \wedge \omega \in J$.  For this, let $T_x T_y^{-1}$ be the following bounding pair map:\\
\Figure{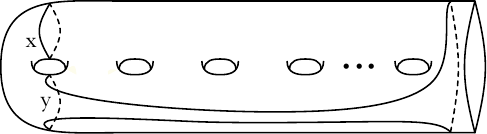}
By Lemma \ref{lemma:taubp}, we have
\begin{align*}
\tau_1(T_x T_y^{-1}) &= a_1 \wedge (a_2 \wedge b_2 + \cdots + a_g \wedge b_g) \\
                     &= a_1 \wedge (a_1 \wedge b_1 + a_2 \wedge b_2 + \cdots + a_g \wedge b_g) = a_1 \wedge \omega.
\end{align*}

\begin{claim}{2}
The subspace $J$ contains the subrepresentation $\bV_{1^3}(g)$ of $\wedge^3 H$.
\end{claim}

This subrepresentation has the highest weight vector $a_1 \wedge a_2 \wedge a_3$, so we must prove that
$a_1 \wedge a_2 \wedge a_3 \in J$.  For this, let $f$ be a simply intersecting
pair map supported on the following subsurface $Z \cong \Sigma_0^4$:\\
\Figure{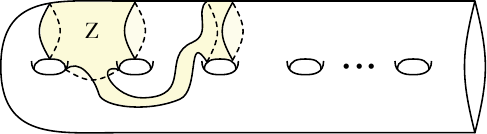}
By Lemma \ref{lemma:simplyintersectingpair}, we have $\tau_1(f) = \pm a_1 \wedge a_2 \wedge a_3$.
\end{proof}

Theorem \ref{theorem:imagetau1} implies that $\tau_1$ induces a surjection
$(\tau_1)_{\ast}\colon \HH_1(\Torelli_g^1;\Q) \rightarrow \wedge^3 H$.  Johnson \cite{Johnson2, Johnson3}
proved that this is an isomorphism:

\begin{theorem}[{Johnson, \cite{Johnson2, Johnson3}}]
\label{theorem:torelliabel}
For $g \geq 3$, the first Johnson homomorphism induces an isomorphism
$\HH_1(\Torelli_g^1;\Q) \cong \wedge^3 H$.
\end{theorem}

\begin{remark}
Johnson proves this by first proving that $\ker(\tau_1) = \Torelli_g^1[2]$ is generated by separating twists.
This actually holds for all $g$.  For $g \geq 3$, he then proves that squares of separating twists lie
in the commutator subgroup of $\Torelli_g^1$, which implies Theorem \ref{theorem:torelliabel}.  See
\cite{PutmanJohnson} for another exposition of this calculation.
\end{remark}

\section{Second Johnson homomorphism}
\label{section:secondjohnson}

Fix some $g \geq 4$ and let $H = \HH_1(\Sigma_g^1;\Q)$.
Recall from Example \ref{example:lie3} that 
\[\FLie_3(H) \cong \frac{H \otimes \wedge^2 H}{\wedge^3 H} \cong \bV_1(g) \oplus \bV_{2,1}(g).\]
We observed in Example \ref{example:lie4} that the kernel
of the Lie bracket map from
\[H \otimes \FLie_3(H) = \frac{H^{\otimes 2} \otimes \wedge^2 H}{H \otimes \wedge^3 H} \cong \bV_0(g) \oplus \bV_{1^2}(g)^{\oplus 2} \oplus \bV_{2^2}(g) \oplus \bV_{2,1^2}(g) \oplus \bV_2(g)^{\oplus 2} \oplus \bV_{3,1}(g)\]
to $\Lie_4(H)$ is
\[\frac{\Sym^2(\wedge^2 H)}{\wedge^4 H}  \cong \bV_{0}(g) \oplus \bV_{1^2}(g) \oplus \bV_{2^2}(g).\]
Theorem \ref{theorem:moritajohnsonbracket}
therefore implies that the $\Q$-span of the image
of $\tau_2\colon \Torelli_g^1[2] \rightarrow H \otimes \FLie_3(H)$
is contained in $\Sym^2(\wedge^2 H)/\wedge^4 H$.  Morita \cite{MoritaCasson} proved that
the $\Q$-span of the image of $\tau_2$ equals $\Sym^2(\wedge^2 H)/\wedge^4 H$.
This section proves this.

\begin{remark}
Unlike for the first Johnson homomorphism, it is not true that the image of $\tau_2$ equals the integer
points of $\Sym^2(\wedge^2 H)/\wedge^4 H$.  Morita actually proved something more precise than
what we will prove, and Yokomizo \cite{Yokomizo} completed this to give a complete description of
the image of $\tau_2$. 
\end{remark}

Morita's theorem is as follows:\footnote{The restriction to $g \geq 4$ is to ensure that the decomposition
of $\Sym^2(\wedge^2 H)$ is stable.}

\begin{theorem}[{Morita, \cite[Proposition 1.2]{MoritaCasson}}]
\label{theorem:imagetau2}
For $g \geq 4$, the $\Q$-span of the image of
$\tau_2\colon \Torelli_g^1[2] \rightarrow H \otimes \frac{H^{\otimes 2} \otimes \wedge^2 H}{H \otimes \wedge^3 H}$
equals $\Sym^2(\wedge^2 H)/\wedge^4 H \cong \bV_{0}(g) \oplus \bV_{1^2}(g) \oplus \bV_{2^2}(g)$.
\end{theorem}
\begin{proof}
As in the following figure, let $\{a_1,b_1,\ldots,a_g,b_g\}$ be the standard symplectic basis for $H$:\\
\Figure{StandardSpBasis}
As we noted above, Theorem \ref{theorem:moritajohnsonbracket}
implies that the $\Q$-span $J$ of the image of $\tau_2$ is contained in $\Sym^2(\wedge^2 H)/\wedge^4 H$.  
We must prove that
$J=\Sym^2(\wedge^2 H)/\wedge^4 H$.  
The subspace $J$ is an $\Sp_{2g}(\Q)$-subrepresentation of $\Sym^2(\wedge^2 H)/\wedge^4 H$ 
(Corollary \ref{corollary:imagejohnsonalgebraic}), so we must study this representation.
Before doing so, we identify some elements of $J$.

\begin{claim}{1}
\label{claim:imagetau2.1}
Let $\{u_1,v_1,\ldots,u_h,v_h\}$ be a symplectic basis for a symplectic subspace
of $H$.  Then the element
\[(u_1 \wedge v_1 + \cdots + u_h \wedge v_h) \Cdot (u_1 \wedge v_1 + \cdots + u_h \wedge v_h) \in \Sym^2(\wedge^2 H)\]
maps to an element of $J$.
\end{claim}
\begin{proof}[Proof of claim]
The group $\Sp_{2g}(\Q)$ acts on $J$.  We can find $\phi \in \Sp_{2g}(\Q)$ such that
$\phi(u_i) = a_i$ and $\phi(v_i) = b_i$ for $1 \leq i \leq h$.  Applying $\phi$, we reduce
ourselves to showing that
\[(a_1 \wedge b_1 + \cdots + a_h \wedge b_h) \Cdot (a_1 \wedge b_1 + \cdots + a_h \wedge b_h) \in \Sym^2(\wedge^2 H)\]
maps to an element of $J$.  For this, let $z$ be the following separating curve:\\
\Figure{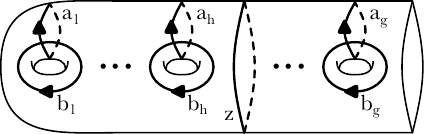}
By Lemma \ref{lemma:tau2septwist}, the element $\tau_2(T_z) \in J$ is the image
of the desired element of $\Sym^2(\wedge^2 H)$.
\end{proof}

Let $\tJ$ be the $\Sp_{2g}(\Q)$-subrepresentation of $\Sym^2(\wedge^2 H)$ spanned by the elements
identified in Claim \ref{claim:imagetau2.1}.  By that claim, it is enough to prove
that $\tJ$ maps surjectively to $\Sym^2(\wedge^2 H)/\wedge^4 H$.
We have
\begin{align*}
\Sym^2(\wedge^2 H) &= \bV_{1^4}(g) \oplus \bV_{1^2}(g)^{\oplus 2} \oplus \bV_0(g)^{\oplus 2} \oplus \bV_{2^2}(g),\\
\wedge^4 H         &= \bV_{1^4}(g) \oplus \bV_{1^2}(g) \oplus \bV_0(g).
\end{align*}
The next step is to tease apart the copies of $\bV_0(g)$ and $\bV_{1^4}(g)$ in $\Sym^2(\wedge^2 H)$ coming
from $\wedge^4 H$ and the copies that we will find in $\tJ$.
The embedding $\iota\colon \wedge^4 H \rightarrow \Sym^2(\wedge^2 H)$ takes the form
\[h_1 \wedge h_2 \wedge h_3 \wedge h_4 \in \wedge^4 H \mapsto (h_1 \wedge h_2) \Cdot (h_3 \wedge h_4) - (h_1 \wedge h_3) \Cdot (h_2 \wedge h_4) + (h_1 \wedge h_4) \Cdot (h_2 \wedge h_3).\]
Letting $\omega = a_1 \wedge b_1 + \cdots + a_g \wedge b_g$ be the symplectic form, the 
following are highest weight vectors for the three irreducible factors of $\wedge^4 H$:
\[v_0 = a_1 \wedge a_2 \wedge a_3 \wedge a_4 \quad \text{and} \quad
v_1 = a_1 \wedge a_2 \wedge \omega \quad \text{and} \quad
v_2 = \omega \wedge \omega.\] 
The vector $\iota(v_0)$ must be a highest weight vector for the unique copy of $\bV_{1^4}(g)$ in
$\Sym^2(\wedge^2 H)$.  As for the other terms, in $\Sym^2(\wedge^2 H)$ we have highest weight vectors
\[w_1 = (a_1 \wedge a_2) \Cdot \omega \quad \text{and} \quad
w_2 = \omega \Cdot \omega \quad \text{and} \quad
w_3 = (a_1 \wedge a_2) \Cdot (a_1 \wedge a_2).\]
Since $\iota(v_1)$ is not a multiple of $w_1$ and $\iota(v_2)$ is not a multiple of $w_2$, we see that:
\begin{itemize}
\item $\iota(v_0)$ is a highest weight vector for the copy of $\bV_{1^4}(g)$ in $\Sym^2(\wedge^2 H)$.
\item $\iota(v_1)$ and $w_1$ are highest weight vectors for the two copies of $\bV_{1^2}(g)$ in $\Sym^2(\wedge^2 H)$.
\item $\iota(v_2)$ and $w_2$ are highest weight vectors for the two copies of $\bV_0(g)$ in $\Sym^2(\wedge^2 H)$.
\item $w_3$ is a highest weight vector for the copy of $\bV_{2^2}(g)$ in $\Sym^2(\wedge^2 H)$.
\end{itemize}
To show that $\tJ$ maps surjectively to $\Sym^2(\wedge^2 H)/\wedge^4 H$, it 
is therefore enough to prove the following.  Note that we prove the three
claims in order of difficulty: first $w_2$, then $w_3$, and then finally $w_1$.

\begin{claim}{2}
The highest weight vector $w_2 = \omega \Cdot \omega \in \Sym^2(\wedge^2 H)$ lies in $\tJ$, so
the image of $\tJ$ in $\Sym^2(\wedge^2 H)/\wedge^4 H$ contains a copy of $\bV_{0}(g)$.
\end{claim}

Since $\omega = a_1 \wedge b_1 + \cdots + a_g \wedge b_g$, we have $\omega \Cdot \omega \in \tJ$
by definition (cf.\ Claim \ref{claim:imagetau2.1}).

\begin{claim}{3}
The highest weight vector $w_3 = (a_1 \wedge a_2) \Cdot (a_1 \wedge a_2) \in \Sym^2(\wedge^2 H)$ lies in $\tJ$, so
the image of $\tJ$ in $\Sym^2(\wedge^2 H)/\wedge^4 H$ contains a copy of $\bV_{2^2}(g)$.
\end{claim}

By definition, the following elements all lie in $\tJ$:
\begin{itemize}
\item $(a_1 \wedge b_1) \Cdot (a_1 \wedge b_1)$; and
\item $(a_1 \wedge (b_1+a_2)) \Cdot (a_1 \wedge (b_1+a_2))$; and
\item $(a_1 \wedge (b_1-a_2)) \Cdot (a_1 \wedge (b_2-a_2))$.
\end{itemize}
Since
\begin{align*}
 &(a_1 \wedge (b_1+a_2)) \Cdot (a_1 \wedge (b_1+a_2)) + (a_1 \wedge (b_1-a_2)) \Cdot (a_1 \wedge (b_2-a_2)) \\
=&2(a_1 \wedge b_1) \Cdot (a_1 \wedge b_1) + 2(a_1 \wedge a_2) \Cdot (a_1 \wedge a_2),
\end{align*}
it follows that $(a_1 \wedge a_2) \Cdot (a_1 \wedge a_2) \in \tJ$.

\begin{claim}{4}
The highest weight vector $w_1 = (a_1 \wedge a_2) \Cdot \omega \in \Sym^2(\wedge^2 H)$ lies in $\tJ$, so
the image of $\tJ$ in $\Sym^2(\wedge^2 H)/\wedge^4 H$ contains a copy of $\bV_{1^2}(g)$.
\end{claim}

Observe first that
\[(a_1 \wedge (b_1+a_2)) \Cdot \omega = (a_1 \wedge b_1) \Cdot \omega + (a_1 \wedge a_2) \Cdot \omega.\]
From this, we see that it is enough to prove that both $(a_1 \wedge (b_1+a_2)) \Cdot \omega$
and $(a_1 \wedge b_1) \Cdot \omega$ lie in $\tJ$.  Since $\Sp_{2g}(\Q)$ fixes $\omega$, these two elements differ by an element of 
$\Sp_{2g}(\Q)$.  It is therefore enough to prove that one of them lies in $\tJ$.
We will prove that $(a_1 \wedge b_1) \Cdot \omega$ lies in $\tJ$.

We expand this out:
\[(a_1 \wedge b_1) \Cdot \omega = (a_1 \wedge b_1) \Cdot (a_1 \wedge b_1) + (a_1 \wedge b_1) \Cdot (a_2 \wedge b_2) + \cdots + (a_1 \wedge b_1) \Cdot (a_g \wedge b_g).\]
Since $(a_1 \wedge b_1) \Cdot (a_1 \wedge b_1)$ lies in $\tJ$ by definition, we see that it is enough to prove
that $(a_1 \wedge b_1) \Cdot (a_i \wedge b_i)$ lies in $\tJ$ for all $2 \leq i \leq g$.  Since these
elements all differ by an element of $\Sp_{2g}(\Q)$, it is enough to prove that one of them lies in $\tJ$.
We will prove that $(a_1 \wedge b_1) \Cdot (a_2 \wedge b_2)$ lies in $\tJ$.

By definition, the following three elements all lie in $\tJ$:
\begin{itemize}
\item $(a_1 \wedge b_1) \Cdot (a_1 \wedge b_1)$; and
\item $(a_2 \wedge b_2) \Cdot (a_2 \wedge b_2)$; and
\item $(a_1 \wedge b_1 + a_2 \wedge b_2) \Cdot (a_1 \wedge b_1 + a_2 \wedge b_2)$.
\end{itemize}
Since
\begin{align*}
&(a_1 \wedge b_1 + a_2 \wedge b_2) \Cdot (a_1 \wedge b_1 + a_2 \wedge b_2) \\
=&(a_1 \wedge b_1) \Cdot (a_1 \wedge b_1) + 2 (a_1 \wedge b_1) \Cdot (a_2 \wedge b_2) + (a_2 \wedge b_2) \Cdot (a_2 \wedge b_2),
\end{align*}
we conclude that $(a_1 \wedge b_1) \Cdot (a_2 \wedge b_2)$ lies in $\tJ$, as desired.
\end{proof}

\section{Ruling out trivial factors}
\label{section:h2notrivial}

Recall that our goal in this part of the paper is to prove Theorem \ref{maintheorem:hain}.
For $\Torelli_g^1$, this theorem calculates the image of the cup
product pairing $\fc\colon \wedge^2 \HH^1(\Torelli_g^1;\Q) \rightarrow \HH^2(\Torelli_g^1;\Q)$.
As a prelude to this calculation, this section proves that
$\HH^2(\Torelli_g^1;\Q)$ has no trivial subrepresentations.

\subsection{Second cohomology group of Torelli group and mapping class group}

Though it is not absolutely necessary for our proof, to avoid technicalities
we start by recalling the following theorem which was discussed in the
introduction and which we are assuming throughout this paper:

\newtheorem*{maintheorem:h2finite}{Theorem \ref{maintheorem:h2finite} (Minahan--Putman, \cite[Theorem B]{MinahanPutmanH2Torelli})}
\begin{maintheorem:h2finite}
The homology group $\HH^2(\Torelli_g^1;\Q)$ is finite-dimensional
for $g \geq 5$ and an algebraic representation of $\Sp_{2g}(\Z)$ for $g \geq 6$.
\end{maintheorem:h2finite}

\begin{remark}
As we noted in the introduction, \cite{MinahanPutmanH2Torelli} also proves similar results for $\Torelli_g$
and $\Torelli_{g,1}$.  In fact, it proves a result for $\Torelli_{g,p}^b$ in general, though it requires
care to properly define the Torelli group on a surface with multiple punctures and boundary components.
\end{remark}

For the whole mapping class group, we have the following:

\begin{theorem}[{Harer, \cite{HarerH2}}]
\label{theorem:harerh2}
For $g \geq 3$, we have $\HH^2(\Mod_g^1;\Q) = \Q$.
\end{theorem}

\begin{remark}
After many further developments, the ultimate result in this direction was
proved by Madsen--Weiss \cite{MadsenWeiss}, who proved that the cohomology
ring $\HH^{\bullet}(\Mod_g^1;\Q)$
is isomorphic to a polynomial ring $\Q[\kappa_1,\kappa_2,\ldots]$ with
$|\kappa_i| = 2i$ in a range of degrees that tends to infinity
as $g \mapsto \infty$.
\end{remark}

\subsection{Borel stability theorem}

We will also need the following special case of the classical Borel stability theorem \cite{BorelStability1, BorelStability2} on
the cohomology of arithmetic groups.  The explicit stable range in the following
is due to Tshishiku \cite{TshishikuBorel}:

\begin{theorem}[{Borel, \cite{BorelStability1, BorelStability2}}]
\label{theorem:borelstability}
For $g \geq 2$, the following hold:
\begin{itemize}
\item[(i)] If $\bV_{\sigma}(g)$ is a nontrivial irreducible representation of $\Sp_{2g}$,
then $\HH^k(\Sp_{2g}(\Z);\bV_{\sigma}(g)) = 0$ for $k \leq g-1$.
\item[(ii)] In degrees $k \leq g-1$, the cohomology ring $\HH^{\bullet}(\Sp_{2g}(\Z);\Q)$ is
isomorphic to a polynomial ring $\Q[c_2,c_6,c_{10},\ldots]$ with $\deg(c_{4i-2}) = 4i-2$ for
$i \geq 1$.
\end{itemize}
\end{theorem}

\subsection{No trivial factors}

We can now rule out trivial factors of $\HH^2(\Torelli_g^1;\Q)$:

\begin{lemma}
\label{lemma:h2torellitrivial}
For $g \geq 6$, the $\Sp_{2g}(\Z)$-representation $\HH^2(\Torelli_g^1;\Q)$ contains no trivial factors.
\end{lemma}
\begin{proof}
The restriction $g \geq 6$ implies that $\HH^2(\Torelli_g^1;\Q)$ is a finite-dimensional
algebraic representation of $\Sp_{2g}(\Z)$ (Theorem \ref{maintheorem:h2finite}).  It
therefore decomposes as a direct sum of irreducible factors, so
the statement of the lemma makes sense.  Consider the Hochschild--Serre
spectral sequence with coefficients in $\bV_{0}(g) = \Q$ associated
to the short exact sequence
\[1 \longrightarrow \Torelli_g^1 \longrightarrow \Mod_g^1 \longrightarrow \Sp_{2g}(\Z) \longrightarrow 1.\]
This spectral sequence takes the form
\[E_2^{pq} = \HH^p(\Sp_{2g}(\Z);\HH^q(\Torelli_g^1;\Q)) \Rightarrow \HH^{p+q}(\Mod_g^1;\Q).\]
To understand $\HH^2$, we must understand $E_2^{pq}$ for $p + q \leq 2$ as well as
$E_2^{21}$ and $E_2^{30}$.  For $q=0$, by Theorem \ref{theorem:borelstability} we have
\[E_2^{p0} = \HH^p(\Sp_{2g}(\Z);\Q) = \begin{cases}
\Q & \text{if $p=0,2$},\\
0  & \text{if $p=1,3$}.
\end{cases}\]
Next, let $H = \HH_1(\Sigma_g^1;\Q)$.  Theorem \ref{theorem:torelliabel} says
that $\HH^1(\Torelli_g^1;\Q) \cong \wedge^3 H$, which has no trivial factors.
Theorem \ref{theorem:borelstability} therefore implies that
\[E_2^{p1} = \HH^p(\Sp_{2g}(\Z);\wedge^3 H) = 0 \quad \text{for $p \leq g-1$}.\]
Finally,
\[E_2^{02} = \HH^0(\Sp_{2g}(\Z);\Hom(\HH^2(\Torelli_g^1;\Q),\Q)) \cong \Hom_{\Sp_{2g}}(\HH^2(\Torelli_g^1;\Q),\bV_{0}(g)).\]
It follows that the dimension of $E_2^{02}$ equals the number of copies of
$\bV_{0}(g)$ in $\HH^2(\Torelli_g^1;\Q)$.
Summarizing, since $g \geq 6$ the $E_2$-page of our spectral sequence takes the form
\begin{center}
\begin{tblr}{|cccc}
$E_2^{02}$  &     &      & \\
$0$         & $0$ & $0$  & \\
$\Q$        & $0$ & $\Q$ & $0$\\
\hline
\end{tblr}
\end{center}
It follows that $\HH^{2}(\Mod_g^1;\Q) \cong E_2^{02} \oplus \Q$.
Theorem \ref{theorem:harerh2} says that this equals $\Q$, so we conclude
that $E_2^{02} = 0$, i.e., that there are no copies of
$\bV_{0}(g)$ in $\HH^2(\Torelli_g^1;\Q)$.
\end{proof}

\section{Cup products and the Johnson homomorphism}
\label{section:cupjohnson}

In this section, we relate the image of the cup product pairing to the Johnson homomorphism.
For $H_{\Z} = \HH_1(\Sigma_g^1;\Z)$, Theorem \ref{theorem:imagetau1} implies that the
first Johnson homomorphism can be regarded as a homomorphism
$\tau_1\colon \Torelli_g^1 \rightarrow \wedge^3 H_{\Z}$.  We then have:

\begin{lemma}
\label{lemma:cupjohnson}
Let $g \geq 3$.  Set $H_{\Z} = \HH_1(\Sigma_g^1;\Z)$, and let
$\tau_1\colon \Torelli_g^1 \rightarrow \wedge^3 H_{\Z}$ be the first
Johnson homomorphism.  Then as representations of $\Sp_{2g}(\Z)$, the
images of the following two maps are the same:
\begin{itemize}
\item The cup product pairing $\fc\colon \wedge^2 \HH^1(\Torelli_g^1;\Q) \rightarrow \HH^2(\Torelli_g^1;\Q)$.
\item The map $(\tau_1)_{\ast}\colon \HH_2(\Torelli_g^1;\Q) \rightarrow \HH_2(\wedge^3 H_{\Z};\Q)$
induced by the first Johnson homomorphism.
\end{itemize}
\end{lemma}
\begin{proof}
The first observation in the proof is as follows:

\begin{unnumberedclaim}
The image of the cup product pairing $\fc\colon \wedge^2 \HH^1(\Torelli_g^1;\Q) \rightarrow \HH^2(\Torelli_g^1;\Q)$
is the same as the image of the induced map
$(\tau_1)^{\ast}\colon \HH^2(\wedge^3 H_{\Z};\Q) \rightarrow \HH^2(\Torelli_g^1;\Q)$.
\end{unnumberedclaim}
\begin{proof}[Proof of claim]
Let $c\colon \wedge^2 \HH^1(\wedge^3 H_{\Z};\Q) \rightarrow \HH^2(\wedge^3 H_{\Z};\Q)$
be the cup product pairing.  Since $H_{\Z}$ is a free abelian group, the map
$c$ is an isomorphism.  By the naturality of the cup product pairing, we have
a commutative diagram
\[\begin{tikzcd}
\wedge^2 \HH^1(\wedge^3 H_{\Z};\Q) \arrow{d}{c}[swap]{\cong} \arrow{r}{\wedge^2 (\tau_1)^{\ast}} & \wedge^2 \HH^1(\Torelli_g^1;\Q) \arrow{d}{\fc} \\
\HH^2(\wedge^3 H_{\Z};\Q) \arrow{r}{(\tau_1)^{\ast}} & \HH^2(\Torelli_g^1;\Q).
\end{tikzcd}\]
Theorem \ref{theorem:torelliabel} implies that the top row is an isomorphism.  The claim follows.
\end{proof}

Working now with homology rather than cohomology, we have $\HH_1(\wedge^3 H_{\Z};\Q) \cong \wedge^3 H$.
This is an isomorphism of representations of $\Sp_{2g}(\Z)$.  As a representation of $\Sp_{2g}(\Z)$,
the representation $H$ is isomorphic to its dual.  Dualizing, we therefore get that
$\HH^1(\wedge^3 H_{\Z};\Q) \cong \wedge^3 H$ and thus that $\HH^2(\wedge^3 H_{\Z};\Q) \cong \wedge^2 \wedge^3 H$.
We deduce that $\HH^2(\wedge^3 H_{\Z};\Q)$ is also self-dual as a representation of $\Sp_{2g}(\Z)$.
From this, we see that as representations of $\Sp_{2g}(\Z)$ the images of the dual maps
\begin{align*}
&(\phi_1)^{\ast}\colon \HH^2(\wedge^3 H_{\Z};\Q) \rightarrow \HH^2(\Torelli_g^1;\Q), \\
&(\phi_1)_{\ast}\colon \HH_2(\Torelli_g^1;\Q)    \rightarrow \HH_2(\wedge^3 H_{\Z};\Q)
\end{align*}
are isomorphic.  The lemma follows.
\end{proof}

\section{Calculation of image of cup product pairing on surfaces with boundary}
\label{section:hainboundary}

We now prove Theorem \ref{maintheorem:hain} for surfaces with boundary.
The statement of this result is as follows.  It is due to Hain \cite{HainInfinitesimal}, but
our exposition is also influenced by an unpublished paper of Habegger--Sorger \cite{HabeggerSorger}.

\newtheorem*{maintheorem:hainboundary}{Theorem \ref{maintheorem:hain} (surface with boundary case)}
\begin{maintheorem:hainboundary}
Let $g \geq 6$.  The image of
the cup product pairing $\fc\colon \wedge^2 \HH^1(\Torelli_g^1;\Q) \rightarrow \HH^2(\Torelli_g^1;\Q)$
is isomorphic to the following representation of $\Sp_{2g}(\Z)$:
\[\bV_{1^2}(g)^{\oplus 2} \oplus \bV_{2,1^2}(g) \oplus \bV_{1^4}(g)^{\oplus 2} \oplus \bV_{2^2,1^2}(g) \oplus \bV_{1^6}(g).\]
\end{maintheorem:hainboundary}
\begin{proof}
Let $H_{\Z} = \HH_1(\Sigma_g^1;\Z)$ and $H = \HH_1(\Sigma_g^1;\Q)$.  Theorem \ref{theorem:imagetau1} implies that the
first Johnson homomorphism can be regarded as a homomorphism $\tau_1\colon \Torelli_g^1 \rightarrow \wedge^3 H_{\Z}$.
By Lemma \ref{lemma:cupjohnson}, it is enough to prove that the image of the induced map
\begin{equation}
\label{eqn:tau1h2}
(\tau_1)_{\ast}\colon \HH_2(\Torelli_g^1;\Q) \rightarrow \HH_2(\wedge^3 H_{\Z};\Q)
\end{equation}
is the indicated representation.  Since $\wedge^3 H_{\Z}$ is a free abelian group, we have that
the representation $\HH_2(\wedge^3 H_{\Z};\Q)$ is isomorphic to\footnote{Here we are using the fact
that $g \geq 6$ to ensure that the decomposition of $\wedge^2 \wedge^3 H$ is stable (see \S \ref{section:stablesp}), and thus
can be computed with LiE \cite{LieProgram}.} 
\begin{equation}
\label{eqn:decomposew2w3h}
\wedge^2 \wedge^3 H \cong \bV_0(g)^{\oplus 2} \oplus \bV_{1^2}(g)^{\oplus 3} \oplus \bV_{2^2}(g) \oplus \bV_{2,1^2}(g) \oplus \bV_{1^4}(g)^{\oplus 2} \oplus \bV_{2^2,1^2}(g) \oplus \bV_{1^6}(g).
\end{equation}
We will first prove that the image of the map \eqref{eqn:tau1h2} is contained in the desired
subrepresentation of this, and after that prove that in fact it is equal to it.

\begin{step}{1}
\label{step:hain.1}
The image of the map $(\tau_1)_{\ast}\colon \HH_2(\Torelli_g^1;\Q) \rightarrow \HH_2(\wedge^3 H_{\Z};\Q)$ is contained
in a subrepresentation of $\HH_2(\wedge^3 H_{\Z};\Q) \cong \wedge^2 \wedge^3 H$ that is isomorphic to the
following:
\begin{equation}
\label{eqn:step1.toprove}
\bV_{1^2}(g)^{\oplus 2} \oplus \bV_{2,1^2}(g) \oplus \bV_{1^4}(g)^{\oplus 2} \oplus \bV_{2^2,1^2}(g) \oplus \bV_{1^6}(g).
\end{equation}
\end{step}

Let $C$ be the cokernel of the map 
$(\tau_1)_{\ast}\colon \HH_2(\Torelli_g^1;\Q) \rightarrow \HH_2(\wedge^3 H_{\Z};\Q)$.  We
must prove that $C$ contains the subrepresentation
$V = \bV_0(g)^{\oplus 2} \oplus \bV_{1^2}(g) \oplus \bV_{2^2}(g)$,
which is obtained from \eqref{eqn:decomposew2w3h} by
deleting the subrepresentation \eqref{eqn:step1.toprove}.
Since $\HH_2(\Torelli_g^1;\Q)$ contains no trivial subrepresentations (Lemma \ref{lemma:h2torellitrivial}),
it is certainly the case that $C$ contains the subrepresentation
$\bV_0(g)^{\oplus 2}$.  It remains to prove that $C$ contains
$W = \bV_{1^2}(g) \oplus \bV_{2^2}(g)$.

The kernel of $\tau_1$ is the second term $\Torelli_g^1[2]$ of the Johnson filtration (Lemma \ref{lemma:johnsonker}).  We
therefore have a short exact sequence of groups
\[\begin{tikzcd}
0 \arrow{r} & \Torelli_g^1[2] \arrow{r} & \Torelli_g^1 \arrow{r}{\tau_1} & \wedge^3 H_{\Z} \arrow{r} & 1.
\end{tikzcd}\]
Associated to this is a five-term exact sequence in group homology (see \cite[Corollary VII.6.4]{BrownCohomology}).  For a group
$G$ acting on a vector space $M$, let $M_G$ be the $G$-coinvariants, i.e., the quotient $M/\SpanSet{$m-gm$}{$g \in G$, $m \in M$}$.
The five-term exact sequence then takes the form
\[\begin{tikzcd}[column sep=small]
\HH_2(\Torelli_g^1;\Q) \arrow{r}{(\tau_1)_{\ast}} & \HH_2(\wedge^3 H_{\Z};\Q) \arrow{r} & \HH_1(\Torelli_g^1[2];\Q)_{\Torelli_g^1} \arrow{r} &
\HH_1(\Torelli_g^1;\Q) \arrow{r}{(\phi_1)_{\ast}} & \HH_1(\wedge^3 H_{\Z};\Q) \arrow{r} & 0.
\end{tikzcd}\]
Here the coinvariants $\HH_1(\Torelli_g^1[2];\Q)_{\Torelli_g^1}$ are with respect to the action of $\Torelli_g^1$ induced by
the conjugation action of $\Torelli_g^1$ on its normal subgroup $\Torelli_g^1[2]$.  By Theorem \ref{theorem:torelliabel}, the
map 
\[(\phi_1)_{\ast}\colon \HH_1(\Torelli_g^1;\Q) \rightarrow \HH_1(\wedge^3 H_{\Z};\Q) = \wedge^3 H\]
is an isomorphism.  The above five-term exact sequence thus gives an exact sequence
\begin{equation}
\label{eqn:hainstep1seq}
\begin{tikzcd}[column sep=small]
\HH_2(\Torelli_g^1;\Q) \arrow{r}{(\tau_1)_{\ast}} & \HH_2(\wedge^3 H_{\Z};\Q) \arrow{r} & \HH_1(\Torelli_g^1[2];\Q)_{\Torelli_g^1} \arrow{r} & 0.
\end{tikzcd}
\end{equation}
From this, we see that $C \cong \HH_1(\Torelli_g^1[2];\Q)_{\Torelli_g^1}$.  This isomorphism is
$\Sp_{2g}(\Z)$-equivariant for the action of $\Sp_{2g}(\Z) = \Mod_g^1/\Torelli_g^1$ on 
$\HH_1(\Torelli_g^1[2];\Q)_{\Torelli_g^1}$ induced by the conjugation action of $\Mod_g^1$ on
$\Torelli_g^1[2]$.

Our goal, therefore, is to
prove that $\HH_1(\Torelli_g^1[2];\Q)_{\Torelli_g^1}$ contains the subrepresentation 
$W = \bV_{1^2}(g) \oplus \bV_{2^2}(g)$.
We will detect this with
the second Johnson homomorphism $\tau_2\colon \Torelli_g^1[2] \rightarrow H \otimes \Lie_3(H)$.  This is $\Torelli_g^1$-invariant,
so it factors through $\HH_1(\Torelli_g^1[2];\Q)_{\Torelli_g^1}$.  Theorem \ref{theorem:imagetau2}
implies that the $\Q$-span of the image of $\tau_2$ is
\[V' = \Sym^2(\wedge^2 H)/\wedge^4 H \cong \bV_{0}(g) \oplus \bV_{1^2}(g) \oplus \bV_{2^2}(g).\]
Each of these is an irreducible factor of $\HH_1(\Torelli_g^1[2];\Q)_{\Torelli_g^1}$, so certainly
$\HH_1(\Torelli_g^1[2];\Q)_{\Torelli_g^1}$ contains $W$.

Before going on, we make a remark.  At the beginning of this step, our goal was
to to detect $V = \bV_0(g)^{\oplus 2} \oplus \bV_{1^2}(g) \oplus \bV_{2^2}(g)$.  We
handled the trivial representations, and then used the second Johnson homomorphism
to detect $V' = \bV_{0}(g) \oplus \bV_{1^2}(g) \oplus \bV_{2^2}(g)$.   

The only difference between $V$ and $V'$ is that $V$ contains
one extra copy of $\bV_0(g)$.  Morita \cite[\S 4]{MoritaCasson} gave a beautiful geometric
interpretation of this extra factor of $\bV_0(g)$.  He used the Casson invariant
of homology $3$-spheres to construct a $\Mod_g^1$-invariant map $m\colon \Torelli_g^1[2] \rightarrow \Q$ that he
calls the ``core'' of the Casson invariant.  It follows from his results that $m$ does not factor through
the second Johnson homomorphism, so it provides the missing $\bV_0(g)$ factor.

Morita's homomorphism $m\colon \Torelli_g^1[2] \rightarrow \Q$ has the following beautiful formula.
Recall that Johnson \cite{Johnson2} proved that $\Torelli_g^1[2]$ is generated by separating twists.  Consider
a separating twist $T_z$.  The curve $z$ separates $\Sigma_g^1$ into two subsurfaces $S$ and $S'$.  Order them
such that $\partial \Sigma_g^1 \subset S'$, so $S \cong \Sigma_h^1$.  We then have
\[m(T_z) = 4h (h-1).\]
See \cite[Theorem 5.3]{MoritaCasson}.  As far as this author is aware, there is no way to prove that this
formula gives a homomorphism that does not use either the Casson invariant or other similarly deep facts
about the cohomology of the mapping class group. 

\begin{step}{2}
\label{step:hain.2}
The image of the map $(\tau_1)_{\ast}\colon \HH_2(\Torelli_g^1;\Q) \rightarrow \HH_2(\wedge^3 H_{\Z};\Q)$ equals
the following subrepresentation of $\HH_2(\wedge^3 H_{\Z};\Q) \cong \wedge^2 \wedge^3 H$:
\[\bV_{1^2}(g)^{\oplus 2} \oplus \bV_{2,1^2}(g) \oplus \bV_{1^4}(g)^{\oplus 2} \oplus \bV_{2^2,1^2}(g) \oplus \bV_{1^6}(g).\]
\end{step}

By Step \ref{step:hain.1}, it is enough to prove that the image of the map
\begin{equation}
\label{eqn:step2.hit}
(\tau_1)_{\ast}\colon \HH_2(\Torelli_g^1;\Q) \rightarrow \HH_2(\wedge^3 H_{\Z};\Q) \cong \wedge^2 \wedge^3 H.
\end{equation}
contains each of the indicated irreducible factors.  Before doing this, we expand on
the isomorphism $\HH_2(\wedge^3 H_{\Z};\Q) \cong \wedge^2 \wedge^3 H$.
Since $H_{\Z}$ is an abelian group, its multiplication map $H_{\Z} \times H_{\Z} \rightarrow H_{\Z}$ induces a product
structure on $\HH_{\bullet}(\wedge^3 H_{\Z};\Q)$ called the Pontryagin product (see \cite[\S V.5]{BrownCohomology}).  This
makes $\HH_{\bullet}(\wedge^3 H_{\Z};\Q)$ into a graded ring, and 
\[\HH_{\bullet}(\wedge^3 H_{\Z};\Q) \cong \wedge^{\bullet} \wedge^3 H.\]
The isomorphism $\HH_2(\wedge^3 H_{\Z};\Q) \cong \wedge^2 \wedge^3 H$ discussed above is a special case of this.

To prove our result, we will need
to detect elements of $\wedge^2 \wedge^3 H$ that lie in the image of the map \eqref{eqn:step2.hit}. 
We will do with using so-called ``abelian cycles''.  
Consider commuting elements $f,h \in \Torelli_g^1$.
Since $f$ and $h$ commute, we can define a homomorphism $\Psi_{f,h}\colon \Z^2 \rightarrow \Torelli_g^1$
taking the standard basis vectors of $\Z^2$ to $f$ and $h$.  We therefore get an induced map
$(\Psi_{f,h})_{\ast}\colon \HH_2(\Z^2;\Q) \rightarrow \HH_2(\Torelli_g^1;\Q)$.  We have
$\HH_2(\Z^2;\Q) \cong \wedge^2 \Q^2 \cong \Q$ generated by the fundamental class $[\Z^2] \in \HH_2(\Z^2;\Q)$,
so we have an element $(\Psi_{f,h})_{\ast}([\Z^2]) \in \HH_2(\Torelli_g^1;\Q)$.

Define $\fC(f,h) \in \HH_2(\wedge^3 H_{\Z};\Q) = \wedge^2 \wedge^3 H$ to be the image of $(\Psi_{f,h})_{\ast}([\Z^2])$ under
the map $(\tau_1)_{\ast}$.  This equals the image of the fundamental class $[\Z^2]$ under the map induced
by the map $\tau_1 \circ \Psi_{f,h}\colon \Z^2 \rightarrow \wedge^3 H_{\Z}$.  The Pontryagin product makes
the homology groups of both $\Z^2$ and $\wedge^3 H_{\Z}$ into a ring, and the map on homology induced
by $\tau_1 \circ \Psi_{f,h}$ is a ring homomorphism.  Since the fundamental class $[\Z^2] \in \HH_2(\Z^2;\Q)$
is the Pontryagin product of the first homology classes of the two standard generators of $\Z^2$, we deduce the following key
formula:
\[\fC(f,h) = \tau_1(f) \wedge \tau_1(h) \in \wedge^2 \wedge^3 H.\]
In other words, given any two commuting elements $f,h \in \Torelli_g^1$ the element $\fC(f,h) = \tau_1(f) \wedge \tau_1(h) \in \wedge^2 \wedge^3 H$
lies in the image of the map \eqref{eqn:step2.hit}.  Using this, we will prove that each of the desired irreducible
factors lies in the image of \eqref{eqn:step2.hit}.  We will do the calculations roughly in increasing
order of difficulty.  As in the following figure, let $\{a_1,b_1,\ldots,a_g,b_g\}$ be the standard symplectic basis for $H$:\\
\Figure{StandardSpBasis}
Also, let $\omega \in \wedge^2 H$ be the symplectic form.

\begin{claim}{1}
\label{claim:hainboundary.1}
The image of the map $(\tau_1)_{\ast}\colon \HH_2(\Torelli_g^1;\Q) \rightarrow \HH_2(\wedge^3 H_{\Z};\Q)$ contains
a copy of $\bV_{2^2,1^2}(g)$.
\end{claim}

The vector 
\[(a_1 \wedge a_2 \wedge a_3) \wedge (a_1 \wedge a_2 \wedge a_4) \in \wedge^2 \wedge^3 H = \HH_2(\wedge^3 H_{\Z};\Q)\]
is a highest weight vector for a copy of $\bV_{2^2,1^2}(g)$, so it is enough to prove that
this vector lies in the image of $(\tau_1)_{\ast}$.
Let $Z \cong \Sigma_0^4$ and $W \cong \Sigma_0^4$ be the following subsurfaces:\\
\Figure{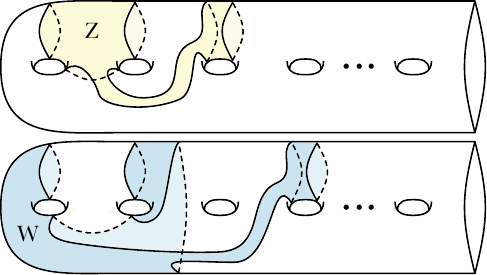}
Let $f \in \Torelli_g^1$ be a simply intersecting pair map supported on $Z$ and let $h \in \Torelli_g^1$ be a simply
intersecting pair map supported on $W$.  By Lemma \ref{lemma:simplyintersectingpair}, we
have $\tau_1(f) = \pm a_1 \wedge a_2 \wedge a_3$ and $\tau_1(h) = \pm a_1 \wedge a_2 \wedge a_4$.
Since the interiors of $Z$ and $W$ are disjoint,
the mapping classes $f$ and $h$ commute.  The image of $(\tau_1)_{\ast}$ thus
contains
\[\fC(f,h) = \pm (a_1 \wedge a_2 \wedge a_3) \wedge (a_1 \wedge a_2 \wedge a_4).\]

\begin{claim}{2}
\label{claim:hainboundary.2}
The image of the map $(\tau_1)_{\ast}\colon \HH_2(\Torelli_g^1;\Q) \rightarrow \HH_2(\wedge^3 H_{\Z};\Q)$ contains
a copy of $\bV_{2,1^2}(g)$.
\end{claim}

Recall that $\omega \in \wedge^2 H$ is the symplectic form.  The vector
\[(a_1 \wedge a_2 \wedge a_3) \wedge (a_1 \wedge \omega) \in \wedge^2 \wedge^3 H = \HH_2(\wedge^3 H_{\Z};\Q)\]
is a highest weight vector for a copy of $\bV_{2,1^2}(g)$, so it is enough to prove that
this vector lies in the image of $(\tau_1)_{\ast}$.  Let $Z \cong \Sigma_0^4$ be the following subsurface
and let $T_x T_y^{-1} \in \Torelli_g^1$ be the following bounding pair map:\\
\Figure{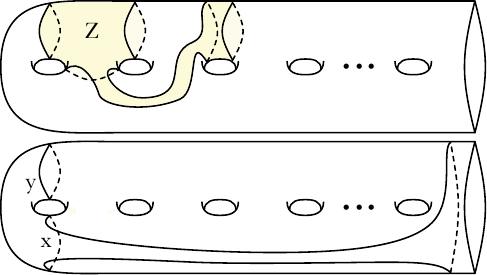}
Let $f \in \Torelli_g^1$ be a simply intersecting pair map supported on $Z$.  By 
Lemma \ref{lemma:simplyintersectingpair} we have $\tau_1(f) = \pm a_1 \wedge a_2 \wedge a_3$, and by
Lemma \ref{lemma:taubpmap} we have
\[\tau_1(T_x T_y^{-1}) = a_1 \wedge (a_2 \wedge b_2 + \cdots + a_g \wedge b_g) = a_1 \wedge (a_1 \wedge b_1 + \cdots + a_g \wedge b_g) = a_1 \wedge \omega.\]
Since $x \cup y$ is disjoint from the interior of $Z$, the mapping classes $f$ and $T_x T_y^{-1}$
commute.  The image of $(\tau_1)_{\ast}$ thus
contains
\[\fC(f,T_x T_y^{-1}) = \pm (a_1 \wedge a_2 \wedge a_3) \wedge (a_1 \wedge \omega).\]

\begin{claim}{3}
\label{claim:hainboundary.3}
The image of the map $(\tau_1)_{\ast}\colon \HH_2(\Torelli_g^1;\Q) \rightarrow \HH_2(\wedge^3 H_{\Z};\Q)$ contains
a copy of $\bV_{1^6}(g)$.
\end{claim}

It is awkward to find a highest weight vector in $\wedge^2 \wedge^3 H$ for $\bV_{1^6}(g)$, so we
do something slightly different.  Let $\phi\colon \wedge^2 \wedge^3 H \rightarrow \wedge^6 H$ be the following
map:
\[\phi((u_1 \wedge v_1 \wedge w_1) \wedge (u_2 \wedge v_2 \wedge w_2)) = u_1 \wedge v_1 \wedge w_1 \wedge u_2 \wedge v_2 \wedge w_2.\]
The vector $a_1 \wedge \cdots \wedge a_6 \in \wedge^6 H$ is a highest weight vector for a copy
of $\bV_{1^6}(g)$ in $\wedge^6 H$.  It is therefore enough to construct an element $\theta$ in the
image of $(\tau_1)_{\ast}$ such that $\phi(\theta) = a_1 \wedge \cdots \wedge a_6$.
Let $Z \cong \Sigma_0^4$ and $W \cong \Sigma_0^4$ be the following subsurfaces:\\
\Figure{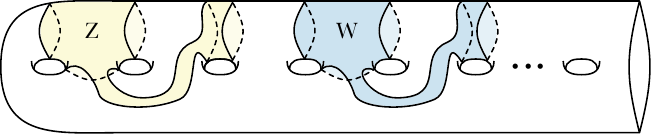}
Let $f \in \Torelli_g^1$ be a simply intersecting pair map supported on $Z$ and let $h \in \Torelli_g^1$ be a simply
intersecting pair map supported on $W$.  By Lemma \ref{lemma:simplyintersectingpair}, we
have $\tau_1(f) = \pm a_1 \wedge a_2 \wedge a_3$ and $\tau_1(h) = \pm a_4 \wedge a_5 \wedge a_6$.
Since the interiors of $Z$ and $W$ are disjoint,
the mapping classes $f$ and $h$ commute.  The image of $(\tau_1)_{\ast}$ thus
contains
\[\theta = \fC(f,h) = \pm (a_1 \wedge a_2 \wedge a_3) \wedge (a_4 \wedge a_5 \wedge a_6),\]
and $\phi(\theta) = \pm a_1 \wedge \cdots \wedge a_6 \in \wedge^6 H$.

\begin{claim}{4}
\label{claim:hainboundary.4}
The image of the map $(\tau_1)_{\ast}\colon \HH_2(\Torelli_g^1;\Q) \rightarrow \HH_2(\wedge^3 H_{\Z};\Q)$ contains
a copy of $\bV_{1^4}(g)^{\oplus 2}$.
\end{claim}

We will have to be careful to ensure that the two copies of $\bV_{1^4}(g)$ we find are genuinely different.  Recall
that $H$ is embedded in $\wedge^3 H$ via the map $H \hookrightarrow \wedge^3 H$ taking $h \in H$ to $h \wedge \omega$.
Let $\cK$ be the kernel of the map $p\colon \wedge^2 \wedge^3 H \rightarrow \wedge^2 ((\wedge^3 H)/H)$, so we have a short exact sequence
of representations
\[0 \longrightarrow \cK \longrightarrow \wedge^2 \wedge^3 H \stackrel{p}{\longrightarrow} \wedge^2 \left(\left(\wedge^3 H\right)/H\right) \longrightarrow 0.\]
We will first find a copy of $\bV_{1^4}(g)$ in the image of $(\tau_1)_{\ast}$ that lies in $\cK$.
Let $\phi\colon \wedge^2 \wedge^3 H \rightarrow \wedge^6 H$ be the map from Claim \ref{claim:hainboundary.3}.  The vector
$a_1 \wedge \cdots \wedge a_4 \wedge \omega \in \wedge^6 H$ is a highest weight vector for a copy
of $\bV_{1^4}(g)$ in $\wedge^6 H$.  It is therefore enough to construct an element $\theta$ in the
image of $(\tau_1)_{\ast}$ such that $\theta \in \cK$ and $\phi(\theta) = a_1 \wedge \cdots \wedge a_4 \wedge \omega$.
Let $Z \cong \Sigma_0^4$ be the following subsurface
and let $T_x T_y^{-1} \in \Torelli_g^1$ be the following bounding pair map:\\
\Figure{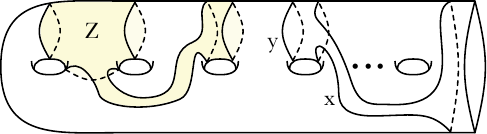} 
Let $f \in \Torelli_g^1$ be a simply intersecting pair map supported on $Z$.  By
Lemma \ref{lemma:simplyintersectingpair} we have $\tau_1(f) = \pm a_1 \wedge a_2 \wedge a_3$, and by
Lemma \ref{lemma:taubpmap} we have
\begin{align*}
\tau_1(T_x T_y^{-1}) &= a_4 \wedge (a_1 \wedge b_1 + \cdots + \widehat{a_4 \wedge b_4} + \cdots + a_g \wedge b_g) \\
                     &= a_4 \wedge (a_1 \wedge b_1 + \cdots + a_g \wedge b_g) = a_4 \wedge \omega.
\end{align*}
Since $x \cup y$ is disjoint from the interior of $Z$, the mapping classes $f$ and $T_x T_y^{-1}$
commute.  The image of $(\tau_1)_{\ast}$ thus
contains
\[\theta = \fC(f,T_x T_y^{-1}) = \pm (a_1 \wedge a_2 \wedge a_3) \wedge (a_4 \wedge \omega).\]
We have $\theta \in \cK$ since $a_4 \wedge \omega \in H \subset \wedge^3 H$, and
$\phi(\theta) = a_1 \wedge \cdots \wedge a_4 \wedge \omega$, as desired.

Since the copy of $\bV_{1^4}(g)$ we found lies in the kernel $\cK$ of the projection
$p\colon \wedge^2 \wedge^3 H \rightarrow \wedge^2 ((\wedge^3 H)/H)$, to find a second
copy of $\bV_{1^4}(g)$ it is enough to find a copy of $\bV_{1^4}(g)$ that maps nontrivially
to $\wedge^2 ((\wedge^3 H)/H)$.  For $\kappa \in \wedge^2 \wedge^3 H$, let
$\okappa = p(\kappa) \in \wedge^2 ((\wedge^3 H)/H)$.  We must find some $\theta' \in \wedge^2 \wedge^3 H$
in the image of $(\tau_1)_{\ast}$ such that the subrepresentation of $\wedge^2 ((\wedge^3 H)/H)$ generated
by $\otheta'$ contains a copy of $\bV_{1^4}(g)$.

Let $Z \cong \Sigma_0^4$ be the following subsurface
and let $T_{x'} T_{y}^{-1} \in \Torelli_g^1$ be the following bounding pair map:\\
\Figure{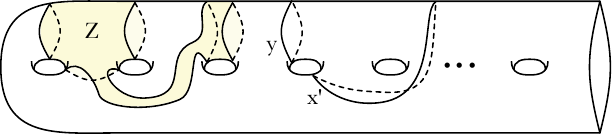}
Note that $Z$ and $y$ are the same as in the previous figure.  Let $f \in \Torelli_g^1$ be a simply intersecting pair map supported on $Z$.  By
Lemma \ref{lemma:simplyintersectingpair} we have $\tau_1(f) = \pm a_1 \wedge a_2 \wedge a_3$, and by
Lemma \ref{lemma:taubpmap} we have $\tau_1(T_{x'} T_{y}^{-1}) = a_4 \wedge a_5 \wedge b_5$.
Since $x' \cup y$ is disjoint from the interior of $Z$, the mapping classes $f$ and $T_{x'} T_y^{-1}$
commute.  The image of $(\tau_1)_{\ast}$ thus
contains
\[\theta' = \fC(f,T_{x'} T_y^{-1}) = \pm (a_1 \wedge a_2 \wedge a_3) \wedge (a_4 \wedge a_5 \wedge b_5).\]
Lemma \ref{lemma:certify1} says that
$\otheta' \in \wedge^2 (\wedge^3 H)/H$ generates a subrepresentation containing a copy
of $\bV_{1^4}(g)$, as desired.

\begin{claim}{5}
\label{claim:hainboundary.5}
The image of the map $(\tau_1)_{\ast}\colon \HH_2(\Torelli_g^1;\Q) \rightarrow \HH_2(\wedge^3 H_{\Z};\Q)$ contains
a copy of $\bV_{1^2}(g)^{\oplus 2}$.
\end{claim}

Just like in Claim \ref{claim:hainboundary.4}, 
let $\cK$ be the kernel of the map $p\colon \wedge^2 \wedge^3 H \rightarrow \wedge^2 ((\wedge^3 H)/H)$, so we have a short exact sequence
of representations
\[0 \longrightarrow \cK \longrightarrow \wedge^2 \wedge^3 H \stackrel{p}{\longrightarrow} \wedge^2 \left(\left(\wedge^3 H\right)/H\right) \longrightarrow 0.\]
We will first find a copy of $\bV_{1^2}(g)$ in the image of $(\tau_1)_{\ast}$ that lies in $\cK$.
Recall that $\omega(-,-)$ is the symplectic form on $H$.
Let $q\colon \wedge^3 H \rightarrow H$ be the map defined by the formula
\[q(h_1 \wedge h_2 \wedge h_3) = \omega(h_1,h_2) h_3 - \omega(h_1,h_3) h_2 + \omega(h_2,h_3) h_1 \quad \text{for $h_1,h_2,h_3 \in H$}\]
and let $\psi\colon \wedge^2 \wedge^3 H \rightarrow \wedge^2 H$ be the map $\psi = \wedge^2 q$.  The vector
$a_1 \wedge a_2 \in \wedge^2 H$ is a highest weight vector for a copy of $\bV_{1^2}(g)$ in $\wedge^2 H$.  It is
therefore enough to construct an element $\theta$ in the
image of $(\tau_1)_{\ast}$ such that $\theta \in \cK$ and $\phi(\theta)$ is a nonzero multiple of $a_1 \wedge a_2$.
Let $T_x T_y^{-1} \in \Torelli_g^1$ and
$T_z T_w^{-1} \in \Torelli_g^1$ be the following bounding pair maps:\\
\Figure{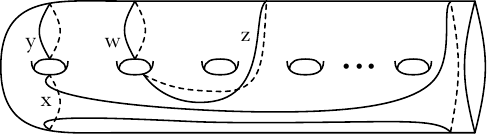}
By Lemma \ref{lemma:taubpmap}, we have
\begin{align*}
\tau_1(T_x T_y^{-1}) &= a_1 \wedge (a_2 \wedge b_2 + \cdots + a_g \wedge b_g) \\
                     &= a_1 \wedge (a_1 \wedge b_1 + \cdots + a_g \wedge b_g) = a_1 \wedge \omega,\\
\tau_1(T_z T_w^{-1}) &= a_2 \wedge a_3 \wedge b_3.
\end{align*}
Since $x \cup y$ and $z \cup w$ are disjoint, the mapping classes $T_x T_y^{-1}$ and $T_z T_w^{-1}$
commute.  The image of $(\tau_1)_{\ast}$ thus contains
\[\theta = \fC(T_x T_y^{-1},T_z T_w^{-1}) = (a_1 \wedge \omega) \wedge (a_2 \wedge a_3 \wedge b_3).\]
We have $\theta \in \cK$ since $a_1 \wedge \omega \in H \subset \wedge^3 H$, and
\begin{align*}
\psi(\theta) &= q(a_1 \wedge \omega) \wedge q(a_2 \wedge a_3 \wedge b_3)) \\
             &= q(a_1 \wedge a_2 \wedge b_2 + a_1 \wedge a_3 \wedge b_3 + \cdots + a_1 \wedge a_g \wedge b_g) \wedge q(a_2 \wedge a_3 \wedge b_3) \\
             &= (g-1) a_1 \wedge a_2,
\end{align*}
as desired.

Since the copy of $\bV_{1^2}(g)$ we found lies in the kernel $\cK$ of the projection
$p\colon \wedge^2 \wedge^3 H \rightarrow \wedge^2 ((\wedge^3 H)/H)$, to find a second
copy of $\bV_{1^4}(g)$ it is enough to find a copy of $\bV_{1^2}(g)$ that maps nontrivially
to $\wedge^2 ((\wedge^3 H)/H)$.  For $\kappa \in \wedge^2 \wedge^3 H$, let
$\okappa = p(\kappa) \in \wedge^2 ((\wedge^3 H)/H)$.  We must find some $\theta' \in \wedge^2 \wedge^3 H$
in the image of $(\tau_1)_{\ast}$ such that the subrepresentation of $\wedge^2 ((\wedge^3 H)/H)$ generated
by $\otheta'$ contains a copy of $\bV_{1^2}(g)$.

Let $T_{x'} T_y^{-1} \in \Torelli_g^1$ and
and $T_z T_w^{-1} \in \Torelli_g^1$ be the following bounding pair maps:\\
\Figure{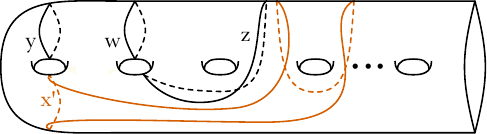}
Note that $y$ and $z$ and $w$ 
are the same as in the previous figure.  
By Lemma \ref{lemma:taubpmap} we have\footnote{The minus sign is there in the first formula for orientation reasons.  It would
be cleaner to use $T_y T_{x'}^{-1}$, but we decided that this would be harder to follow.}
\begin{align*}
\tau_1(T_{x'} T_y^{-1}) &= -a_1 \wedge a_4 \wedge b_4,\\
\tau_1(T_z T_w^{-1})    &= a_2 \wedge a_3 \wedge b_3.
\end{align*}
Since $x' \cup y$ and $z \cup w$ are disjoint, the mapping classes $T_{x'} T_y^{-1}$ and $T_z T_w^{-1}$ commute
The image of $(\tau_1)_{\ast}$ thus
contains
\[\theta' = \fC(T_{x'} T_y^{-1},T_z T_w^{-1}) = - (a_1 \wedge a_4 \wedge b_4) \wedge (a_2 \wedge a_3 \wedge b_3).\]
Lemma \ref{lemma:certify2} says that $\otheta' \in \wedge^2 (\wedge^3 H)/H$ generates a subrepresentation containing a copy
of $\bV_{1^2}(g)$, as desired.
\end{proof}

Before moving on to deal with punctured and closed surfaces, we record a consequence
of the above proof.  Recall that for a group
$G$ acting on a vector space $M$, we denote by $M_G$ the $G$-coinvariants, i.e., the quotient 
$M/\SpanSet{$m-gm$}{$g \in G$, $m \in M$}$.  The conjugation action of $\Mod_g^1$ on its normal
subgroup $\Torelli_g^1[2]$ induces an action of $\Mod_g^1$ on $\HH_1(\Torelli_g^1[2];\Q)$.
Passing to the $\Torelli_g^1$-coinvariants, we get an action of
$\Mod_g^1/\Torelli_g^1 \cong \Sp_{2g}(\Z)$
on $\HH_1(\Torelli_g^1[2];\Q)_{\Torelli_g^1}$.

\begin{corollary}
\label{corollary:kg1coinv}
For $g \geq 6$, we have 
$\HH_1(\Torelli_g^1[2];\Q)_{\Torelli_g^1} \cong \bV_0(g)^{\oplus 2} \oplus \bV_{1^2}(g) \oplus \bV_{2^2}(g)$.
\end{corollary}
\begin{proof}
Let $H = \HH_1(\Sigma_g^1;\Q)$ and $H_{\Z} = \HH_1(\Sigma_g^1;\Z)$.
Recall the following exact sequence \eqref{eqn:hainstep1seq} that we constructed
in Step \ref{step:hain.1} of the proof above:
\[\begin{tikzcd}[column sep=small]
\HH_2(\Torelli_g^1;\Q) \arrow{r}{(\tau_1)_{\ast}} & \HH_2(\wedge^3 H_{\Z};\Q) \arrow{r} & \HH_1(\Torelli_g^1[2];\Q)_{\Torelli_g^1} \arrow{r} & 0.
\end{tikzcd}\]
As we discussed before Step \ref{step:hain.1}, the representation $\HH_2(\wedge^3 H_{\Z};\Q)$ is isomorphic to
\[\wedge^2 \wedge^3 H \cong \bV_0(g)^{\oplus 2} \oplus \bV_{1^2}(g)^{\oplus 3} \oplus \bV_{2^2}(g) \oplus \bV_{2,1^2}(g) \oplus \bV_{1^4}(g)^{\oplus 2} \oplus \bV_{2^2,1^2}(g) \oplus \bV_{1^6}(g).\]
We proved in Theorem \ref{maintheorem:hain} that the image of $(\tau_1)_{\ast}$ is
\[\bV_{1^2}(g)^{\oplus 2} \oplus \bV_{2,1^2}(g) \oplus \bV_{1^4}(g)^{\oplus 2} \oplus \bV_{2^2,1^2}(g) \oplus \bV_{1^6}(g).\]
Combining the above three math displays gives the corollary.
\end{proof}

\section{Calculation of image of cup product pairing on punctured surfaces}
\label{section:hainpuncture}

We now show how to deal with the Torelli group $\Torelli_{g,1}$ on a genus $g$ surface
$\Sigma_{g,1}$ with one puncture:

\newtheorem*{maintheorem:hainpunctured}{Theorem \ref{maintheorem:hain} (punctured surface case)}
\begin{maintheorem:hainpunctured}
Let $g \geq 6$.  The image of
the cup product pairing $\fc\colon \wedge^2 \HH^1(\Torelli_{g,1};\Q) \rightarrow \HH^2(\Torelli_{g,1};\Q)$
is isomorphic to the following representation of $\Sp_{2g}(\Z)$:
\[\bV_{1^2}(g)^{\oplus 2} \oplus \bV_{2,1^2}(g) \oplus \bV_{1^4}(g)^{\oplus 2} \oplus \bV_{2^2,1^2}(g) \oplus \bV_{1^6}(g).\]
\end{maintheorem:hainpunctured}
\begin{proof}
Note that the only difference between this and what we proved for surfaces with boundary is the presence
of one additional $\bV_0(g)$ factor.  
Let $H = \HH_1(\Sigma_g^1;\Q)$ and $H_{\Z} = \HH_1(\Sigma_g^1;\Z)$
and let $b = \partial \Sigma_g^1$.  There is a central extension
\begin{equation}
\label{eqn:boundarypunctureseq}
1 \longrightarrow \Z \longrightarrow \Mod_g^1 \longrightarrow \Mod_{g,1} \longrightarrow 1
\end{equation}
whose central $\Z$ term is generated by the Dehn twist $T_b$ (see \cite[Proposition 3.19]{FarbMargalitPrimer}).
Since the separating twist $T_b$ lies in $\Torelli_g^1$, the above restricts to a central
extension
\[1 \longrightarrow \Z \longrightarrow \Torelli_g^1 \longrightarrow \Torelli_{g,1} \longrightarrow 1.\]
In fact, $T_b$ lies in $\Torelli_g^1[2]$.  This implies that the first Johnson
$\tau_1\colon \Torelli_g^1 \rightarrow \wedge^3 H_{\Z}$ factors through a homomorphism
$\Torelli_{g,1} \rightarrow \wedge^3 H_{\Z}$.  We will also call this the first Johnson homomorphism
and denote it by $\tau_1$.  Since $\tau_1$ detects $\HH_1(\Torelli_g^1;\Q)$ (Theorem \ref{theorem:torelliabel})
and factors through $\Torelli_{g,1}$, it follows that $\tau_1$ also detects $\HH_1(\Torelli_{g,1};\Q)$,
i.e., it induces an isomorphism
\[(\tau_1)_{\ast}\colon \HH_1(\Torelli_{g,1};\Q) \stackrel{\cong}{\longrightarrow} \HH_1(\wedge^3 H_{\Z};\Q) = \wedge^3 H.\]
In light of this, the exact same proof we gave in Lemma \ref{lemma:cupjohnson} for $\Torelli_g^1$ shows
that the image of the cup product pairing $\wedge^2 \HH^1(\Torelli_{g,1};\Q) \rightarrow \HH^2(\Torelli_{g,1};\Q)$
is isomorphic to the image of the map
$(\tau_1)_{\ast}\colon \HH_2(\Torelli_{g,1};\Q) \rightarrow \HH_2(\wedge^3 H_{\Z};\Q)$.  Summarizing where
we are, our goal is now the following:
\begin{itemize}
\item[($\dagger$)] Prove that the image of $(\tau_1)_{\ast}\colon \HH_2(\Torelli_{g,1};\Q) \rightarrow \HH_2(\wedge^3 H_{\Z};\Q)$ equals
the image of $(\tau_1)_{\ast}\colon \HH_2(\Torelli_g^1;\Q) \rightarrow \HH_2(\wedge^3 H_{\Z};\Q)$ plus one additional $\bV_0(g)$ factor.
\end{itemize}
As notation, let $\Torelli_{g,1}[2]$ be the kernel of the map
$\tau_1\colon \Torelli_{g,1} \rightarrow \wedge^3 H_{\Z}$.  Just like in the proof of
Step \ref{step:hain.1} in \S \ref{section:hainboundary}, the five-term exact sequence in
group homology associated to the short exact sequence
\[1 \longrightarrow \Torelli_{g,1}[2] \longrightarrow \Torelli_{g,1} \stackrel{\tau_1}{\longrightarrow} \wedge^3 H_{\Z} \longrightarrow 1\]
can be analyzed to give an exact sequence
\[\begin{tikzcd}[column sep=small]
\HH_2(\Torelli_{g,1};\Q) \arrow{r}{(\tau_1)_{\ast}} & \HH_2(\wedge^3 H_{\Z};\Q) \arrow{r} & \HH_1(\Torelli_{g,1}[2];\Q)_{\Torelli_{g,1}} \arrow{r} & 0.
\end{tikzcd}\]
This fits into a commutative diagram
\[\begin{tikzcd}[column sep=small]
\HH_2(\Torelli_g^1;\Q) \arrow{d} \arrow{r}{(\tau_1)_{\ast}} & \HH_2(\wedge^3 H_{\Z};\Q) \arrow[equals]{d} \arrow{r} & \HH_1(\Torelli_g^1[2];\Q)_{\Torelli_g^1} \arrow{d} \arrow{r} & 0 \\
\HH_2(\Torelli_{g,1};\Q) \arrow{r}{(\tau_1)_{\ast}} & \HH_2(\wedge^3 H_{\Z};\Q) \arrow{r} & \HH_1(\Torelli_{g,1}[2];\Q)_{\Torelli_{g,1}} \arrow{r} & 0.
\end{tikzcd}\]
In light of this, to prove $(\dagger)$ it is enough to prove the following:

\begin{unnumberedclaim}
The $\Sp_{2g}(\Z)$ representation $\HH_1(\Torelli_g^1[2];\Q)_{\Torelli_g^1}$ equals
$\HH_1(\Torelli_{g,1}[2];\Q)_{\Torelli_{g,1}}$ plus one additional $\bV_0(g)$ factor.
\end{unnumberedclaim}

In the central extension \eqref{eqn:boundarypunctureseq}, the kernel $\Z$ is generated
by the Dehn twist $T_b$ about the boundary component $b$.  Since this is a separating twist,
it lies in $\Torelli_g^1[2]$ and \eqref{eqn:boundarypunctureseq} restricts to a central extension
\[1 \longrightarrow \Z \longrightarrow \Torelli_g^1[2] \longrightarrow \Torelli_{g,1}[2] \longrightarrow 1.\]
Since this is a central extension, we see that the conjugation action of $\Torelli_{g,1}$ on $\Torelli_{g,1}[2]$
lifts to an action of $\Torelli_g^1$ on $\Torelli_{g,1}[2]$.  This satisfies
\[\HH_1(\Torelli_{g,1}[2];\Q)_{\Torelli_g^1} = \HH_1(\Torelli_{g,1}[2];\Q)_{\Torelli_{g,1}}.\]
We can therefore attempt to compare $\HH_1(\Torelli_g^1[2];\Q)_{\Torelli_g^1}$ to
$\HH_1(\Torelli_{g,1}[2];\Q)_{\Torelli_g^1}$, which will be easier since we are taking coinvariants
with respect to the same group.

Since the above is a central extension, the associated 
five-term exact sequence in group homology (see \cite[Corollary VII.6.4]{BrownCohomology}) contains
the segment
\[\Q \longrightarrow \HH_1(\Torelli_g^1[2];\Q) \longrightarrow \HH_1(\Torelli_{g,1}[2];\Q) \longrightarrow 0.\]
Taking coinvariants is right exact, so this gives an exact sequence
\[\Q \longrightarrow \HH_1(\Torelli_g^1[2];\Q)_{\Torelli_g^1} \longrightarrow \HH_1(\Torelli_{g,1}[2];\Q)_{\Torelli_g^1} \longrightarrow 0.\]
We claim that the map from $\Q \cong \bV_0(g)$ to $\HH_1(\Torelli_g^1[2];\Q)_{\Torelli_g^1}$ is injective.  Indeed, the second
Johnson homomorphism $\tau_2$
induces a map on $\HH_1(\Torelli_g^1[2];\Q)_{\Torelli_g^1}$, and Lemma \ref{lemma:tau2septwist}
implies that $\tau_2(T_b) \neq 0$.  Since $T_b$ generates the $\Z$ that became the $\Q$ above, the claim
follows.  We therefore have a short exact sequence
\[0 \longrightarrow \bV_0(g) \longrightarrow \HH_1(\Torelli_g^1[2];\Q)_{\Torelli_g^1} \longrightarrow \HH_1(\Torelli_{g,1}[2];\Q)_{\Torelli_g^1} \longrightarrow 0,\]
as desired.
\end{proof}

Exactly like at the end of \S \ref{section:hainboundary}, the following is an immediate consequence
of the above proof:

\begin{corollary}
\label{corollary:kgpunccoinv}
For $g \geq 6$, we have
$\HH_1(\Torelli_{g,1}[2];\Q)_{\Torelli_g^1} \cong \bV_0(g) \oplus \bV_{1^2}(g) \oplus \bV_{2^2}(g)$.
\end{corollary}

\section{Calculation of image of cup product pairing on closed surfaces}
\label{section:hainclosed}

We close this part of the paper by showing how to deal with the Torelli group $\Torelli_g$ on a
closed genus $g$ surface $\Sigma_g$:

\newtheorem*{maintheorem:hainclosed}{Theorem \ref{maintheorem:hain} (closed surface case)}
\begin{maintheorem:hainclosed}
Let $g \geq 6$.  The image of
the cup product pairing $\fc\colon \wedge^2 \HH^1(\Torelli_{g};\Q) \rightarrow \HH^2(\Torelli_{g};\Q)$
is isomorphic to the following representation of $\Sp_{2g}(\Z)$:
\[\bV_{1^2}(g) \oplus \bV_{1^4}(g) \oplus \bV_{2^2,1^2}(g) \oplus \bV_{1^6}(g).\]
\end{maintheorem:hainclosed}
\begin{proof}
Let $H = \HH_1(\Sigma_g^1;\Q)$ and $H_{\Z} = \HH_1(\Sigma_g^1;\Z)$
and let $\pi = \pi_1(\Sigma_g,\ast)$.  The groups $\Mod_{g,1}$ and $\Mod_g$ are connected by a Birman exact sequence
of the form
\[1 \longrightarrow \pi \longrightarrow \Mod_{g,1} \longrightarrow \Mod_g \longrightarrow 1.\]
See \cite[\S 4.2]{FarbMargalitPrimer}.  Since $\HH_1(\Sigma_{g,1}) = \HH_1(\Sigma_g)$, the action of
$\Mod_{g,1}$ on $H_{\Z}$ factors through $\Mod_g$.  This implies that the kernel $\pi$ of the Birman
exact sequence acts trivially on $H_{\Z}$, so $\pi$ is contained in the Torelli group.  It follows
that the Birman exact sequence restricts to an exact sequence
\[1 \longrightarrow \pi \longrightarrow \Torelli_{g,1} \longrightarrow \Torelli_g \longrightarrow 1.\]
Let $\omega \in \wedge^2 H$ be the symplectic form.  Since $\omega$ restricts to a $\Z$-valued symplectic
form on $H_{\Z}$, we can actually regard $\omega$ as an element of $\wedge^2 H_{\Z}$.
Johnson \cite{JohnsonHomo} proved that the first Johnson homomorphism $\tau_1\colon \Torelli_{g,1} \longrightarrow \wedge^3 H_{\Z}$
constructed in the previous section fits into a commutative diagram
\[\begin{tikzcd}
\pi \arrow{r} \arrow{d}{\tau_1} & \Torelli_{g,1} \arrow{d}{\tau_1} \\
H_{\Z} \arrow{r}{\iota}         & \wedge^3 H_{\Z},
\end{tikzcd}\]
where $\iota\colon H_{\Z} \rightarrow \wedge^3 H_{\Z}$ takes $h \in H_{\Z}$ to $h \wedge \omega$.  Identifying
$H_{\Z}$ with its image in $\wedge^3 H_{\Z}$, the above implies that $\tau_1$ induces a homomorphism
$\tau_1\colon \Torelli_g \rightarrow (\wedge^3 H_{\Z})/H_{\Z}$ fitting into a commutative diagram
\[\begin{tikzcd}
1 \arrow{r} & \pi \arrow{d}{\tau_1} \arrow{r} & \Torelli_{g,1} \arrow{d}{\tau_1} \arrow{r} & \Torelli_g \arrow{d}{\tau_1} \arrow{r} & 1 \\
0 \arrow{r} & H_{\Z} \arrow{r}{\iota}         & \wedge^3 H_{\Z} \arrow{r}                  & (\wedge^3 H_{\Z})/H_{\Z} \arrow{r} & 0.
\end{tikzcd}\]
Theorem \ref{theorem:torelliabel} implies that $\tau_1\colon \Torelli_{g,1} \rightarrow \wedge^3 H_{\Z}$ induces an isomorphism\footnote{See
the proof in \S \ref{section:hainpuncture} for why this result about $\Torelli_{g,1}$ follows from Theorem \ref{theorem:torelliabel},
which concerns $\Torelli_g^1$.}
\[(\tau_1)_{\ast}\colon \HH_1(\Torelli_{g,1};\Q) \stackrel{\cong}{\longrightarrow} \HH_1(\wedge^3 H_{\Z};\Q) = \wedge^3 H,\]
so it follows that $\tau_1\colon \Torelli_g \rightarrow (\wedge^3 H_{\Z})/H_{\Z}$ also induces an isomorphism\footnote{This fact was first proved by Johnson in \cite{Johnson2, Johnson3}.}
\[(\tau_1)_{\ast}\colon \HH_1(\Torelli_{g};\Q) \stackrel{\cong}{\longrightarrow} \HH_1((\wedge^3 H_{\Z})/H_{\Z};\Q) = (\wedge^3 H)/H.\]
In light of this, the exact same proof we gave in Lemma \ref{lemma:cupjohnson} for $\Torelli_g^1$ shows
that the image of the cup product pairing $\wedge^2 \HH^1(\Torelli_{g};\Q) \rightarrow \HH^2(\Torelli_{g};\Q)$
is isomorphic to the image of the map
$(\tau_1)_{\ast}\colon \HH_2(\Torelli_{g};\Q) \rightarrow \HH_2((\wedge^3 H_{\Z})/H_{\Z};\Q)$.  Summarizing where
we are, our goal is now the following:
\begin{itemize}
\item[($\dagger$)] Prove that the image of $(\tau_1)_{\ast}\colon \HH_2(\Torelli_{g};\Q) \rightarrow \HH_2((\wedge^3 H_{\Z})/H_{\Z};\Q)$
is isomorphic to $\bV_{1^2}(g) \oplus \bV_{1^4}(g) \oplus \bV_{2^2,1^2}(g) \oplus \bV_{1^6}(g)$.
\end{itemize} 
As notation, let $\Torelli_{g}[2]$ be the kernel of the map
$\tau_1\colon \Torelli_{g} \rightarrow (\wedge^3 H_{\Z})/H_{\Z}$.  Just like in the proof of
Step \ref{step:hain.1} in \S \ref{section:hainboundary}, the five-term exact sequence in
group homology associated to the short exact sequence
\[1 \longrightarrow \Torelli_g[2] \longrightarrow \Torelli_g \stackrel{\tau_1}{\longrightarrow} (\wedge^3 H_{\Z})/H_{\Z} \longrightarrow 1\]
can be analyzed to give an exact sequence
\begin{equation}
\label{eqn:h2torelligseq}
\begin{tikzcd}[column sep=small]
\HH_2(\Torelli_{g};\Q) \arrow{r}{(\tau_1)_{\ast}} & \HH_2((\wedge^3 H_{\Z})/H_{\Z};\Q) \arrow{r} & \HH_1(\Torelli_{g}[2];\Q)_{\Torelli_{g}} \arrow{r} & 0.
\end{tikzcd}
\end{equation}
We have that the representation $\HH_2((\wedge^3 H_{\Z})/H_{\Z};\Q)$ is isomorphic to
\[\wedge^2 (\wedge^3 H)/H \cong \bV_0(g) \oplus \bV_{1^2}(g) \oplus \bV_{2^2}(g) \oplus \bV_{1^4}(g) \oplus \bV_{2^2,1^2}(g) \oplus \bV_{1^6}(g).\]
In light of this isomorphism and the exact sequence \eqref{eqn:h2torelligseq}, to prove $(\dagger)$ it is enough to perform
the following three steps:

\begin{step}{1}
The image of the map $(\tau_1)_{\ast}\colon \HH_2(\Torelli_g;\Q) \rightarrow \HH_2((\wedge^3 H_{\Z})/H_{\Z};\Q)$ contains
the subrepresentation $\bV_{1^2}(g) \oplus \bV_{1^4}(g) \oplus \bV_{2^2,1^2}(g) \oplus \bV_{1^6}(g)$
of $\HH_2((\wedge^3 H_{\Z})/H_{\Z};\Q) \cong \wedge^2 (\wedge^3 H)/H$.
\end{step}

We have a commutative diagram
\[\begin{tikzcd}
\HH_2(\Torelli_g^1;\Q) \arrow{r}{(\tau_1)_{\ast}} \arrow{d} & \HH_2((\wedge^3 H_{\Z});\Q) \arrow[two heads]{d} \arrow[equals]{r} & \wedge^2 \wedge^3 H \arrow[two heads]{d} \\
\HH_2(\Torelli_g;\Q)   \arrow{r}{(\tau_1)_{\ast}}           & \HH_2((\wedge^3 H_{\Z})/H_{\Z};\Q) \arrow[equals]{r} & \wedge^2 (\wedge^3 H)/H.
\end{tikzcd}\]
In \S \ref{section:hainboundary}, we proved that the image of the top row is the subrepresentation
\begin{align*}
&\bV_{1^2}(g)^{\oplus 2} \oplus \bV_{2,1^2}(g) \oplus \bV_{1^4}(g)^{\oplus 2} \oplus \bV_{2^2,1^2}(g) \oplus \bV_{1^6}(g) \\
\subset &\wedge^2 \wedge^3 H = \bV_{1^2}(g)^{\oplus 3} \oplus \bV_{2^2}(g) \oplus \bV_{2,1^2}(g) \oplus \bV_{1^4}(g)^{\oplus 2} \oplus \bV_{2^2,1^2}(g) \oplus \bV_{1^6}(g).
\end{align*}
This maps to the subrepresentation
\begin{align*}
&\bV_{1^2}(g) \oplus \bV_{1^4}(g) \oplus \bV_{2^2,1^2}(g) \oplus \bV_{1^6}(g) \\
\subset &\wedge^2 (\wedge^3 H)/H = \bV_0(g) \oplus \bV_{1^2}(g) \oplus \bV_{2^2}(g) \oplus \bV_{1^4}(g) \oplus \bV_{2^2,1^2}(g) \oplus \bV_{1^6}(g).
\end{align*}
The step follows.

\begin{step}{2}
The homology group $\HH_2(\Torelli_g;\Q)$ contains no trivial subrepresentation, so in particular the image
of $(\tau_1)_{\ast}\colon \HH_2(\Torelli_g;\Q) \rightarrow \HH_2((\wedge^3 H_{\Z})/H_{\Z};\Q)$ does not
contain the subrepresentation $\bV_0(g)$.
\end{step}

Dualizing, it is enough to prove that the $\Sp_{2g}(\Z)$-representation
$\HH^2(\Torelli_g;\Q)$ contains no trivial factors.\footnote{For this, recall that Theorem \ref{maintheorem:h2finite}
says that $\HH^2(\Torelli_g;\Q)$ is a finite-dimensional algebraic representation of $\Sp_{2g}(\Z)$.}
The proof is very similar to that of Lemma \ref{lemma:h2torellitrivial}, which proved that
$\HH^2(\Torelli_g^1;\Q)$ contains no trivial factors.  We recall the basic idea of that proof.  Consider the
Hochschild--Serre spectral sequence of the short exact sequence
\[1 \longrightarrow \Torelli_g \longrightarrow \Mod_g \longrightarrow \Sp_{2g}(\Z) \longrightarrow 1.\]
Harer \cite{HarerH2} proved that $\HH^2(\Mod_g;\Q) = \Q$.  Using the Borel stability theorem (Theorem \ref{theorem:borelstability}),
we can calculate the portion of the $E_2^{pq}$-page relevant to computing $\HH^2(\Mod_g;\Q)$.  The desired
result falls out of this calculation.  We refer the reader to Lemma \ref{lemma:h2torellitrivial} for the details.

\begin{step}{3}
$\HH_1(\Torelli_{g}[2];\Q)_{\Torelli_{g}}$ contains the subrepresentation $\bV_{2^2}(g)$.
\end{step}

Just like in \S \ref{section:hainpuncture}, it is convenient to lift the conjugation action of $\Torelli_g$ on $\Torelli_g[2]$
to an action of $\Torelli_{g,1}$ on $\Torelli_2[2]$ that factors through $\Torelli_g$.  This allows us to
rewrite our coinvariants as
\[\HH_1(\Torelli_{g}[2];\Q)_{\Torelli_{g}} = \HH_1(\Torelli_{g}[2];\Q)_{\Torelli_{g,1}}.\]
From the beginning of the proof, recall the commutative diagram
\[\begin{tikzcd}
1 \arrow{r} & \pi \arrow{d}{\tau_1} \arrow{r} & \Torelli_{g,1} \arrow{d}{\tau_1} \arrow{r} & \Torelli_g \arrow{d}{\tau_1} \arrow{r} & 1 \\
0 \arrow{r} & H_{\Z} \arrow{r}{\iota}         & \wedge^3 H_{\Z} \arrow{r}                  & (\wedge^3 H_{\Z})/H_{\Z} \arrow{r} & 0.
\end{tikzcd}\]
All the vertical maps are surjective.  Using the fact that $[\pi,\pi]$ is the kernel of the map $\pi \rightarrow H_{\Z}$,
we get a short exact sequence
\[1 \longrightarrow [\pi,\pi] \longrightarrow \Torelli_{g,1}[2] \longrightarrow \Torelli_g[2] \longrightarrow 1.\]
See, e.g., \cite[Theorem 4.1]{PutmanJohnson} for more details.  This induces an exact sequence
\[\HH_1([\pi,\pi];\Q) \longrightarrow \HH_1(\Torelli_{g,1}[2];\Q) \longrightarrow \HH_1(\Torelli_g[2];\Q) \longrightarrow 0.\]
Taking coinvariants is right exact, so we have an exact sequence
\[\HH_1([\pi,\pi];\Q)_{\Torelli_{g,1}} \longrightarrow \HH_1(\Torelli_{g,1}[2];\Q)_{\Torelli_{g,1}} \longrightarrow \HH_1(\Torelli_g[2];\Q)_{\Torelli_{g,1}} \longrightarrow 0.\]
Corollary \ref{corollary:kgpunccoinv} says that
\[\HH_1(\Torelli_{g,1}[2];\Q)_{\Torelli_g^1} \cong \bV_0(g) \oplus \bV_{1^2}(g) \oplus \bV_{2^2}(g).\]
To prove that $\HH_1(\Torelli_g[2];\Q)_{\Torelli_{g,1}}$ contains a copy of $\bV_{2^2}(g)$, it is therefore
enough to prove that $\HH_1([\pi,\pi];\Q)_{\Torelli_{g,1}}$ does not contain a copy of $\bV_{2^2}(g)$.
Recall that $\Torelli_{g,1}$ contains as a subgroup the kernel $\pi$ of the Birman exact sequence.  The action
of $\Torelli_{g,1}$ on $\pi$ restricts to the usual action of $\pi$ on itself by conjugation.  From this, we see
that\footnote{The isomorphism $[\pi,\pi]/[\pi,[\pi,\pi]] \otimes \Q \cong \wedge^2 H$ follows from the fact
that $[\pi,\pi]/[\pi,[\pi,\pi]] \otimes \Q$ is the second graded piece of the Lie algebra associated to the
lower central series of $\pi$, which is the free Lie algebra (see \S \ref{section:freelie}).  It is therefore
isomorphic to $\Lie_2(H) \cong \wedge^2 H$; see Example \ref{example:lie2}.}
\[\HH_1([\pi,\pi];\Q)_{\pi} = [\pi,\pi]/[\pi,[\pi,\pi]] \otimes \Q \cong \wedge^2 H.\]
Since $\Torelli_{g,1}$ acts trivially on $H$, this implies that
\[\HH_1([\pi,\pi];\Q)_{\Torelli_{g,1}} \cong \wedge^2 H \cong \bV_0(g) \oplus \bV_{1^2}(g).\]
In particular, $\HH_1([\pi,\pi];\Q)_{\Torelli_{g,1}}$ does not contain a copy of $\bV_{2^2}(g)$, as desired.
\end{proof}

\part{Cup products span}
\label{part:2}

In this second part of the paper, we prove that the second rational cohomology
group of the Torelli group is spanned by cup products of elements of $\HH^1$.
Most of this part is devoting to proving this for $\Torelli_g^1$, and then at the end
we show how to deal with $\Torelli_{g,1}$ and $\Torelli_g$.  Though we handle
the technical details differently, the heart of our proof follows work of
Kupers--Randal-Williams (\cite{KupersRandalWilliams};
see also its sequel \cite{RandalWilliamsII}). 

Recall that Theorem \ref{maintheorem:h2finite} says that $\HH^2(\Torelli_g^1;\Q)$
is a finite-dimensional algebraic representation of $\Sp_{2g}$ for $g \geq 6$.
In fact, the proof of Theorem \ref{maintheorem:h2finite} in \cite{MinahanPutmanH2Torelli}
gives a bit more than this.  We start in \S \ref{section:repstability} by combining
this extra information with tools from the theory of representation stability
to bound the degrees of irreducible factors of
of $\HH^2(\Torelli_g^1;\Q)$ for $g \gg G$.  The constant $G$ here is not effective.

In \S \ref{section:twistedcohomology}, we relate the irreducible factors of
$\HH^2(\Torelli_g^1;\Q)$ to the twisted cohomology of the mapping class
group $\Mod_g^1$.  This allows us to apply recent work of
Miller--Patzt--Petersen--Randal-Williams \cite{MPPRW} to show that we
can take $G$ to be $12$.

In \S \ref{section:cupspanboundary}, we use all of this together with computations
of Kawazumi \cite{KawazumiTwisted} of the twisted cohomology of $\Mod_g^1$ to prove
that cup products span $\HH^2(\Torelli_g^1;\Q)$ for $g \geq 12$.  Finally,
in \S \ref{section:cupspanpuncture} and \S \ref{section:cupspanclosed} we show
how to derive the corresponding results for $\Torelli_{g,1}$ and $\Torelli_g$.

\section{Representation stability and the degrees of the factors of \texorpdfstring{$\HH^2(\Torelli_g^1;\Q)$}{H2Torelli}}
\label{section:repstability}

Fix some $g \geq 6$.  Let $H = \HH_1(\Sigma_g^1;\Q)$.  The first step will be to use tools from
the theory of representation stability to bound the degrees of the irreducible factors
of $\HH^2(\Torelli_g^1;\Q)$.

\subsection{Degree of irreducible representation}

Recall that we defined the degree of an irreducible representation of the algebraic group $\Sp_{2g}$ in
\S \ref{section:sprep} as follows.  An irreducible representation is of the form
$\bV_{\sigma}(g)$ with $\sigma$ a partition with at most $g$ parts.  Writing
$\sigma = (k_1,\ldots,k_m)$, the degree of $\bV_{\sigma}(g)$ is
$d = k_1 + \cdots + k_m$.  This is also the minimal $d$ such that $\bV_{\sigma}(g)$
appears in $H^{\otimes d}$.

\subsection{Algebraicity}

Recall that we are assuming the following theorem, which for convenience we state for homology rather
than cohomology:

\newtheorem*{maintheorem:h2finite.2}{Theorem \ref{maintheorem:h2finite} (Minahan--Putman, \cite[Theorem B]{MinahanPutmanH2Torelli})}
\begin{maintheorem:h2finite.2}
The homology group $\HH_2(\Torelli_g^1;\Q)$ is finite-dimensional
for $g \geq 5$ and an algebraic representation of $\Sp_{2g}(\Z)$ for $g \geq 6$.
\end{maintheorem:h2finite.2}

\begin{remark}
As we noted in the introduction, \cite{MinahanPutmanH2Torelli} also proves similar results for $\Torelli_g$
and $\Torelli_{g,1}$.  In fact, it proves a result for $\Torelli_{g,p}^b$ in general, though it requires
care to properly define the Torelli group on a surface with multiple punctures and boundary components.
\end{remark}

\subsection{Cokernel of stabilization map}

We will need a consequence of the proof of Theorem \ref{maintheorem:h2finite}.  Embed $\Sigma_{g-1}^1$ into
$\Sigma_{g}^1$ as follows:\\
\Figure{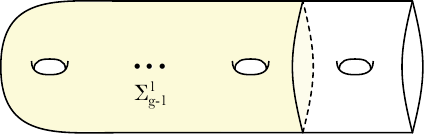}
By extending mapping classes on $\Sigma_{g-1}^1$ to $\Sigma_{g}^1$ by the identity, we get an inclusion
$\Mod_{g-1}^1 \hookrightarrow \Mod_{g}^1$.  This restricts to an inclusion 
$\Torelli_{g-1}^1 \hookrightarrow \Torelli_{g}^1$.  This inclusion induces a map on homology that we will write as
\[f_g\colon \HH_2(\Torelli_{g-1}^1;\Q) \rightarrow \HH_2(\Torelli_{g}^1;\Q).\]
We will call $f_g$ the {\em stabilization map}.
The domain of $f_g$ is a representation of $\Sp_{2(g-1)}(\Z)$ and the codomain of $f_g$ is a representation
of $\Sp_{2g}(\Z)$.  Moreover, $f_g$ is $\Sp_{2(g-1)}(\Z)$-equivariant.  The cokernel $\coker(f_g)$
is therefore a representation of $\Sp_{2(g-1)}(\Z)$.  We then have:

\begin{theorem}
\label{theorem:cokerneldegree}
Let $f_g\colon \HH_2(\Torelli_{g-1}^1;\Q) \rightarrow \HH_2(\Torelli_{g}^1;\Q)$ be the stabilization map.
The following hold for $g \geq 5$:
\begin{itemize}
\item[(i)] The cokernel $\coker(f_g)$ is a finite-dimensional algebraic representation of $\Sp_{2(g-1)}(\Z)$.
\item[(ii)] Each irreducible factor of $\coker(f_g)$ has degree at most $5$.
\end{itemize}
\end{theorem}
\begin{proof}
Conclusion (i) is the main technical result that goes into the proof of Theorem \ref{maintheorem:h2finite}
in \cite{MinahanPutmanH2Torelli}.  The proof actually shows how to write $\coker(f_g)$ in terms
of a series of explicit representations.\footnote{It does not calculate $\coker(f_g)$ completely since for
instance some of these representations appear as domains of maps into $\coker(f_g)$, and these maps
might be $0$.  Of course, as the paper you are reading shows once Theorem \ref{maintheorem:h2finite} is
known we can calculate $\HH_2(\Torelli_g^1;\Q)$ for $g \geq 12$, and thus in particular work out $\coker(f_g)$
for $g \gg 0$.}  If you carefully go through the argument, you will see that all these explicit representations
have degree at most $5$.  The theorem follows.
\end{proof}

\begin{remark}
The paper \cite{MinahanPutmanH2Torelli} introduces a new approach for proving these
kinds of algebraicity results.  If this approach is used in other contexts, it will
likely also give some $d$ such that all the irreducible factors of the relevant
cokernels have degree at most $d$.  The representation stability argument we give
below will then prove that the irreducible factors of the
representations in question all eventually have degree at most $d+1$.
\end{remark}

\subsection{Representation stability}

We next prove the following.  Using the language of \cite{ChurchFarbRepStability},
it says that $\HH_2(\Torelli_g^1;\Q)$ satisfies Church--Farb's notion
of representation stability.  We remark that (iii) is due to Boldsen--Dollerup \cite{BoldsenDollerup}.
We include it here since it forms part of the representation stability package.

\begin{theorem}
\label{theorem:repstability}
There exists some $G \geq 1$ and partitions $\sigma_1,\ldots,\sigma_n$ with at most $G$ parts such 
that the following holds for $g \geq G$:
\begin{itemize}
\item[(i)] We have $\HH_2(\Torelli_g^1;\Q) \cong \bV_{\sigma_1}(g) \oplus \cdots \oplus \bV_{\sigma_n}(g)$.
\item[(ii)] The stabilization map $f_g\colon \HH_2(\Torelli_{g-1}^1;\Q) \rightarrow \HH_2(\Torelli_{g}^1;\Q)$
is injective.
\item[(iii)] The $\Sp_{2g}(\Z)$-orbit of the image of $f_g$ spans $\HH_2(\Torelli_{g}^1;\Q)$.
\end{itemize}
\end{theorem}
\begin{proof}
Set $V_g = \HH_2(\Torelli_g^1;\Q)$.  This is a representation of $\Sp_{2g}(\Z)$, and
$f_g\colon V_{g-1} \rightarrow V_{g}$ is $\Sp_{2(g-1)}(\Z)$-equivariant.  The sequence of
representations
\[V_1 \stackrel{f_2}{\longrightarrow} V_2 \stackrel{f_3}{\longrightarrow} V_3 \stackrel{f_4}{\longrightarrow} \cdots\]
has the following key property:
\begin{itemize}
\item For all $m \geq n \geq 1$, the subgroup $1 \times \Sp_{2(m-n)}(\Z)$ of $\Sp_{2m}(\Z)$ acts trivially
on the image of the composition
\[V_n \stackrel{f_{n+1}}{\longrightarrow} \cdots \stackrel{f_{m}}{\longrightarrow} V_m.\]
\end{itemize}
As was noted in \cite[Proposition 3.6]{PatztRepStability}, this is equivalent to the $V_g$ forming
an $\SI(\Z)$-module in the sense of Putman--Sam \cite{PutmanSam}.  We have the following two
key properties: 
\begin{itemize}
\item For $g \geq 5$, the vector space $V_g$ is finite-dimensional (Theorem \ref{maintheorem:h2finite}).
\item Conclusion (iii) holds for $g \geq 7$.  This was proved by Boldsen--Dollerup \cite{BoldsenDollerup}.
\end{itemize}
In the language of $\SI(\Z)$-modules, these two facts imply that the $V_g$ form a finitely generated
$\SI(\Z)$-module.  Since $V_g$ is also an algebraic representation of $\Sp_{2g}$ for
$g \geq 6$ (Theorem \ref{maintheorem:h2finite}), it is also a rational $\SI(\Z)$-module
in the sense of \cite{PatztRepStability}.  We remark that \cite{PatztRepStability} requires
all the $V_g$ to be algebraic, but for its proofs it is enough for them to be algebraic
for $g$ sufficiently large.  We can now apply \cite[Theorem B]{PatztRepStability}, which
gives our three conclusions.
\end{proof}

\subsection{Branching rule}

The last ingredient we need for the main result of this section is
the following theorem describing which irreducible representations
appear when we branch an irreducible representation of $\Sp_{2g}$ to
$\Sp_{2(g-1)}$:

\begin{theorem}[{\cite[Theorem 8.1.5]{GoodmanWallach}}]
\label{theorem:spbranch}
For some $g \geq 2$, let $\sigma$ be a partition with at most $g$ parts.
Write $\sigma = (k_1,\ldots,k_m)$.
The following are a complete list of the irreducible factors appearing in
the restriction $\Res^{\Sp_{2g}}_{\Sp_{2(g-1)}} \bV_{\sigma}(g)$:
\begin{itemize}
\item Irreducible representations of the form $\bV_{\sigma'}(g-1)$
such that $\sigma'$ is a partition with at most $g-1$ parts satisfying
the following condition.  Write $\sigma' = (k'_1,\ldots,k'_{m'})$. 
We then require that the following holds for all $i$:
\[k_{i+2} \leq k'_i \leq k_i.\]
Here our convention is that $k_j = 0$ for $j \geq m$ and $k'_j = 0$ for $j \geq m'$.
\end{itemize}
\end{theorem}

We need the following immediate corollary of this:

\begin{corollary}
\label{corollary:spbranchcor}
For some $g \geq 2$, let $\bV_{\sigma}(g)$ be an irreducible representation of $\Sp_{2g}$
of degree $d \geq 1$.  Assume that $\sigma$ has at most $g-1$ parts.
Then $\Res^{\Sp_{2g}}_{\Sp_{2(g-1)}} \bV_{\sigma}(g)$ contains
at least two irreducible factors: $\bV_{\sigma}(g-1)$, and a factor $\bV_{\sigma'}(g-1)$ of degree
$d-1$.
\end{corollary}

\subsection{Bounding the degrees}

We now come to the main result of this section:

\begin{proposition}
\label{proposition:bounddegrees}
There exists some $G \geq 1$ such that the following hold for $g \geq G$:
\begin{itemize}
\item The $\Sp_{2g}(\Z)$-representation $\HH^2(\Torelli_g^1;\Q)$ is a finite-dimensional
algebraic representation of $\Sp_{2g}$ all of whose irreducible factors have degree
at most $6$.
\end{itemize}
\end{proposition}
\begin{proof}
Passing to duals preserves degrees, so it is enough to prove this for homology rather
than cohomology.  For $g \geq 1$, set $V_g = \HH_2(\Torelli_g^1;\Q)$, and for $g \geq 2$
let $f_g\colon V_{g-1} \rightarrow V_g$ be the stabilization map.
Combining Theorems \ref{theorem:cokerneldegree} and \ref{theorem:repstability}, we can find
$G \geq 1$ such that for $g \geq G$ the following all hold:  
\begin{itemize}
\item There exist irreducible representations $\bV_{\sigma_1}(g-1),\ldots,\bV_{\sigma_n}(g-1)$ of
$\Sp_{2{g-1}}$ such that $V_{g-1} = \bV_{\sigma_1}(g-1) \oplus \cdots \oplus \bV_{\sigma_n}(g-1)$
and $V_g = \bV_{\sigma_1}(g) \oplus \cdots \oplus \bV_{\sigma_n}(g)$.
\item The stabilization map $f_g\colon V_{g-1} \rightarrow V_g$ is injective.
\item The cokernel of $f_g$ is an algebraic representation of $\Sp_{2(g-1)}$ all of
whose irreducible factors have degree at most $5$.
\end{itemize}
Consider some $g \geq G$, and let the $\sigma_i$ be as in the first bullet point.
We claim that each $\sigma_i$ has degree at most $6$.  To see this, set
\[W = \Res^{\Sp_{2g}}_{\Sp_{2(g-1)}} V_g \cong \bigoplus_{i=1}^n \Res^{\Sp_{2g}}_{\Sp_{2(g-1)}} \bV_{\sigma_i}(g).\]
Corollary \ref{corollary:spbranchcor} say that the $\Sp_{2(g-1)}$-representation
$\bV_{\sigma_i}(g-1)$ appears as one of the irreducible factors of $\Res^{\Sp_{2g}}_{\Sp_2(g-1)} \bV_{\sigma_i}(g)$.
The representation $\coker(f_g)$ is therefore isomorphic to the one obtained from
$W$ by deleting copies of $\bV_{\sigma_1}(g-1),\ldots,\bV_{\sigma_n}(g-1)$.  Each of
the remaining irreducible factors has degree at most $5$.  The proof
concludes by noting that Corollary \ref{corollary:spbranchcor} also implies
that if $\bV_{\sigma_i}(g)$ has degree $d \geq 1$, then $\Res^{\Sp_{2g}}_{\Sp_{2(g-1)}} \bV_{\sigma_i}(g)$
contains an irreducible factor $\bV_{\sigma'}(g-1)$ of degree $d-1$.  Necessarily $d-1 \leq 5$, so
$d \leq 6$.
\end{proof}

\begin{remark}
In the next section, we will prove that Proposition \ref{proposition:bounddegrees} holds
for $G=12$.
\end{remark}

\section{Twisted cohomology of the mapping class group and uniform degree bounds}
\label{section:twistedcohomology}

In this section, we relate $\HH^2(\Torelli_g^1;\Q)$ to the cohomology of the mapping class
group with twisted coefficients.  Using recent breakthrough work of Miller--Patzt--Petersen--Randal-Williams \cite{MPPRW}
on stability phenomena for these twisted cohomology groups, we will be able to prove that
Proposition \ref{proposition:bounddegrees} holds for $G=12$.

\subsection{Borel stability theorem}

We will need the following special case of the classical Borel stability theorem \cite{BorelStability1, BorelStability2} on
the cohomology of arithmetic groups, which we already stated in \S \ref{section:h2notrivial} but whose statement
we recall for convenience  The explicit stable range in the following
is due to Tshishiku \cite{TshishikuBorel}:

\newtheorem*{theorem:borelstability}{Theorem \ref{theorem:borelstability} (Borel, \cite{BorelStability1, BorelStability2})}
\begin{theorem:borelstability}
For $g \geq 2$, the following hold:
\begin{itemize}
\item[(i)] If $\bV_{\sigma}(g)$ is a nontrivial irreducible representation of $\Sp_{2g}$,
then $\HH^k(\Sp_{2g}(\Z);\bV_{\sigma}(g)) = 0$ for $k \leq g-1$.
\item[(ii)] In degrees $k \leq g-1$, the cohomology ring $\HH^{\bullet}(\Sp_{2g}(\Z);\Q)$ is
isomorphic to a polynomial ring $\Q[c_2,c_6,c_{10},\ldots]$ with $\deg(c_{4i-2}) = 4i-2$ for
$i \geq 1$.
\end{itemize}
\end{theorem:borelstability}

\subsection{Twisted cohomology and the Torelli group}

We now relate the irreducible factors of $\HH^2(\Torelli_g^1;\Q)$ to the twisted
cohomology of $\Mod_g^1$:

\begin{lemma}
\label{lemma:h2torellitwisted}
For some $g \geq 12$, let $\bV_{\sigma}(g)$ be a nontrivial irreducible representation of
$\Sp_{2g}$.  Then the dimension of $\HH^2(\Mod_g^1;\bV_{\sigma}(g))$ equals the number
of copies of $\bV_{\sigma}(g)$ in $\HH^2(\Torelli_g^1;\Q)$.
\end{lemma}
\begin{proof}
The restriction $g \geq 12$ implies that $\HH^2(\Torelli_g^1;\Q)$ is a finite-dimensional
algebraic representation of $\Sp_{2g}(\Z)$ (Theorem \ref{maintheorem:h2finite}).  It  
therefore decomposes as a direct sum of irreducible factors, so
the statement of the lemma makes sense.  Consider the Hochschild--Serre
spectral sequence with coefficients in $\bV_{\sigma}(g)$ associated
to the short exact sequence
\[1 \longrightarrow \Torelli_g^1 \longrightarrow \Mod_g^1 \longrightarrow \Sp_{2g}(\Z) \longrightarrow 1.\]
This spectral sequence takes the form
\[E_2^{pq} = \HH^p(\Sp_{2g}(\Z);\HH^q(\Torelli_g^1;\bV_{\sigma}(g))) \Rightarrow \HH^{p+q}(\Mod_g^1;\bV_{\sigma}(g)).\]
We are interested in $\HH^2$, so the relevant terms are those with $p+q \leq 2$.  To understand
potential differentials, we will also need to understand $E_2^{12}$ and $E_2^{03}$.
We divide the proof into two cases.

\begin{case}{1}
$\bV_{\sigma}(g) \notin \{\bV_1(g),\bV_{1^3}(g)\}$.
\end{case}

Since $\Torelli_g^1$ acts trivially on $\bV_{\sigma}(g)$, we have
\begin{equation}
\label{eqn:torellitwisted}
\HH^q(\Torelli_g^1;\bV_{\sigma}(g)) = \Hom(\HH^q(\Torelli_g^1;\Q),\bV_{\sigma}(g)).
\end{equation}
For $q=0$ this is just $\bV_{\sigma}(g)$, so by Theorem \ref{theorem:borelstability}
we have
\[E_2^{p0} = \HH^p(\Sp_{2g}(\Z);\bV_{\sigma}(g)) = 0 \quad \text{for $p \leq g-1$}.\]
For $q=1$, letting $H = \HH_1(\Sigma_g^1;\Q)$ we can apply Theorem \ref{theorem:torelliabel}
to see that \eqref{eqn:torellitwisted} equals
\[\Hom(\HH^1(\Torelli_g^1;\Q),\bV_{\sigma}(g)) = \Hom(\wedge^3 H,\bV_{\sigma}(g)).\]
This decomposes as a direct sum of irreducible representations of $\Sp_{2g}$.  Moreover,
since $\wedge^3 H = \bV_1(g) \oplus \bV_{1^3}(g)$ we can use our assumption
that $\bV_{\sigma}(g) \notin \{\bV_1(g),\bV_{1^3}(g)\}$ to see that it has
no trivial factors.  We can therefore apply Theorem \ref{theorem:borelstability} to see
that
\[E_2^{p1} = \HH^p(\Sp_{2g}(\Z);\Hom(\wedge^3 H,\bV_{\sigma}(g))) = 0 \quad \text{for $p \leq g-1$}.\]
Finally,
\[E_2^{02} = \HH^0(\Sp_{2g}(\Z);\Hom(\HH^2(\Torelli_g^1;\Q),\bV_{\sigma}(g))) \cong \Hom_{\Sp_{2g}}(\HH^2(\Torelli_g^1;\Q),\bV_{\sigma}(g)).\]
It follows that the dimension of $E_2^{02}$ equals the number of copies of 
$\bV_{\sigma}(g)$ in $\HH^2(\Torelli_g^1;\Q)$.
Summarizing, since $g \geq 12$ the $E_2$-page of our spectral sequence takes the form
\begin{center}
\begin{tblr}{|cccc}
$E_2^{02}$ &     &     & \\
$0$        & $0$ & $0$ & \\
$0$        & $0$ & $0$ & $0$\\
\hline
\end{tblr}
\end{center}
It follows that
\[\HH^{2}(\Mod_g^1;\bV_{\sigma}(g)) \cong E_{\infty}^{02} = E_2^{02}.\]
The dimension of this equals the number of copies of $\bV_{\sigma}(g)$ in $\HH^2(\Torelli_g^1;\Q)$, as desired.

\begin{case}{2}
$\bV_{\sigma}(g) \in \{\bV_1(g),\bV_{1^3}(g)\}$.
\end{case}

The only difference between this and the previous case is that now
$\Hom(\wedge^3 H,\bV_{\sigma}(g))$ has a $1$-dimensional trivial representation.  Applying
Theorem \ref{theorem:borelstability}, we therefore get that
\[E_2^{p1} = \HH^p(\Sp_{2g}(\Z);\Hom(\wedge^3 H,\bV_{\sigma}(g))) = 
\begin{cases}
\Q & \text{if $p=0,2$},\\
0  & \text{if $p=1$}.
\end{cases}\]
Our spectral sequence therefore takes the form
\begin{center}
\begin{tblr}{|cccc}
$E_2^{02}$ &     &      & \\
$\Q$       & $0$ & $\Q$ & \\
$0$        & $0$ & $0$  & $0$\\
\hline
\end{tblr}
\end{center}
Since the dimension of $E_2^{02}$ is the number of copies of $\bV_{\sigma}(g)$ in $\HH^2(\Torelli_g^1;\Q)$,
to prove the lemma we need to show that the differential
\[d\colon E_2^{02} \rightarrow E_2^{21} = \Q\]
is $0$.  We will sketch a proof of this below, but first we want to point
out that a much more general vanishing holds.
For this, we refer the reader to Kupers--Randal-Williams's argument
in the proof of \cite[Theorem 4.1]{KupersRandalWilliams}, which uses an elegant
criterion to certify that differentials in these kinds of spectral sequences vanish
(see \cite[Lemma 4.3]{KupersRandalWilliams}).\footnote{The proof of \cite[Theorem 4.1]{KupersRandalWilliams}
shows that the relevant differentials vanish with coefficients in tensor powers
$H^{\otimes d}$, not in $V_{\sigma}(g)$.  Since $V_{\sigma}(g)$ is a factor of
an appropriate $H^{\otimes d}$, this gives the desired result.  In fact, since
we only care about $\bV_{\sigma}(g) \in \{\bV_1(g),\bV_{1^3}(g)\}$ it is enough
to handle $H^{\otimes 3}$.}

Here is a sketch of Kupers--Randal-Williams's argument specialized to this
case.  What we need to show is that the $\Q$ in $E_2^{21}$ survives
to give a $\Q$ in $\HH^3(\Mod_g^1;V_{\sigma}(g))$.  Since $\bV_{\sigma}(g) \in \{\bV_1(g),\bV_{1^3}(g)\}$,
the irreducible representation $V_{\sigma}(g)$ is a factor of $H^{\otimes 3}$.  This
implies that $\HH^3(\Mod_g^1;V_{\sigma}(g))$ is naturally a direct
factor of $\HH^3(\Mod_g^1;H^{\otimes 3})$.  Let
\[M = \bigoplus_{k \geq 0} \HH^k(\Mod_g^1;H^{\otimes 3}).\]
The vector space $M$ is a graded module over the graded ring
$R=\HH^{\bullet}(\Mod_g^1;\Q)$.  Madsen--Weiss \cite{MadsenWeiss} proved that
$R$ is isomorphic to a polynomial ring $\Q[\kappa_1,\kappa_2,\ldots]$ with
$|\kappa_i| = 2i$ in a range of degrees that tends to infinity
as $g \mapsto \infty$.  

As we will discuss more detail
in \S \ref{section:cupspanboundary} below, Kawazumi \cite{KawazumiTwisted}
proved that in a range of degrees $M$ is a free module over $R$.
Our hypothesis $g \geq 12$ implies that every term we care about lies in this
stable range, so we can treat $M$ as a free module over $R$.
In the spectral sequence above, the $\Q$ in $E_2^{01}$ necessarily
survives to give a $\Q$ in
\[\HH^1(\Mod_g^1;V_{\sigma}(g)) \subset \HH^1(\Mod_g^1;H^{\otimes 3}) \subset M.\]
Let $b \in M$ generate this copy of $\Q$.  For degree reasons,
$b$ must be a scalar multiple of one of Kawazumi's free generators.  The $\Q$ in
\[\HH^3(\Mod_g^1;V_{\sigma}(g)) \subset \HH^3(\Mod_g^1;H^{\otimes 3}) \subset M\]
we are trying to show survives is spanned by\footnote{Really, the product
structure in this spectral sequence comes from an action of $\HH^{\bullet}(\Sp_{2g}(\Z);\Q)$,
but we can use $\kappa_1$ since $\kappa_1$ is a scalar multiple of the pullback of the generator
$c_2 \in \HH^2(\Sp_{2g}(\Z);\Q)$ from Theorem \ref{theorem:borelstability}.} $\kappa_1 b$, which is nonzero
since $b$ is a free generator.
\end{proof}

\subsection{Uniform twisted stability}

Just like we did in \S \ref{section:repstability}, embed $\Sigma_{g}^1$ into
$\Sigma_{g+1}^1$ as follows:\\
\Figure{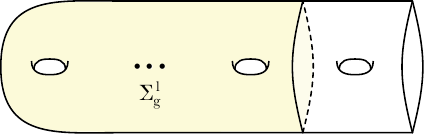}
By extending mapping classes on $\Sigma_{g}^1$ to $\Sigma_{g+1}^1$ by the identity, we get an inclusion
$\Mod_{g}^1 \hookrightarrow \Mod_{g+1}^1$.  This induces a stabilization map
$\HH_k(\Mod_g^1) \rightarrow \HH_k(\Mod_{g+1}^1)$.  A classical theorem of Harer \cite{HarerStability}
says that this is an isomorphism for $g \gg k$.  Later Ivanov \cite{IvanovStability} showed how
to incorporate twisted coefficients into this kind of stability result.  See \cite{RandalWilliamsWahl}
for a modern treatment of this and its many generalizations.

For twisted coefficients, all of these classical results gives stable ranges that depend
strongly on the coefficients.  Recently Miller--Patzt--Petersen--Randal-Williams \cite{MPPRW}
proved a theorem whose range is independent of the coefficients.  To set it up,
for some $g \geq 1$ let $\sigma$ be a partition with at most $g$ parts.
We then have the irreducible representation $\bV_{\sigma}(g)$ of $\Sp_{2g}$.  There
is a natural map $\bV_{\sigma}(g) \rightarrow \bV_{\sigma}(g+1)$ (cf.\ the branching
rule discussed in Theorem \ref{theorem:spbranch}).  This gives a stabilization map
$\HH_k(\Mod_g^1;\bV_{\sigma}(g)) \rightarrow \HH_k(\Mod_{g+1}^1;\bV_{\sigma}(g+1))$.

The paper \cite{MPPRW} proves the following result:\footnote{The statement of
\cite[Theorem 1.1]{MPPRW} gives a range in which the relative twisted homology
group $\HH_d(\Mod_{g+1}^1,\Mod_g^1;\bV_{\sigma}(g+1),\bV_{\sigma}(g))$ vanishes.
To derive stability for $\HH_k$ we need this relative twisted homology group
to vanish for $k \leq d+1$.  This explains the range given in our statement.}

\begin{theorem}[{Miller--Patzt--Petersen--Randal-Williams \cite[Theorem 1.1]{MPPRW}}]
\label{theorem:mpprw}
Let $g \geq 1$ and let $\sigma$ be a partition with $m \leq g$ parts.
Then the stabilization map $\HH_k(\Mod_g^1;\bV_{\sigma}(g)) \rightarrow \HH_k(\Mod_{g+1}^1;\bV_{\sigma}(g+1))$ is an
isomorphism for $g \geq 4k+4$.
\end{theorem}

\subsection{Uniform degree bounds}

We now prove the main result of this section:

\begin{proposition}
\label{proposition:bounddegreesuniform}
For $g \geq 12$, the following holds:
\begin{itemize} 
\item The $\Sp_{2g}(\Z)$-representation $\HH^2(\Torelli_g^1;\Q)$ is a finite-dimensional
algebraic representation of $\Sp_{2g}$ all of whose irreducible factors have degree
at most $6$.
\end{itemize}
\end{proposition}
\begin{proof}
Theorem \ref{maintheorem:h2finite} implies that $\HH^2(\Torelli_g^1;\Q)$ is a finite-dimensional
algebraic representation of $\Sp_{2g}$ in this range.  Moreover,
Proposition \ref{proposition:bounddegrees} says that there is some $G \geq 1$
such that for $g \geq G$ each irreducible factor of $\HH^2(\Torelli_g^1;\Q)$ has
degree at most $6$.  We must verify that this propagates down to $g \geq 12$.

Consider some $12 \leq g \leq G$.  Let $\bV_{\sigma}(g)$ be an irreducible
representation of $\Sp_{2g}$ of degree greater than $6$.  Lemma \ref{lemma:h2torellitwisted}
implies that the dimension of $\HH^2(\Mod_g^1;\bV_{\sigma}(g))$ equals the number
of copies of $\bV_{\sigma}(g)$ in $\HH^2(\Torelli_g^1;\Q)$.  We therefore must
prove that $\HH^2(\Mod_g^1;\bV_{\sigma}(g)) = 0$.  Theorem \ref{theorem:mpprw} implies
that $\HH^2(\Mod_g^1;\bV_{\sigma}(g))$ is isomorphic to $\HH^2(\Mod_G^1;\bV_{\sigma}(G))$,
which vanishes since all irreducible factors of $\HH^2(\Torelli_G^1;\Q)$ have degree
at most $6$.  The proposition follows.
\end{proof}

\section{Cup products span on surfaces with boundary}
\label{section:cupspanboundary}

Recall that Theorem \ref{maintheorem:h2calc} says that for $g \geq 12$ the second rational cohomology
of the Torelli group is spanned by cup products of classes in $\HH^1$.  This section
proves this for $\Torelli_g^1$.

\subsection{Detecting the decomposition}

Let $H = \HH_1(\Sigma_g^1;\Q)$.  Our proof uses work of Kawazumi \cite{KawazumiTwisted}
describing $\HH^{\bullet}(\Mod_g^1;H^{\otimes d})$ in a stable range.  We
start by recasting our results from the previous section in these terms:

\begin{lemma}
\label{lemma:h2torellihd}
Fix $g \geq 12$ and $d \geq 1$.  Set $H = \HH_1(\Sigma_g^1;\Q)$.  Let $t \geq 0$
be the dimension of the trivial subrepresentation of $H^{\otimes d}$.  We then have
\[\dim_{\Q} \HH^2(\Mod_g^1;H^{\otimes d}) = t + \dim_{\Q} \Hom_{\Sp_{2g}}(\HH^2(\Torelli_g^1;\Q),H^{\otimes d}).\]
\end{lemma}
\begin{proof}
Write $H^{\otimes d} = \bV_{0}(g)^{\oplus t} \oplus V$ with $V$ a direct sum of nontrivial
irreducible representations of $\Sp_{2g}$.  We have
\[\HH^2(\Mod_g^1;H^{\otimes d}) = \HH^2(\Mod_g^1;\bV_0(g))^{\oplus t} \oplus \HH^2(\Mod_g^1;V).\]
Theorem \ref{theorem:harerh2} says that
\[\dim_{\Q} \HH^2(\Mod_g^1;\bV_0(g))^{\oplus t} = t,\]
and Lemma \ref{lemma:h2torellitwisted} together with Schur's Lemma says that
\[\dim_{\Q} \HH^2(\Mod_g^1;V) = \dim_{\Q} \Hom_{\Sp_{2g}}(\HH^2(\Torelli_g^1;\Q),V).\]
Since $\HH^2(\Torelli_g^1;\Q)$ has no trivial subrepresentations (Lemma \ref{lemma:h2torellitrivial}), this
is the same as
\[\dim_{\Q} \Hom_{\Sp_{2g}}(\HH^2(\Torelli_g^1;\Q),\bV_0(g)^{\oplus t} \oplus V) = \dim_{\Q} \Hom_{\Sp_{2g}}(\HH^2(\Torelli_g^1;\Q),H^{\otimes d}).\]
The lemma follows.
\end{proof}

\subsection{Weighted partitions}

We now turn to Kawazumi's description of $\HH^{\bullet}(\Mod_g^1;H^{\otimes d})$.  This
requires a preliminary definition.  A {\em weighted partition} of $\{1,\ldots,d\}$
consists of the following data:
\begin{itemize}
\item For some $n \geq 1$, a decomposition $\{1,\ldots,d\} = S_1 \sqcup \cdots \sqcup S_n$
where the $S_i$ are disjoint nonempty sets.
\item For each $1 \leq i \leq n$, a weight $w_i \geq 0$.  If $|S_i| = 1$, then
we require $w_i \geq 1$.
\end{itemize}
The ordering on the $S_i$ is not important, so we identify two weighted partitions
if they differ by a permutation of the $S_i$.  Let $\cP_d$ be the set of all weighted
partitions of $\{1,\ldots,d\}$.

\subsection{Twisted Morita--Mumford classes}

Consider some $P \in \cP_d$.  Define $k(P) \geq 0$ in the following way.
Assume that $P$ consists of the decomposition $\{1,\ldots,d\} = S_1 \sqcup \cdots \sqcup S_n$
and the weights $w_1,\ldots,w_n \geq 0$.  We then set\footnote{To see that $k(P) \geq 0$, note
that it can also be written as $k(P) = 2\sum_{i=1}^n (|S_i|/2+w_i-1)$.  Our condition on
the weights ensures that each term of this sum is nonnegative.}
\[k(P) = d + 2\sum_{i=1}^n (w_i-1) \geq 0.\]
This only depends on the weights, not the decomposition.  Kawazumi \cite[p.\ 388]{KawazumiTwisted}
defines an associated twisted Morita--Mumford class $m_P \in \HH^{k(P)}(\Mod_g^1;H^{\otimes d})$.
It is a cup product of classes $m_{S_i,w_i} \in \HH^{|S_i|+2(w_i-1)}(\Mod_g^1;H^{\otimes |S_i|})$,
where the cup product uses the product map
\[H^{\otimes |S_1|} \otimes \cdots \otimes H^{\otimes |S_n|} \stackrel{\cong}{\longrightarrow} H^{\otimes d}\]
coming from the decomposition $S_1 \sqcup \cdots \sqcup S_n = \{1,\ldots,d\}$.

\subsection{Kawazumi's work}

For a fixed $d \geq 0$, the collection of cohomology groups $M_d = \HH^{\bullet}(\Mod_g^1;H^{\otimes d})$
is a graded module over the cohomology ring $R = \HH^{\bullet}(\Mod_g^1;\Q)$.  Madsen--Weiss \cite{MadsenWeiss} proved that
$R$ is isomorphic to a polynomial ring $\Q[\kappa_1,\kappa_2,\ldots]$ with
$|\kappa_i| = 2i$ in a range of degrees that tends to infinity as $g \mapsto \infty$.  Kawazumi \cite{KawazumiTwisted}
proved that $M_d$ is a free $R$-module in a range of degrees.  

A version of his
theorem is as follows.  The stable range in it is not the one in Kawazumi's paper, but
it follows from Theorem \ref{theorem:mpprw} together with the fact that the decomposition
of $H^{\otimes d}$ into irreducible factors is stable once $g \geq d$ (see \S \ref{section:stablesp}).

\begin{theorem}[{Kawazumi, \cite[Theorem 1.B]{KawazumiTwisted}}]
\label{theorem:kawazumi}
Let $d,g \geq 0$ be such that $g \geq d$.  Then in degrees $k$ such that $g \geq 4k+4$,
we have that $\HH^{\bullet}(\Mod_g^1;H^{\otimes d})$ is a free graded
module over $\HH^{\bullet}(\Mod_g^1;\Q)$ with free basis 
the twisted Morita--Mumford classes $\Set{$m_P$}{$P \in \cP_d$}$.
\end{theorem}

\begin{remark}
Before Kawazumi's work, Looijenga \cite{LooijengaTwisted} used Hodge theory to prove
a similar theorem for the mapping class group $\Mod_g$ of a closed surface.  The
proofs in \cite{LooijengaTwisted} and \cite{KawazumiTwisted} are quite different.
\end{remark}

\subsection{Proof of main theorem}

We close this section by proving Theorem \ref{maintheorem:h2calc} for surfaces with boundary:

\newtheorem*{maintheorem:h2calcboundary}{Theorem \ref{maintheorem:h2calc} (surface with boundary case)}
\begin{maintheorem:h2calcboundary}
Let $g \geq 12$.  The image of
the cup product pairing $\fc\colon \wedge^2 \HH^1(\Torelli_g^1;\Q) \rightarrow \HH^2(\Torelli_g^1;\Q)$
spans $\HH^2(\Torelli_g^1;\Q)$.
\end{maintheorem:h2calcboundary}
\begin{proof}
Proposition \ref{proposition:bounddegreesuniform} says that $\HH^2(\Torelli_g^1;\Q)$ is a
finite-dimensional algebraic representation of $\Sp_{2g}$ all of whose irreducible factors
have degree at most $6$.  Letting $H = \HH_1(\Sigma_g^1;\Q)$, this implies that all
irreducible factors of $\HH^2(\Torelli_g^1;\Q)$ appear in $H^{\otimes d}$ for
some $1 \leq d \leq 6$.  We proved in Theorem \ref{maintheorem:hain} that the image
of the cup product pairing is isomorphic to
\[\cV_{\cp} = \bV_{1^2}(g)^{\oplus 2} \oplus \bV_{2,1^2}(g) \oplus \bV_{1^4}(g)^{\oplus 2} \oplus \bV_{2^2,1^2}(g) \oplus \bV_{1^6}(g).\]
This appears as a direct summand of $\HH^2(\Torelli_g^1;\Q)$.  To prove that
this is all of $\HH^2(\Torelli_g^1;\Q)$, it is enough to check that it accounts
for all of $\Hom_{\Sp_{2g}}(\HH^2(\Torelli_g^1;\Q),H^{\otimes d})$ for $1 \leq d \leq 6$.
In other words, we must prove that
\[\dim_{\Q} \Hom_{\Sp_{2g}}(\HH^2(\Torelli_g^1;\Q),H^{\otimes d}) = \dim_{\Q} \Hom_{\Sp_{2g}}(\cV_{\cp},H^{\otimes d})\]
for $1 \leq d \leq 6$.  Let $t_d$ be the dimension of the trivial subrepresentation of $H^{\otimes d}$.
By Lemma \ref{lemma:h2torellihd}, we have
\[\dim_{\Q} \HH^2(\Mod_g^1;H^{\otimes d}) = t_d + \dim_{\Q} \Hom_{\Sp_{2g}}(\HH^2(\Torelli_g^1;\Q),H^{\otimes d}).\]
Combining all of this, we see that we must prove the following:

\begin{unnumberedclaim}
For $1 \leq d \leq 6$, we have
$\dim_{\Q} \HH^2(\Mod_g^1;H^{\otimes d}) = t_d + \dim_{\Q} \Hom_{\Sp_{2g}}(\cV_{\cp},H^{\otimes d})$.
\end{unnumberedclaim}

Since $g \geq 12$, the decomposition of $H^{\otimes d}$ is stable for $1 \leq d \leq 6$ (see
\S \ref{section:stablesp}).  We can therefore compute the right hand side of
the desired identity using LiE \cite{LieProgram}:
\begin{center}
\begin{tblr}{c|c|c|c}
$d$ & $t_d$ & $\dim_{\Q} \Hom_{\Sp_{2g}}(\cV_{\cp},H^{\otimes d})$ & $t_d + \dim_{\Q} \Hom_{\Sp_{2g}}(\cV_{\cp},H^{\otimes d})$ \\
\hline
1 & 0  & 0   & 0  \\
2 & 1  & 2   & 3  \\
3 & 0  & 0   & 0  \\
4 & 3  & 17  & 20 \\
5 & 0  & 0   & 0  \\
6 & 15 & 175 & 190
\end{tblr}
\end{center}
We now focus on the left hand side $\dim_{\Q} \HH^2(\Mod_g^1;H^{\otimes d})$.
Set $M_d = \HH^{\bullet}(\Mod_g^1;H^{\otimes d})$ and $R = \HH^{\bullet}(\Mod_g^1;\Q)$.
By the Madsen--Weiss theorem \cite{MadsenWeiss}, in a range of degrees that 
includes $2$ we have that $R \cong \Q[\kappa_1,\kappa_2,\ldots]$ with $\kappa_i \in \HH^{2i}(\Mod_g^1;\Q)$.
Theorem \ref{theorem:kawazumi} says that in the range of degrees we care about, $M_d$
is a free $R$-module on the classes $\Set{$m_P$}{$P \in \cP_d$}$.  For $P \in \cP_d$, recall
that $m_P \in \HH^{k(m_P)}(\Mod_g^1;H^{\otimes d})$.  It follows that
the $\Q$-vector space $\HH^2(\Mod_g^1;H^{\otimes d})$ has a basis consisting
of the following elements:
\begin{itemize}
\item $m_P$ for $p \in \cP_d$ with $k(m_P) = 2$; and
\item $\kappa_1 m_P$ for $p \in \cP_d$ with $k(m_P) = 0$.
\end{itemize}
Using a computer, it is easy to enumerate the weighted partitions $P \in \cP_d$
with $k(m_P) \in \{0,2\}$ for $d \in \{1,\ldots,6\}$.  See \cite{PutmanCode} for the code.\footnote{This code
was written and tested entirely by hand.  No AI tools were used.}  The
results are as follows:
\begin{center}
\begin{tblr}{c|c|c|c}
$d$ & $|\Set{$P \in \cP_d$}{$k(m_P) = 0$}|$ & $|\Set{$P \in \cP_d$}{$k(m_P) = 2$}|$ & $\dim_{\Q} \HH^2(\Mod_g^1;H^{\otimes d})$ \\
\hline
1   & 0   & 0   & 0 \\
2   & 1   & 2   & 3 \\
3   & 0   & 0   & 0 \\
4   & 3   & 17  & 20 \\
5   & 0   & 0   & 0 \\
6   & 15  & 175 & 190 
\end{tblr}
\end{center}
Comparing the above tables, the right hand columns are the same.\footnote{In fact, the tables are the
same, which is no accident: $\Set{$P \in \cP_d$}{$k(m_P) = 0$}$ is in bijection with a basis for the
trivial subrepresentation of $H^{\otimes d}$.}  The claim follows.
\end{proof}

\section{Cup products span on punctured surfaces}
\label{section:cupspanpuncture}

We next prove Theorem \ref{maintheorem:h2calc} for punctured surfaces:

\newtheorem*{maintheorem:h2calcpunctured}{Theorem \ref{maintheorem:h2calc} (punctured surface case)}
\begin{maintheorem:h2calcpunctured}
Let $g \geq 12$.  The image of
the cup product pairing $\fc\colon \wedge^2 \HH^1(\Torelli_{g,1};\Q) \rightarrow \HH^2(\Torelli_{g,1};\Q)$
spans $\HH^2(\Torelli_{g,1};\Q)$.
\end{maintheorem:h2calcpunctured}
\begin{proof}
Since we already proved in \S \ref{section:cupspanboundary} that $\HH^2(\Torelli_g^1;\Q)$ is spanned by the image
of the cup product pairing, we know from Theorem \ref{maintheorem:hain} that
\[\HH^2(\Torelli_g^1;\Q) \cong \bV_{1^2}(g)^{\oplus 2} \oplus \bV_{2,1^2}(g) \oplus \bV_{1^4}(g)^{\oplus 2} \oplus \bV_{2^2,1^2}(g) \oplus \bV_{1^6}(g).\]
Theorem \ref{maintheorem:hain} also says that the image of the cup product pairing for
$\Torelli_{g,1}$ is isomorphic to
\[\bV_0(g) \oplus \bV_{1^2}(g)^{\oplus 2} \oplus \bV_{2,1^2}(g) \oplus \bV_{1^4}(g)^{\oplus 2}\oplus \bV_{2^2,1^2}(g) \oplus \bV_{1^6}(g).\]
The only difference between this and $\HH^2(\Torelli_g^1;\Q)$ is a single copy of the trivial
representation $\bV_0(g) \cong \Q$.  To prove that the image of the cup product pairing
for $\Torelli_{g,1}$ is everything, it is thus enough to prove the following claim (in which for
convenience we dualize and switch to homology):

\begin{unnumberedclaim}
We have $\dim_{\Q} \HH_2(\Torelli_{g,1};\Q) = 1 + \dim_{\Q} \HH_2(\Torelli_g^1;\Q)$.
\end{unnumberedclaim}

Let $b = \partial \Sigma_g^1$.  As we noted in the proof of Theorem \ref{maintheorem:hain}
for punctured surfaces in \S \ref{section:hainpuncture}, there is a central extension
\[1 \longrightarrow \Z \longrightarrow \Torelli_g^1 \longrightarrow \Torelli_{g,1} \longrightarrow 1\]
whose central $\Z$ term is generated by the Dehn twist $T_b$ (see \cite[Proposition 3.19]{FarbMargalitPrimer}
for the corresponding exact sequence for the mapping class group).
The associated Hochschild--Serre spectral sequence takes the form
\[E^2_{pq} = \HH_p(\Torelli_{g,1};\HH_q(\Z;\Q)) \Rightarrow \HH_{p+q}(\Torelli_g^1;\Q).\]
Since our extension is central, the action of $\Torelli_{g,1}$ here is trivial and thus
\[E^2_{pq} = \begin{cases}
\HH_p(\Torelli_{g,1};\Q) & \text{for $q=0,1$},\\
0                        & \text{otherwise}.
\end{cases}\]
Let $H = \HH_1(\Sigma_{g,1};\Q)$.  As we noted in the proof of Theorem \ref{maintheorem:hain}
for punctured surfaces in \S \ref{section:hainpuncture}, we have
$\HH_1(\Torelli_{g,1};\Q) \cong \wedge^3 H$.
This isomorphism is induced by the first Johnson homomorphism.  Because of all of this,
the portion of our spectral sequence that can contribute to $\HH_2(\Torelli_{g}^1;\Q)$
looks like this:
\begin{center}
\begin{tblr}{|ccc}
$\Q$ & $\wedge^3 H$ & \\
$\Q$ & $\wedge^3 H$ & $\HH_2(\Torelli_{g,1};\Q)$\\ 
\hline
\end{tblr}
\end{center}
Theorem \ref{theorem:torelliabel} says that $\HH_1(\Torelli_g^1;\Q) \cong \wedge^3 H$,
so $E^2_{01} = \Q$ must be killed by the differential
\[\HH_2(\Torelli_{g,1};\Q) = E^2_{20} \rightarrow E^2_{01} = \Q.\]
In the claim below we will prove that $E^2_{11} = \wedge^3 H$ is also killed by a differential.
From this, we can conclude that
\[\HH_2(\Torelli_g^1;\Q) \cong \ker(\HH_2(\Torelli_{g,1};\Q) \rightarrow \Q),\]
so $\dim_{\Q} \HH_2(\Torelli_{g,1};\Q) = 1 + \dim_{\Q} \HH_2(\Torelli_g^1;\Q)$, as desired.

It remains to prove the following:

\begin{unnumberedclaim}
In the above spectral sequence, $E^2_{11} = \wedge^3 H$ is killed by a differential.
\end{unnumberedclaim}

Recall that $\HH_1(\Torelli_g^1;\Q) = \wedge^3 H$.  For $f \in \Torelli_g^1$, let
\[[f] \in (\Torelli_g^1)^{\text{ab}} \otimes \Q = \HH_1(\Torelli_g^1;\Q)\]
be the corresponding element.  Since $g \geq 3$, it follows from a theorem of
Johnson \cite{JohnsonFirst} that $\Torelli_g^1$ is generated by bounding
pair maps $T_x T_y^{-1}$ such that $x \cup y$ bounds a subsurface homeomorphic
to $\Sigma_1^2$:\\
\Figure{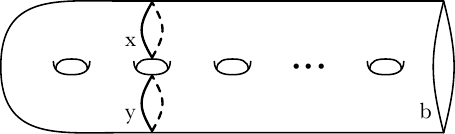}
These are called genus $1$ bounding pair maps.  Letting $T_x T_y^{-1}$
be a genus $1$ bounding pair map, it is enough to prove that
$[T_x T_y^{-1}] \in \wedge^3 H = E^2_{11}$ is killed by a differential.

Recall that $T_b$ is the Dehn twist about $b = \partial \Sigma_g^1$.  This generates
the kernel of the central extension
\[1 \longrightarrow \Z \longrightarrow \Torelli_g^1 \longrightarrow \Torelli_{g,1} \longrightarrow 1\]
whose Hochschild--Serre spectral sequence we are studying.  The subgroup
of $\Torelli_g^1$ generated by $T_b$ and $T_x T_y^{-1}$ is isomorphic to $\Z^2$.
We have a commutative diagram of central extensions
\[\begin{tikzcd}[row sep=small]
1 \arrow{r} & \Z \arrow{r} \arrow[equals]{d}         & \Z^2 \arrow{r} \arrow[equals]{d}                          & \Z \arrow{r} \arrow[equals]{d}             & 1\\  
1 \arrow{r} & \Span{T_b} \arrow{r} \arrow[equals]{d} & \Span{T_b} \times \Span{T_x T_y^{-1}} \arrow{r} \arrow{d} & \Span{T_x T_y^{-1}} \arrow{d} \arrow{r} & 1\\
1 \arrow{r} & \Z         \arrow{r}                   & \Torelli_g^1 \arrow{r}                                    & \Torelli_{g,1} \arrow{r} & 1
\end{tikzcd}\]
whose bottom right arrow is induced by the map taking $T_x T_y^{-1}$ to its image in $\Torelli_{g,1}$.
Consider the map of Hochschild--Serre spectral sequences from the spectral sequence of the middle extension to the spectral sequence of the bottom extension.  The map
on $E^2_{11}$-terms is of the form
\[\HH_1(\Span{T_x T_y^{-1}};\HH_1(\Span{T_b};\Q)) \rightarrow \HH_1(\Torelli_{g,1};\HH_1(\Z;\Q)) = \HH_1(\Torelli_{g,1};\Q) = \wedge^3 H.\]
Our goal is to prove that the image of this is killed by a differential.  We have
\[\HH_1(\Span{T_x T_y^{-1}};\HH_1(\Span{T_b};\Q)) = \HH_1(\Span{T_x T_y^{-1}};\Q) \otimes \HH_1(\Span{T_b};\Q) = \HH_2(\Span{T_b} \times \Span{T_x T_y^{-1}};\Q).\]
From this, we see that our goal is equivalent to showing that the image of the map
\begin{equation}
\label{eqn:killboundary}
\HH_2(\Span{T_b} \times \Span{T_x T_y^{-1}};\Q) \rightarrow \HH_2(\Torelli_g^1;\Q)
\end{equation}
induced by the inclusion $\Span{T_b} \times \Span{T_x T_y^{-1}} \hookrightarrow \Torelli_g^1$ is zero.
Let $T \cong \Sigma_{g-2}^2$ be the following subsurface:\\
\Figure{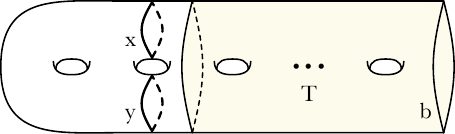}
Denote by $\Torelli_g^1(T)$ the subgroup of $\Torelli_g^1$ consisting of mapping classes supported on $T$.
We can factor our inclusion as
\[\Span{T_b} \times \Span{T_x T_y^{-1}} \hookrightarrow \Torelli_g^1(T) \times \Span{T_x T_y^{-1}} \rightarrow \Torelli_g^1.\]
Using this, our map \eqref{eqn:killboundary} factors as
\begin{align*}
&\HH_2(\Span{T_b} \times \Span{T_x T_y^{-1}};\Q) = \HH_1(\Span{T_b};\Q) \otimes \HH_1(\Span{T_x T_y^{-1}};\Q) \\
&\quad \rightarrow \HH_1(\Torelli_g^1(T);\Q) \otimes \HH_1(\Span{T_x T_y^{-1}};\Q) \hookrightarrow \HH_2(\Torelli_g^1(T) \times \Span{T_x T_y^{-1}};\Q) 
\rightarrow \HH_2(\Torelli_g^1;\Q).
\end{align*}
It is therefore enough to prove that the map $\HH_1(\Span{T_b};\Q) \rightarrow \HH_1(\Torelli_g^1(T);\Q)$ is the zero map,
i.e., that the separating twist $T_b$ vanishes in 
\[\HH_1(\Torelli_g^1(T);\Q) = \Torelli_g^1(T)^{\text{ab}} \otimes \Q.\]
In fact, $T_b^2$ vanishes in the abelianization of $\Torelli_g^1(T)$.
This calculation is essentially due to Johnson \cite{Johnson2}, though
in this level of generality it was first written out in \cite{PutmanJohnson}.
\end{proof}

\section{Cup products span on closed surfaces}
\label{section:cupspanclosed}

We close this paper by proving Theorem \ref{maintheorem:h2calc} for closed surfaces:

\newtheorem*{maintheorem:h2calcclosed}{Theorem \ref{maintheorem:h2calc} (closed surface case)}
\begin{maintheorem:h2calcclosed}
Let $g \geq 12$.  The image of
the cup product pairing $\fc\colon \wedge^2 \HH^1(\Torelli_{g};\Q) \rightarrow \HH^2(\Torelli_{g};\Q)$
spans $\HH^2(\Torelli_{g};\Q)$.
\end{maintheorem:h2calcclosed}
\begin{proof}
Set $H = \HH_1(\Sigma_g;\Q)$.
Since we already proved in \S \ref{section:cupspanpuncture} that $\HH^2(\Torelli_{g,1};\Q)$ 
is spanned by the image of the cup product pairing, we know from Theorem \ref{maintheorem:hain} that
\[\HH^2(\Torelli_{g,1};\Q) \cong \bV_0(g) \oplus \bV_{1^2}(g)^{\oplus 2} \oplus \bV_{2,1^2}(g) \oplus \bV_{1^4}(g)^{\oplus 2}\oplus \bV_{2^2,1^2}(g) \oplus \bV_{1^6}(g).\]
Theorem \ref{maintheorem:hain} also says that the image of the cup product pairing for
$\Torelli_{g}$ is isomorphic to
\[\bV_{1^2}(g) \oplus \bV_{1^4}(g) \oplus \bV_{2^2,1^2}(g) \oplus \bV_{1^6}(g)\]
The only difference between this and $\HH^2(\Torelli_{g,1};\Q)$ is that $\HH^2(\Torelli_{g,1};\Q)$
contains the following additional representations:
\[\bV_0(g) \oplus \bV_{1^2}(g) \oplus \bV_{2,1^2}(g) \oplus \bV_{1^4}(g) \cong \bV_0(g) \oplus \left(H \otimes (\wedge^3 H)/H\right).\]
Since $\bV_0(g)$ is $1$-dimensional, it is thus enough to prove that
\[\dim_{\Q} \HH^2(\Torelli_{g,1};\Q) = 1 + \dim_{\Q} \left(H \otimes (\wedge^3 H)/H\right) + \dim_{\Q} \HH^2(\Torelli_g^1;\Q).\]
In fact, since $\dim_{\Q} \HH^2(\Torelli_g^1;\Q)$ is at least as large as the image of the cup product
pairing we already know that the right hand side is greater than or equal to the left hand side,
so we only need to prove the reverse inequality.  For this, we will dualize and switch to homology.

\begin{unnumberedclaim}
We have $\dim_{\Q} \HH_2(\Torelli_{g,1};\Q) \leq 1 + \dim_{\Q} \left(H \otimes (\wedge^3 H)/H\right) + \dim_{\Q} \HH_2(\Torelli_g^1;\Q)$.
\end{unnumberedclaim}

Let $\pi = \pi_1(\Sigma_g)$.  As we already noted in the proof of Theorem \ref{maintheorem:hain} for
closed surfaces in \S \ref{section:hainclosed}, the Birman exact sequence for
the mapping class group (see \cite[\S 4.2]{FarbMargalitPrimer}) restricts to a short
exact sequence
\[1 \longrightarrow \pi \longrightarrow \Torelli_{g,1} \longrightarrow \Torelli_g \longrightarrow 1.\]
The associated Hochschild--Serre spectral sequence takes the form
\[E^2_{pq} = \HH_p(\Torelli_g;\HH_q(\pi;\Q)) \Rightarrow \HH_{p+q}(\Torelli_{g,1};\Q).\]
Let $H = \HH_1(\pi;\Q) = \HH_1(\Sigma_g;\Q)$.  We also have $\HH_2(\pi;\Q) = \HH_2(\Sigma_g;\Q) = \Q$.
The group $\Torelli_g$ acts trivially on both of these, so
\begin{align*}
E^2_{01} &= \HH_0(\Torelli_g;\HH_1(\pi;\Q)) = H, \\
E^2_{02} &= \HH_0(\Torelli_g;\HH_2(\pi;\Q)) = \Q.
\end{align*}
We also noted during the proof of Theorem \ref{maintheorem:hain} for
closed surfaces in \S \ref{section:hainclosed} that the first Johnson
homomorphism descends to the following isomorphism, which was first proved by
Johnson in \cite{Johnson2, Johnson3}:
\[\HH_1(\Torelli_g;\Q) \cong (\wedge^3 H)/H.\]
We deduce that
\begin{align*}
E^2_{10} &= \HH_1(\Torelli_g;\Q) \cong (\wedge^3 H)/H,\\
E^2_{11} &= \HH_1(\Torelli_g;\HH_1(\pi;\Q)) = \HH_1(\pi;\Q) \otimes \HH_1(\Torelli_g;\Q) \cong H \otimes ((\wedge^3 H)/H).
\end{align*}
Summarizing all of this, the portion of this spectral sequence that can contribute to $\HH_2(\Torelli_{g,1};\Q)$
is
\begin{center}
\begin{tblr}{|ccc}
$\Q$ &                              & \\
$H$  & $H \otimes ((\wedge^3 H)/H)$ & \\
$\Q$ & $(\wedge^3 H)/H$             & $\HH_2(\Torelli_g;\Q)$\\
\hline
\end{tblr}
\end{center}
As we noted in \S \ref{section:hainpuncture}, Theorem \ref{theorem:torelliabel} implies that
\[\HH_1(\Torelli_{g,1};\Q) \cong \wedge^3 H \cong H \oplus ((\wedge^3 H)/H).\]
This implies that the differential
\[H = E^2_{01} \rightarrow E^2_{02} = \HH_2(\Torelli_g;\Q)\]
must vanish.  We conclude that all of $E^2_{02} = \HH_2(\Torelli_g;\Q)$ survives
to the $E^{\infty}$ page.  Potentially\footnote{Once this proof is complete, since
the inequality we are proving is an equality we will know that these differentials vanish too.} 
there might be nonzero differentials involving $E^2_{02}$ and $E^2_{11}$, but in any
case we conclude that
\[\dim_{\Q} \HH_2(\Torelli_{g,1};\Q) \leq 1 + \dim_{\Q} \left(H \otimes (\wedge^3 H)/H\right) + \dim_{\Q} \HH_2(\Torelli_g^1;\Q),\]
as desired.
\end{proof}

\end{document}